\documentclass[12pt]{amsart}

\title{The structure and stability of persistence modules}
\author{Fr\'{e}d\'{e}ric Chazal}
\author{Vin de Silva}
\author{Marc Glisse}
\author{Steve Oudot}
\date{2013--mar--18}

\usepackage{amsmath,amsthm}
\usepackage{amssymb,latexsym}
\usepackage{eucal}

\usepackage{mathabx}
\usepackage{wasysym}

\usepackage{graphicx}
\usepackage[usenames,dvipsnames]{color}
\usepackage{pdfcolmk}
\usepackage[cmtip,arrow]{xy}
\usepackage{pb-diagram,pb-xy}

\usepackage{fullpage}
\usepackage{hyperref}

\numberwithin{equation}{section}

\theoremstyle{plain}

\newtheorem{theorem}[equation]{Theorem}
\newtheorem{lemma}[equation]{Lemma}
\newtheorem{proposition}[equation]{Proposition}
\newtheorem{corollary}[equation]{Corollary}

\newtheorem*{lemma*}{Lemma}

\theoremstyle{definition}

\newtheorem{definition}[equation]{Definition}
\newtheorem*{definition*}{Definition}
\newtheorem{example}[equation]{Example}
\newtheorem*{example*}{Example}

\newtheorem*{notation*}{Notation}

\theoremstyle{remark}
\newtheorem*{remark}{Remark}
 { \begin{list}%
         {$\bullet$}%
         {\setlength{\labelwidth}{20pt}%
          \setlength{\leftmargin}{25pt}%
          \setlength{\topsep}{0pt}
          \setlength{\itemsep}{0pt}
          \setlength{\parsep}{0pt}}}%
 { \end{list} }

 \newenvironment{vlist}%
 { \begin{list}%
         {$\bullet$}%
         {\setlength{\labelwidth}{20pt}%
          \setlength{\leftmargin}{25pt}%
          \setlength{\topsep}{0pt}
          \setlength{\itemsep}{1.5ex}
          \setlength{\parsep}{0pt}}}%
 { \end{list} }

\newcommand{\half}{\frac{1}{2}}

\newcommand{\ringfont}{\mathbf}
  \newcommand{\Rr}{{\ringfont{R}}}
  \newcommand{\Zz}{{\ringfont{Z}}}
  \newcommand{\Nn}{{\ringfont{N}}}
  \newcommand{\Ss}{{\ringfont{S}}}
  \newcommand{\Tt}{{\ringfont{T}}}

  	\newcommand{\rR}{\reflectbox{$\Rr$}}
	\newcommand{\ov}[1]{\underline{#1}}
	\newcommand{\ord}{\text{ord}}
	\newcommand{\rel}{\text{rel}}
	\newcommand{\ext}{\text{ext}}

  \newcommand{\RR}{\overline{\Rr}{}}

  \newcommand{\kk}{\mathbf{k}}

\newcommand{\regionfont}{\mathcal}
   \newcommand{\Dd}{\regionfont{D}}
   
   \newcommand{\rintsup}{\raisebox{0.75ex}{$\sqbullet$}}
   \newcommand{\intsup}{\raisebox{0.75ex}{$\circ$}}

   \newcommand{\drint}{{\Dd\rintsup}}
   \newcommand{\dint}{{\Dd\intsup}}
   
   \newcommand{\frint}{{\regionfont{F}\rintsup}}
   \newcommand{\fint}{{\regionfont{F}\intsup}}
   \newcommand{\fin}{\regionfont{F}}
   \newcommand{\gin}{\regionfont{G}}
   \newcommand{\hin}{\regionfont{H}}
   \newcommand{\qin}{\regionfont{Q}}
    
\newcommand{\upper}{\regionfont{H}}     % the half-plane

   \newcommand{\Upper}{\overline{\upper}}
   \newcommand{\UpperRint}{\Upper\rintsup}
   \newcommand{\UpperInt}{\Upper\intsup}

\newcommand{\Hgr}{\operatorname{H}}

\newcommand{\End}{\operatorname{End}}
\newcommand{\Hom}{\operatorname{Hom}}

\newcommand{\sub}{\mathrm{sub}}

\newcommand{\ep}{\textsc{ep}}

\newcommand{\pmfont}{\mathbb}
  \newcommand{\Uu}{{\pmfont{U}}}
  \newcommand{\Vv}{{\pmfont{V}}}
  \newcommand{\Ww}{{\pmfont{W}}}
  \newcommand{\Xx}{{\pmfont{X}}}
  \newcommand{\Ii}{{\pmfont{I}}}
  \newcommand{\Jj}{{\pmfont{J}}}
  
  \newcommand{\aaa}{\pmfont{A}}
  \newcommand{\bbb}{\pmfont{B}}
  \newcommand{\ccc}{\pmfont{C}}
  \newcommand{\ddd}{\pmfont{D}}

\newcommand{\multifont}{\mathsf}
\newcommand{\Dgm}{\multifont{Dgm}}
\newcommand{\dgm}{\multifont{dgm}}
\newcommand{\Int}{\multifont{Int}}
\newcommand{\intt}{\multifont{int}}

\newcommand{\Aa}{\multifont{A}}
\newcommand{\Bb}{\multifont{B}}
\newcommand{\Cc}{\multifont{C}}

\newcommand{\Mm}{\multifont{M}}
\newcommand{\Ff}{\multifont{F}}

\newcommand{\metricfont}{\mathrm}
  \newcommand{\bottle}{\metricfont{d_b}}
  \newcommand{\inter}{\metricfont{d_i}}
  \newcommand{\ellinf}{{\metricfont{d}^\infty}}
  
  \newcommand{\exit}{{\metricfont{ex}^\infty}}

  \newcommand{\rk}{\operatorname{\metricfont{r}}}
  \newcommand{\mult}{\metricfont{m}}
  \newcommand{\rect}{\metricfont{Rect}}

\newcommand{\rank}{\operatorname{\mathrm{rank}}}
\newcommand{\nullity}{\operatorname{\mathrm{nullity}}}
\newcommand{\conullity}{\operatorname{\mathrm{conullity}}}
\newcommand{\img}{\operatorname{\mathrm{im}}}
\newcommand{\coimg}{\operatorname{\mathrm{coim}}}
\renewcommand{\ker}{\operatorname{\mathrm{ker}}}
\newcommand{\coker}{\operatorname{\mathrm{coker}}}

\newcommand{\card}{\operatorname{card}}

\newcommand{\qlen}{1em}
\newcommand{\Qlen}{1.5em}

\newcommand{\qem}{\makebox[\qlen]{---}}

\newcommand{\qno}{\makebox[\qlen]{---}}  
\newcommand{\qon}[1]{\makebox[\qlen]{$_{\phantom{}}\bullet_{#1}$}}
\newcommand{\qoff}[1]{\makebox[\qlen]{$_{\phantom{}}\circ_{#1}$}}

\newcommand{\Qno}{\makebox[\Qlen]{--{}--{}--}}
\newcommand{\Qon}[1]{\makebox[\Qlen]{$_{\phantom{}}\bullet_{#1}$}}
\newcommand{\Qoff}[1]{\makebox[\Qlen]{$_{\phantom{}}\circ_{#1}$}}

\begin{document}
%------------------------------------------------------------------

\begin{abstract}
We give a self-contained treatment of the theory of persistence modules indexed over the real line. We give new proofs of the standard results. Persistence diagrams are constructed using measure theory. Linear algebra lemmas are simplified using a new notation for calculations on quiver representations. We show that the stringent finiteness conditions required by traditional methods are not necessary to prove the existence and stability of the persistence diagram. We introduce weaker hypotheses for taming persistence modules, which are met in practice and are strong enough for the theory still to work. The constructions and proofs enabled by our framework are, we claim, cleaner and simpler.
\end{abstract}

%------------------------------------------------------------------
\maketitle
\tableofcontents

%\newpage
%-------------------------------------------------------------------
\section*{Introduction}
%--------------------------------------------------

We intend this paper to be a self-contained treatment of the theory of persistence modules  over the real line. 
We give the best proofs we know of the most important results. Each theorem is located at the appropriate level of abstraction (we believe).

Many authors have studied persistence modules in recent years, and many of the theorems presented here are not original in themselves. The originality lies in the methods we use, which give easier proofs and sharper results.
Our main innovations are these:
\begin{vlist}
\item
We use measure theory to construct persistence diagrams. The existence of a diagram is equivalent to the existence of a certain kind of measure on rectangles in the plane.

\item
We introduce `decorated' real numbers to remove ambiguity about interval endpoints. Decorations are also what make the measure theory work.

\item
We define several kinds of `tameness' for a persistence module. These occur naturally in practice. The {most} restrictive of these, finite type, is what is normally seen in the literature. We show how to work effectively with the less restrictive hypotheses. 

\item
We introduce a special notation for calculations on quiver representations. This considerably simplifies the linear algebra (for instance, in proving the `box lemma').

\end{vlist}

Our goal in introducing these ideas is to enable other authors to define persistence diagrams cleanly, in a wide variety of situations, without imposing unnecessary restrictions (such as assuming a function to be Morse). For instance, it will be seen in forthcoming work that the approach here can be used to define the levelset zigzag persistence of~\cite{Carlsson_dS_M_2009} quite broadly.

This paper owes much to~\cite{Chazal_CS_G_G_O_2008} (and its published journal version~\cite{Chazal_CS_G_G_O_2009}), which established the existence and stability of persistence diagrams for modules whose persistence maps are of finite rank. In the present work we call these modules `q-tame'.

Traditionally, continuous persistence diagrams have been treated in one of two ways. Most commonly, one makes the aggressive assumption that the situation being studied has only finitely many `critical values'. Alternatively, as carried out in~\cite{Chazal_CS_G_G_O_2008} for q-tame modules, the diagram is constructed using a careful limiting process through ever-finer discretisations of the parameter. The former strategy may be appropriate when working with real-world data, where every persistence module really is finite in every way, but it excludes very common theoretical situations. The latter strategy was devised to overcome these restrictions, but unfortunately the limiting arguments turn out to be quite complicated.
Our new approach gives the best of both worlds; we are able to work with broader classes of persistence modules, and we can reason about their diagrams in a clean way using arguments of a finite nature.

The other debt to~\cite{Chazal_CS_G_G_O_2008} is the recasting of stability as a statement about interleaved persistence modules, the so-called `algebraic stability theorem'. In this paper we re-recast the result as a statement about 1-parameter families of measures. This allows us to prove stability results for even quite badly behaved persistence modules.

\bigskip{\bf Overview.}
The paper is organised as follows.

In section~\ref{sec:pmodules}, we set up the basic properties of persistence modules. These can be defined over any partially ordered set; we are primarily interested in persistence modules over the real line. In the best case a persistence module can be expressed as a direct sum of `interval modules', which can be thought of as the atomic building blocks of the theory. Not all persistence modules decompose in this way, so we spend much of this paper developing techniques that work without this assumption. These techniques depend on a thorough understanding of finitely-indexed persistence modules known as `$A_n$-quiver representations'~\cite{Gabriel_1972,Derksen_Weyman_2005}. We introduce a special notation for performing calculations on these quiver representations. This `quiver calculus' is used throughout the paper.\footnote{%
Readers who wish to adopt our notation are invited to contact us for the \LaTeX\ macros.
}

Section~\ref{sec:measures} addresses the question of how to define the diagram of a persistence module. This is easy for modules which decompose into intervals. To handle the general case, we establish an equivalence between diagrams and a certain kind of measure defined on rectangles in the plane. This means that whenever a persistence diagram is sought, it is enough to construct the corresponding persistence measure;
and theorems about a diagram can be replaced by simpler-to-prove theorems about its measure.
The diagram exists wherever the measure takes finite values. This leads to several different notions of `tameness'. There are large classes of examples of naturally occurring persistence modules which are tame enough for their diagrams to be defined everywhere or almost everywhere.

In order to make the measure theory work, we use real numbers that are `decorated' with a superscript $^+$ or~$^-$. We use decorations also to indicate whether a real interval is open, closed or half-open. For persistence modules which decompose into interval modules, these two uses match up perfectly.

There is a `snapping principle' by which our abstractly defined diagrams are seen to agree with the diagrams that are produced by the standard algorithms~\cite{Edelsbrunner_L_Z_2002,Zomorodian_Carlsson_2005}, for example when working with finite simplicial complexes derived from real data.

In section~\ref{sec:interleaving}, we study interleavings. An interleaving is an approximate isomorphism between two persistence modules. They occur naturally in applications when the input data are known only up to some bounded error. After presenting the basic properties, we give a clean proof of the technical lemma (from~\cite{Chazal_CS_G_G_O_2008}) that two interleaved modules can be interpolated by a 1-parameter family.

Section~\ref{sec:isometry} is devoted to the isometry theorem, which asserts that the interleaving distance between two persistence modules is equal to the bottleneck distance between their persistence diagrams. The two inequalities that comprise this result are treated separately. One direction is the celebrated stability theorem of~\cite{CohenSteiner_E_H_2007}. The more recent converse inequality appears in~\cite{Lesnick_2011}.
We formulate the stability theorem as a statement about measures and their diagrams. The proof of this more abstract result closely follows the original proof in~\cite{CohenSteiner_E_H_2007}.

Our version of the isometry theorem supposes that the persistence modules are `q-tame'. We consider this to be the natural realm of the theorem. We also prove a more general version of the stability theorem which allows us to compare diagrams of persistence modules with no assumptions on their tameness: wherever the two diagrams are defined, they must be close to each other.

Given the length of this paper, the reader may wonder if the framework developed here is truly simpler than existing approaches. In section~\ref{sec:examples}, we give examples of how to use the results and ideas in this paper.

\bigskip{\bf Related Work.}
The early history of persistence is concerned with the quantity
\[
\rk_s^t = \rank(\Hgr(X_s) \to \Hgr(X_t))
\]
for an object~$X$ represented at two different scales $s,t$, and where $\Hgr$ is homology.
This appeared in the early 1990s in the work of Frosini~\cite{Frosini_1992}, with different notation and under the name `size function'.
Independently, a few years later, Robins~\cite{Robins_1999} introduced the term `persistent Betti numbers' for quantities of the form $\rk_{\epsilon}^{\epsilon+\rho}$, and noted their stability with respect to Hausdorff distance.

The modern theory of persistence is built on three central pillars.

\begin{vlist}
\item
The persistence diagram, and an algorithm for computing it, were introduced by Edelsbrunner, Letscher and Zomorodian~\cite{Edelsbrunner_L_Z_2002}.

\item
Zomorodian and Carlsson~\cite{Zomorodian_Carlsson_2005} defined persistence modules, indexed by the natural numbers and viewed as graded modules over the polynomial ring~$\kk[t]$.

\item
Cohen-Steiner, Edelsbrunner and Harer~\cite{CohenSteiner_E_H_2007} proved the stability theorem.
\end{vlist}

All three papers make strong finiteness assumptions about the starting data. In~\cite{Chazal_CS_G_G_O_2008}, the results are generalised to persistence modules parametrised over the real line, under the assumption that $\rk_s^t < \infty$ for $s < t$.
The present paper extends that work, with many new concepts and proofs.

We draw attention to two recent papers which share our goal of understanding continuous-parameter persistence modules.
Lesnick~\cite{Lesnick_2011} gives an extensive algebraic treatment of modules over one or more real parameters. The converse stability inequality (and hence the isometry theorem) appears for the first time there.
Bubenik and Scott~\cite{Bubenik_Scott_2012} develop the category-theoretical view of persistence modules.

Both \cite{Lesnick_2011} and~\cite{Bubenik_Scott_2012} appeared during the writing of this paper. There is a fair amount of overlap between the three papers, most of it reached independently, but not quite all: from Lesnick~\cite{Lesnick_2011} we learned of the results of Webb~\cite{Webb_1985}, which resolved a sticking-point for us.

%-----------------------------------------
\section*{Multisets}

Persistence diagrams are multisets rather than sets. For our purposes, a multiset is a pair $\Aa = (S, \mult)$
where $S$ is a set and
\[
\mult: S \to \{1, 2, 3, \dots \} \cup \{ \infty \}
\]
is the multiplicity function, which tells us how many times each element of $S$ occurs in~$\Aa$.

Here are our conventions regarding multisets.
\begin{vlist}
\item
The cardinality of $\Aa = (S, \mult)$ is defined to be
\[
\card \Aa = \sum_{s\in S} \mult(s)
\]
which takes values in $\{ 0, 1, 2, \dots \} \cup \{ \infty \}$. We do not distinguish between different infinite cardinals.

\item
We never form the intersection of two multisets, but we will sometimes restrict a multiset~$\Aa$ to a set~$B$:
\[
\Aa|_B = (S \cap B, \mult|_{S \cap B})
\]
We may write this as $\Aa \cap B$ when $\Aa|_B$ is typographically inconvenient.

\item
A pair $(B, \mult)$ where
\[
\mult: B \to \{0, 1, 2, \dots \} \cup \{ \infty \}
\]
is implicitly regarded as defining a multiset $\Aa = (S, \mult|_S)$ where $S = B - \mult^{-1}(0)$ is the support of~$\mult$.

\item
If $\Aa = (S, \mult)$ is a multiset and $f: S \to B$ where $B$ is a set, then the notation
\[
\{ f(a) \mid a \in \Aa \}
\]
is interpreted as the multiset in $B$ with multiplicity function
\[
\mult'(b) = \sum_{f^{-1}(b)} \mult(s)
\]

\end{vlist}

Except in definitions like these, we seldom refer explicitly to $S$. 

%-------------------------------------------------------------------
\section{Persistence Modules}
\label{sec:pmodules}

All vector spaces are taken to be over an arbitrary field~$\kk$, fixed throughout the paper.

%--------------------------------------------------
\subsection{Persistence modules over a real parameter}
\label{subsec:11}

A {persistence module} $\Vv$ over the real numbers~$\Rr$ is defined to be an indexed family of vector spaces
\[
(V_t \mid t \in \Rr),
\]
and a doubly-indexed family of linear maps
\[
(v_s^t: V_s \to V_t \mid s \leq t)
\]
which satisfy the composition law
\[
v_s^t \circ v_r^s = v_r^t
\]
whenever $r \leq s \leq t$, and where $v_t^t$ is the identity map on $V_t$.

\begin{remark}
Equivalently, a persistence module is a functor from the real line (viewed as a category with a unique morphism $s \to t$ whenever $s \leq t$) to the category of vector spaces. The uniqueness of the morphism $s \to t$ corresponds to the fact that all possible compositions
\[
v^t_{s_{n-1}} \circ v^{s_{n-1}}_{s_{n-2}} \circ \dots \circ v^{s_2}_{s_1} \circ v^{s_1}_s
\]
from $V_s$ to $V_t$ are equal to each other, and in particular to $v_s^t$.
\end{remark}

Here is the standard class of examples from applied topology.
Let $X$ be a topological space and let $f : X \to \Rr$ be a function (not necessarily continuous). Consider the sublevel sets:
\[
X^t = 
(X,f)^t =
\left\{
x \in X \mid f(x) \leq t
\right\}
\]
The inclusion maps
\[
(i_s^t: X^s \to X^t \mid s \leq t)
\]
trivially satisfy the composition law
\[
i_s^t \circ i_r^s = i_r^t
\]
whenever $r \leq s \leq t$, and  $i_t^t$ is the identity on $X^t$. Collectively this information is called the {\bf sublevelset filtration} of $(X,f)$ and may be denoted $\Xx_\sub$ or $\Xx^f_\sub$.

We obtain a persistence module by applying any functor from topological spaces to vector spaces. For example, let $\Hgr = \Hgr_p(-; \kk)$ be the functor `$p$-dimensional singular homology  with coefficients in~$\kk$'. We define a persistence module $\Vv$ by setting
\[
V_t = \Hgr(X^t),
\]
and
\[
v_s^t = \Hgr(i_s^t) : \Hgr(X^s) \to \Hgr(X^t)
\]
(the maps on homology induced by the inclusion maps).
More simply:
\[
\Vv = \Hgr(\Xx_\sub)
\]

In the applied topology literature, there are many examples $(X,f)$ whose persistent homology is of interest. Very often $X$ is a finite simplicial complex and each $X^t$ is a subcomplex. It follows that the vector spaces $\Hgr(X^t)$ are finite-dimensional; and as $t$ increases there are finitely many `critical values' at which the complex changes, growing by one or more new cells. Suppose these critical values are
\[
a_1 < a_2 < \dots < a_n.
\]
Then all the information in the persistence module is contained in the {finite} diagram
\[
\Hgr(X^{a_1}) \to
\Hgr(X^{a_2}) \to
\dots \to
\Hgr(X^{a_n})
\]
of {finite}-dimensional vector spaces and linear maps.
In this situation,
\begin{vlist}
\item
the isomorphism type of $\Hgr(\Xx_\sub)$ admits a compact description~\cite{Edelsbrunner_L_Z_2002,Zomorodian_Carlsson_2005};
\item
there is a fast algorithm for computing this description~\cite{Edelsbrunner_L_Z_2002,Zomorodian_Carlsson_2005};
\item
the description is continuous (indeed 1-Lipschitz) in $f$~\cite{CohenSteiner_E_H_2007}.
\end{vlist}	
This description is the famous \textbf{persistence diagram}, or \textbf{barcode}.

In practical applications all the persistence modules that we encounter are finite, for the trivial reason that a computer only stores finite data. However, there are good grounds for extending the results of~\cite{Edelsbrunner_L_Z_2002,Zomorodian_Carlsson_2005,CohenSteiner_E_H_2007} to a more general setting. For example, theoretical guarantees are commonly formulated in terms of an idealised model, such as a continuous space to which the data form an approximation.  Finiteness becomes unnatural and difficult to enforce in these ideal models, but one still wants the main results to be true.

Following~\cite{Chazal_CS_G_G_O_2008}, we say that a persistence module $\Vv$ is {\bf tame} if
\[
\rk_s^t =
\rank(v_s^t) < \infty
\quad
\text{whenever $s < t$}.
\]
Since the word `tame' is overloaded with too many meanings in the persistence literature, we offer {\bf q-tame} as an alternative (for reasons that will be explained later).

It is shown in~\cite{Chazal_CS_G_G_O_2008} that persistence diagrams can be constructed for q-tame persistence modules, and that these diagrams are stable with respect to certain natural metrics. We reproduce these results here, using different methods for many of the arguments. We complete the picture by showing that the map from q-tame persistence modules to persistence diagrams is an isometry.
This isometry theorem is due independently to Lesnick~\cite{Lesnick_2011}.

We believe that q-tame persistence modules are the `right' class of objects to work with. This is for two complementary reasons: (i) we can prove almost everything we want to prove about q-tame modules and their persistence diagrams; and (ii) they occur in practice. For example, a continuous function on a finite simplicial complex has q-tame sublevelset persistent homology (Theorem~\ref{thm:poly-q-tame}).
See~\cite{Chazal_dS_O_2012} for many other examples.

%--------------------------------------------------
\subsection{Different index sets}
\label{subsec:index}

We can define persistence modules over any partially ordered set~$\Tt$, formally in the same way as for~$\Rr$, by specifying indexed families
\[
( V_t \mid t \in \Tt )
\quad
\text{and}
\quad
( v_s^t \mid s,t \in \Tt,\, s \leq t )
\]
of vector spaces and linear maps, for which $v_r^t = v_s^t \circ v_r^s$ whenever $r \leq s \leq t$, and where $v_t^t$ is the identity on~$V_t$.
The resulting collection of data is called a {\bf $\Tt$-persistence module} or a persistence module {\bf over $\Tt$}.

If $\Vv$ is a $\Tt$-persistence module and $\Ss \subset \Tt$, then we get an $\Ss$-persistence module by considering only those spaces and maps with indices in~$\Ss$. This is called the \textbf{restriction} of $\Vv$ to~$\Ss$, and may be written $\Vv_\Ss$.

Most commonly, we work with locally finite subsets $\Tt \subset \Rr$, that is, those subsets with no accumulation points in~$\Rr$. We collect information about an $\Rr$-persistence module by considering its restriction to finite or locally finite subsets. This works well because persistence modules over $\{1,2,\dots,n\}$ or over the integers~$\Zz$ are well understood.

In section~\ref{sec:interleaving}, we will consider some other posets.

%--------------------------------------------------
\subsection{Module categories}
\label{subsec:mod-cat}

A homomorphism $\Phi$ between two $\Tt$-persistence modules $\Uu, \Vv$ is a collection of linear maps
\[
( \phi_t : U_t \to V_t \mid t \in \Tt )
\]
such that the diagram
\[
\begin{diagram}
\dgARROWLENGTH=2em
\node{U_s}
  \arrow{e,t}{u_s^t}
  \arrow{s,l}{\phi_s}
\node{U_t}
  \arrow{s,r}{\phi_t}
\\
\node{V_s}
  \arrow{e,t}{v_s^t}
\node{V_t}
\end{diagram}
\]
commutes for all $s \leq t$. Composition is defined in the obvious way, as are identity homomorphisms. This makes the collection of persistence modules into a category.

The category contains kernel, image, and cokernel objects for every map $\Phi$, and there is a zero object.

Write
\begin{align*}
\Hom(\Uu, \Vv) &= \{ \text{homomorphisms $\Uu \to \Vv$} \},
\\
\End(\Vv) &= \{ \text{homomorphisms $\Vv \to \Vv$} \}.
\end{align*}
Note that $\End(\Vv)$ is a $\kk$-algebra.

Later we will consider homomorphisms that shift the index. We will introduce these when we consider the interleaving relation between persistence modules.

%--------------------------------------------------
\subsection{Interval modules}
\label{subsec:intervals}

The basic building blocks of persistence are the {interval modules}. One seeks to understand a persistence module by decomposing it into intervals. This is not always possible, but it is possible often enough for our purposes.

Let $\Tt \subseteq \Rr$, and let $J \subseteq \Tt$ be an interval.\footnote{%
By this we mean that if $r < s < t$ are elements of~$T$ with $r, t \in J$, then $s \in J$.}
Then $\Ii^J$ is defined to be the $\Tt$-persistence module with spaces
\[
I_t = 
\left\{
\begin{array}{ll}
 \kk \quad& \text{if $t \in J$}\\
 0 & \text{otherwise}
\end{array}
\right.
\]
and maps
\[
i_s^t = 
\left\{
\begin{array}{ll}
 1 \quad& \text{if $s, t \in J$}\\
 0 & \text{otherwise}
\end{array}
\right.
\]
In informal language, $\Ii^J$ represents a `feature' which `persists' over the interval~$J$ but is absent elsewhere. We write $\Ii^J_\Tt$ when we wish to specify the index set unambiguously.

We now establish notation for describing interval modules.

If $\Tt$ is a locally finite subset of $\Rr$, then any bounded interval contains its endpoints. We write these as closed intervals, and unbounded intervals as open or half-open intervals. For example, in the case of $\Zz$-persistence modules there are four kinds of interval module:
\[
\Ii^{[m,n]}
\qquad
\Ii^{(-\infty , n]}
\qquad
\Ii^{[m, +\infty)}
\qquad
\Ii^{(-\infty, +\infty)}
\]
For visual clarity, we sometimes lower the superscripts, writing
\[
\Ii{[m,n]}
\qquad
\Ii{(-\infty , n]}
\qquad
\Ii{[m, +\infty)}
\qquad
\Ii{(-\infty, +\infty)}
\]
instead.

For persistence modules over~$\Rr$, it becomes important to distinguish intervals which have the same endpoints but different topology (open, closed, half-open). For this purpose we introduce \textbf{decorated real numbers}, which are written as ordinary real numbers but with a superscript $^+$ (plus) or $^-$ (minus). For finite intervals we adopt the following dictionary:
\begin{align*}
\lgroup p^-, q^- \rgroup &\quad\text{means}\quad [p,q) \\
\lgroup p^-, q^+ \rgroup &\quad\text{means}\quad [p,q] \\
\lgroup p^+, q^- \rgroup &\quad\text{means}\quad (p,q) \\
\lgroup p^+, q^+ \rgroup &\quad\text{means}\quad (p,q]
\end{align*}
We require $p < q$, except for the special case $\lgroup r^-, r^+ \rgroup$, which represents the 1-point interval $[r,r]$.

We also include the symbols $-\infty$ and $+\infty$, for infinite intervals. Since real intervals are always open at infinity these implicitly carry the superscripts $-\infty^+$ and $+\infty^-$, but we usually omit the superscripts. The interval notation extends in the obvious way, so
\[
\lgroup -\infty, q^+ \rgroup
\quad\text{means}\quad
(-\infty, q]
\]
for instance.

When we wish to refer to a decorated real number but we don't know what the decoration is, we use an asterisk. Thus $p^*$ means $p^+$ or $p^-$. The notation for an arbitrary interval is $\lgroup p^*, q^* \rgroup$, where $p^* < q^*$ in the obvious ordering on decorated reals.

Here are some visual conventions for representing interval modules over~$\Rr$.

We work in the half-plane
\[
\upper = \left\{ (p,q) \mid p \leq q \right\}
\]
of points in~$\Rr^2$ which lie on or above the diagonal.
Any finite interval module $\Ii{\lgroup p^*, q^* \rgroup}$ may be represented in several different ways (see Figure~\ref{fig:interval}):
\begin{vlist}
\item
as an interval in the real line;
\item
as $\rank(i_s^t)$, viewed as a function $\upper \to \{0, 1\}$;
\item
as a point $(p,q)$ in $\upper$, with a tick to specify the decoration.
\end{vlist}
\begin{figure}
\centerline{
\includegraphics[scale=0.75]{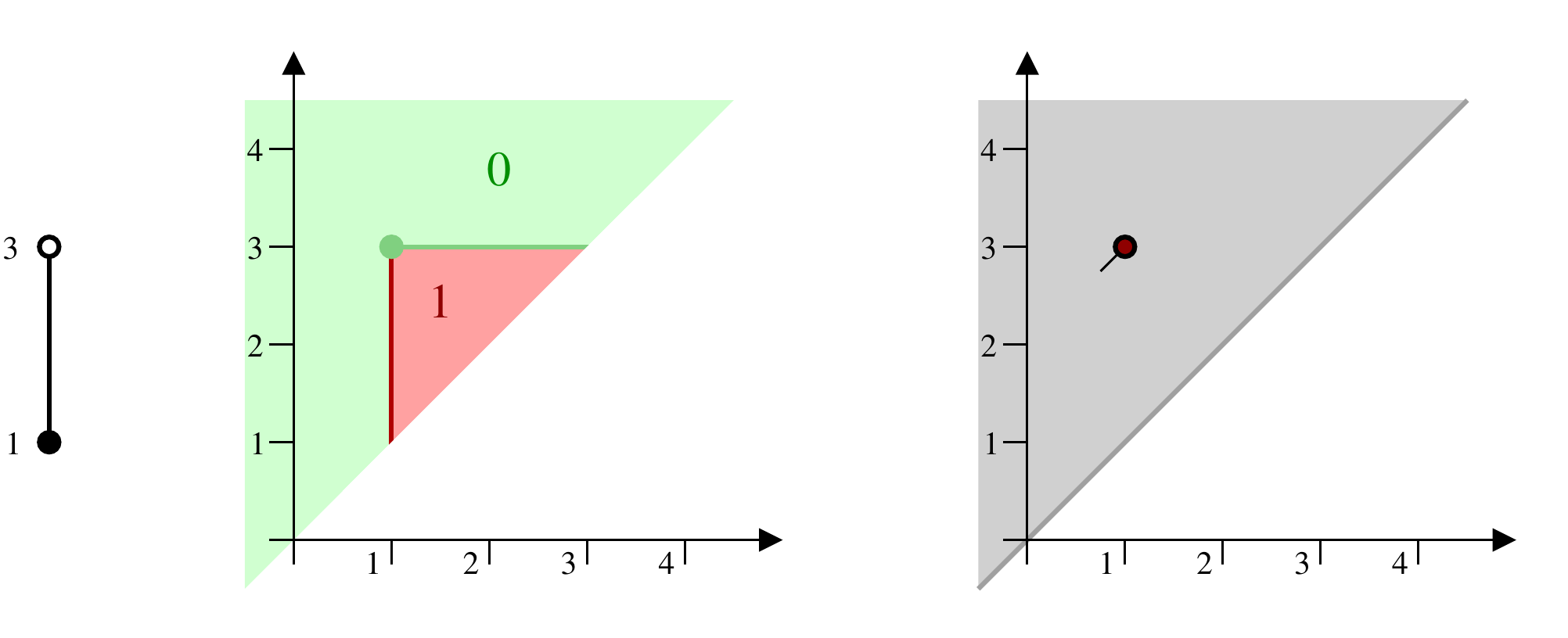}
}
\caption{The interval, rank function, and decorated point representations of the interval module $\Ii^{[1,3)} = \Ii^{\lgroup 1^-, 3^- \rgroup}$.
}
\label{fig:interval}
\end{figure}
Here are the four tick directions explicitly:
\begin{alignat*}{2}
\lgroup p^-, q^+ \rgroup &=
\raisebox{-0.8ex}{\includegraphics[scale=0.75]{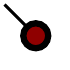}}
&\qquad
\lgroup p^+, q^+ \rgroup &=
\raisebox{-0.8ex}{\includegraphics[scale=0.75]{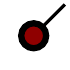}}
\\
\lgroup p^-, q^- \rgroup &=
\raisebox{-0.8ex}{\includegraphics[scale=0.75]{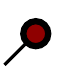}}
&\qquad
\lgroup p^+, q^- \rgroup &=
\raisebox{-0.8ex}{\includegraphics[scale=0.75]{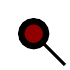}}
\end{alignat*}
The convention is that the tick points into the quadrant suggested by the decorations.

A fourth option is to draw the point $(p,q)$ without indicating the decoration. This `forgetful' representation is the classical convention for persistence diagrams, and is adequate for most purposes. However, the extra precision provided by decoration is important for the correspondence between diagrams and measures.

To represent an infinite interval as a (decorated or undecorated) point, we work in the extended half-plane
\[
 \Upper
\;\;=\;\;
\upper
\;\;\cup\;\; \{ -\infty \} \times \Rr
\;\;\cup\;\; \Rr \times \{ + \infty \}
\;\;\cup\;\; \{ (-\infty, +\infty) \}
\]
which can be drawn schematically as a triangle. See Figure~\ref{fig:deltabar}.
\begin{figure}
\centerline{
\includegraphics[height=2.25in]{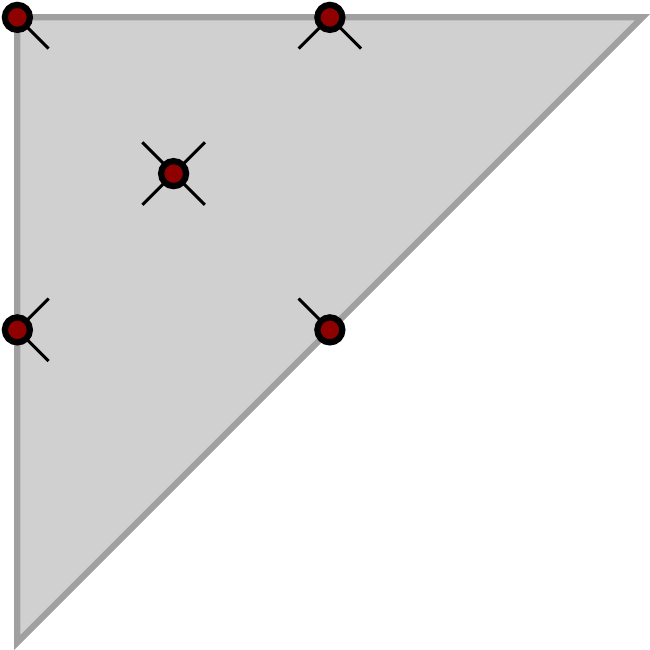}
}
\caption{The extended half-plane $\Upper$ with examples of each type of interval, drawn as points with ticks.}
\label{fig:deltabar}
\end{figure}
%

%--------------------------------------------------
\subsection{Interval decomposition}
\label{subsec:decomp}

The \textbf{direct sum} $\Ww = \Uu \oplus \Vv$ of two persistence modules $\Uu, \Vv$ is defined as follows:
\[
W_t = U_t \oplus V_t,
\quad
w_s^t = u_s^t \oplus v_s^t
\]
This generalises immediately to arbitrary (finite or infinite) direct sums.

A persistence module $\Ww$ is {\bf indecomposable} if the only decompositions $\Ww = \Uu \oplus \Vv$ are the trivial decompositions $\Ww \oplus 0$ and $0 \oplus \Ww$.

Direct sums play both a \textit{synthetic} role and an \textit{analytic} role in our theory. On the one hand, given an indexed family of intervals $\left( J_\ell \mid \ell \in L \right)$ we can synthesise a persistence module
\[
\Vv = \bigoplus_{\ell \in L} \Ii^{J_\ell}
\]
whose isomorphism type depends only on the multiset $\{ J_\ell \mid \ell \in L \}$. In light of the direct-sum decomposition, we can think of $\Vv$ as having an independent feature for each $\ell \in L$, supported over the interval $J_\ell$.
On the other hand, we can attempt to analyse a given persistence module $\Vv$ by decomposing it into interval modules.

We now present the necessary theory for this.
A `building block' in a module category can be characterised by having a comparatively simple endomorphism ring. Interval modules have the simplest possible:

\begin{proposition}
\label{prop:EndI}
Let $\Ii = \Ii^J_\Tt$ be an interval module over $\Tt \subseteq \Rr$; then $\End(\Ii) = \kk$.
\end{proposition}

\begin{proof}
Any endomorphism of $\Ii$ acts on each nonzero $I_t = \kk$ by scalar multiplication. By the commutation law, it is the same scalar for each~$t$.
\end{proof}

\begin{proposition}
\label{prop:indecomp}
Interval modules are indecomposable.
\end{proposition}

\begin{proof}
Given a decomposition $\Ii = \Uu \oplus \Vv$, the projection maps onto $\Uu$ and $\Vv$ are idempotents in the endomorphism ring of~$\Ii$. But the only idempotents of $\End(\Ii) = \kk$ are 0 and 1.
\end{proof}

\begin{theorem}[Krull--Remak--Schmidt--Azumaya]
\label{thm:azumaya}
Suppose that a persistence module $\Vv$ over $\Tt \subseteq \Rr$ can be expressed as a direct sum of interval modules in two different ways:
\[
\Vv
  = \bigoplus_{\ell \in L} \Ii^{J_\ell}
  = \bigoplus_{m \in M} \Ii^{K_m}
\]
Then there is a bijection $\sigma: L \to M$ such that $J_\ell = K_{\sigma(\ell)}$ for all~$\ell$.
\end{theorem}

\begin{proof}
This is from Azumaya~\cite{Azumaya_1950} (Theorem~1), plus the trivial observation that $\Ii^J \cong \Ii^K$ implies $J = K$.
The theorem requires a certain condition on the endomorphism ring of each possible interval module: if $\alpha, \beta \in \End(\Ii)$ are non-isomorphisms then $\alpha+\beta$ is a non-isomorphism. Since each $\End(\Ii) = \kk$, the only non-isomorphism is the zero map and the condition is satisfied.
\end{proof}

In other words, {provided} we can decompose a given persistence module~$\Vv$ as a direct sum of interval modules, {then} the multiset of intervals is an isomorphism invariant of~$\Vv$. But when does such a decomposition exist?

\begin{theorem}[Gabriel, Auslander, Ringel--Tachikawa, Webb]
\label{thm:gabriel+}
Let $\Vv$ be a persistence module over $\Tt \subseteq \Rr$. In each of the following situations, $\Vv$ can be decomposed as a direct sum of interval modules:
\begin{enumerate}
\item
$\Tt$ is a finite set.

\medskip
\item
$\Tt$ is a locally finite subset of $\Rr$ and each $V_t$ is finite-dimensional.
\end{enumerate}
On the other hand, {\rm (3)} there exists a persistence module over $\Zz$ (indeed, over the nonpositive integers) which does not admit an interval decomposition.
\end{theorem}

\begin{remark}
Crawley-Boevey~\cite{CrawleyBoevey_2012} has recently shown that a persistence module over~$\Rr$ admits an interval decomposition if each $V_t$ is finite-dimensional. Thus, statement~(2) of the theorem is valid for all $\Tt \subseteq \Rr$.
\end{remark}

\begin{proof}
(1) The decomposition of a diagram
\[
V_1 \to V_2 \to \dots \to V_n
\]
into interval summands, when each $\dim(V_i)$ is finite, is one of the simpler instances of Gabriel's theorem~\cite{Gabriel_1972}; see \cite{Zomorodian_Carlsson_2005} or~\cite{Carlsson_deSilva_2010} for a concrete explanation.
The extension to infinite-dimensional modules follows abstractly from a theorem of Auslander~\cite{Auslander_1974} and, independently, Ringel and Tachikawa~\cite{Ringel_Tachikawa_1975}.
Alternatively, note that the argument given in~\cite{Carlsson_deSilva_2010} does not really require finite-dimensionality.

(2) We may assume that $\Tt = \Zz$ because any locally finite subset of~$\Rr$ is equivalent as an ordered set to a subset of~$\Zz$. Then this follows from Propostions 2 and~3 and Theorem~3 of Webb~\cite{Webb_1985}.

(3) Webb~\cite{Webb_1985} gives the following example, indexed over the nonpositive integers:
\begin{alignat*}{2}
V_{0} &= \{ \text{sequences $(x_1, x_2, x_3, \dots)$ of real numbers} \}
\\
V_{-n} &= \{ \text{sequences with $x_1 = \dots = x_{n} = 0$} \}
	&& \quad\text{for $n \geq 1$}
\end{alignat*}
The maps $v_{-m}^{-n}$ are the inclusions $V_{-m} \subset V_{-n}$ ($m \geq n$).

Suppose $\Vv$ has an interval decomposition. Since each map $v_{-n-1}^{-n}$ is injective, all of the intervals must be of the form $[-n,0]$ or $(-\infty,0]$. Since $\dim(V_{-n}/V_{-n-1}) = 1$, each interval $[-n,0]$ occurs with multiplicity~1. Since $\bigcap V_{-n} = \{0\}$, the interval $(-\infty,0]$ does not occur at all.

This would imply that $\Vv \cong \bigoplus_{n \geq 0} \Ii^{[-n,0]}$. This contradicts the fact that $\dim(V_0)$ is uncountable. Therefore $\Vv$ does not admit an interval decomposition after all.
\end{proof}

\begin{remark}
There are several other examples of persistence modules not decomposing into intervals.
Lesnick~\cite{Lesnick_2012pc} has an example indexed over~$\Zz$ which is countable-dimensional over each index; and Crawley-Boevey~\cite{CrawleyBoevey_2012pc} has an example indexed over~$\Rr$ which is q-tame.
\end{remark}

\medskip
For a persistence module which decomposes into intervals, the way is now clear to define its persistence diagram. Simply record which intervals occur in the decomposition (with multiplicity): see section~\ref{subsec:pd1}. Theorem~\ref{thm:azumaya} tells us that this is an isomorphism invariant.

However, we have seen that arbitrary persistence modules over~$\Rr$ are not guaranteed an interval decomposition. Here are three ways around the problem:

\begin{vlist}
\item
Work in restricted settings to ensure that the structure of~$\Vv$ depends only on finitely many index values $t \in \Rr$. For example, if $X$ is a compact manifold and $f$ is a Morse function, then $\Hgr(\Xx_\sub)$ is determined by the finite sequence
\[
\Hgr(X^{a_1}) \to \Hgr(X^{a_2}) \to \dots \to \Hgr(X^{a_n})
\]
where $a_1, a_2, \dots, a_n$ are the critical values of~$f$. This is the traditional approach. In this setting, the word `tame' is often used to demarcate pairs $(X,f)$ for which $\Hgr(\Xx_\sub)$ is determined by a finite diagram of finite-dimensional vector spaces.

\item
Sample the persistence module $\Vv$ over a finite grid. Consider limits as the grid converges to the whole real line. This is the approach taken in~\cite{Chazal_CS_G_G_O_2008}, where it is shown that the q-tame hypothesis is sufficient to guarantee good limiting behaviour.

\item
Show that the persistence intervals (in the decomposable case) can be inferred from the behaviour of $\Vv$ on short finite index sets. Apply this indirect definition to define the persistence diagram in the non-decomposable case. This is the method of `rectangle measures' developed in this paper. %We assume that the modules are q-tame.

\end{vlist}

%Although Theorem~\ref{thm:gabriel+} doesn't cover it, it remains possible that there is a decomposition theorem for q-tame modules. Such a theorem would reduce the impetus for some of the arguments in this paper.

%--------------------------------------------------
\subsection{The persistence diagram of a decomposable module}
\label{subsec:pd1}

If a persistence module $\Vv$ indexed over $\Rr$ can be decomposed
\[
\Vv \cong \bigoplus_{\ell \in L} \Ii{\lgroup p_\ell^*, q_\ell^* \rgroup},
\]
then we define the {\bf decorated persistence diagram} to be the multiset
\[
\Dgm(\Vv)
= \Int(\Vv)
= {\{} (p_\ell^*, q_\ell^*) \mid \ell \in L {\}}
\]
and the {\bf undecorated persistence diagram} to be the multiset
\[
\dgm(\Vv)
= \intt(\Vv)
= {\{} (p_\ell, q_\ell) \mid \ell \in L {\}} - \Delta
\]
where $\Delta = \{ (r,r) \mid r \in \Rr \}$ is the diagonal in the plane.

\begin{remark}
In section~\ref{subsec:pd+} we will give a quite different definition of the persistence diagram of~$\Vv$, based on the persistence measure rather than the interval decomposition. When occasionally we must distinguish between the two notions, we use the alternate names $\Int, \intt$ rather than $\Dgm, \dgm$ for the diagrams defined here.
\end{remark}

Theorem~\ref{thm:azumaya} implies that $\Dgm(\Vv)$ and $\dgm(\Vv)$ do not depend on the decomposition of~$\Vv$.

Notice that $\Dgm$ is a multiset of decorated points in $\Upper$, whereas $\dgm$ is a multiset of undecorated points in the interior of $\Upper$. Here `interior' means that we exclude the diagonal but keep the points at infinity.
The information retained by $\dgm$ is precisely the information we care about later, when we discuss bottleneck distances. See section~\ref{sec:isometry}.

Let us see how these definitions play out in an example of the traditional kind.

\begin{example}
\label{ex:morse}
Consider the curve in~$\Rr^2$ shown in Figure~\ref{fig:curve-height}, filtered by the height function. The sublevelset persistent homology decomposes into half-open intervals as follows:
\begin{align*}
\Hgr_0(\Xx_\sub)
&\cong
\Ii{\lgroup a^-, +\infty \rgroup} \oplus \Ii{\lgroup b^-, c^- \rgroup} \oplus \Ii{\lgroup d^-, e^- \rgroup}
\\
\Hgr_1(\Xx_\sub)
&\cong
\Ii{\lgroup f^-, +\infty \rgroup}
\end{align*}
If we let $[x]$ denote the chain corresponding to the critical point with critical value~$x$, then the three summands of $\Hgr_0$ can be taken to be generated by $[a]$, $[b]-[a]$, and $[d]-[a]$ respectively.
For instance, the independent 0-cycle $[b]-[a]$, which is born at time~$b$, becomes a boundary at time~$c$. Thus it gives rise to the half-open interval $[b,c) = \lgroup b^-, c^- \rgroup$.
\begin{figure}
\centerline{
\hfill\hfill
\includegraphics[scale=0.9]{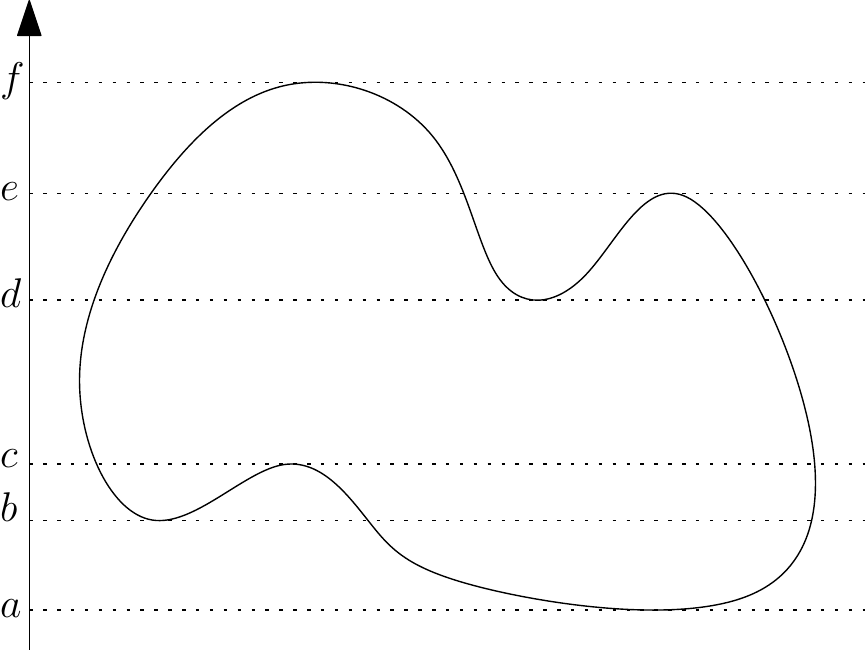}
\hfill
\includegraphics[scale=0.9]{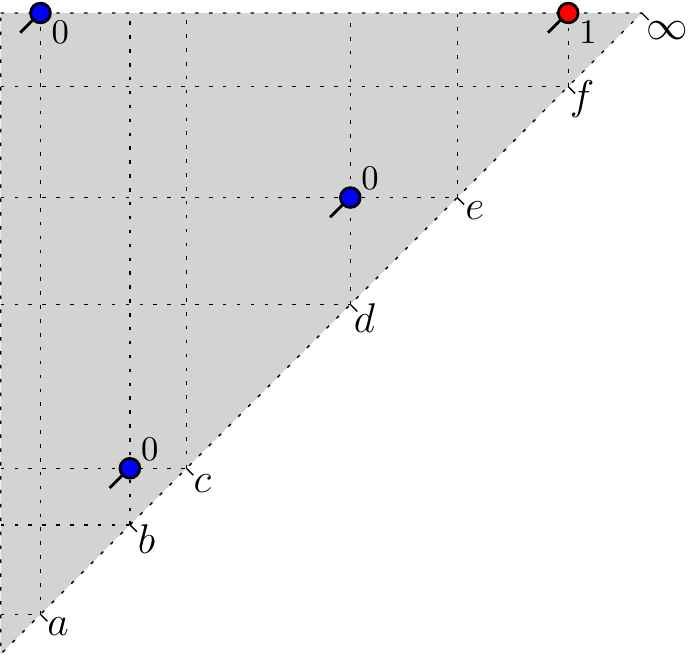}
\hfill\hfill
}
\caption{A traditional example in persistence theory: (left) $X$ is a smoothly embedded curve in the plane, and $f$ is its $y$-coordinate or `height' function; (right) the decorated persistence diagram of $\Hgr(\Xx_\sub)$. There are three intervals in $\Hgr_0$ and one interval in $\Hgr_1$.}
\label{fig:curve-height}
\end{figure}
\end{example}

\begin{remark}
For a Morse function on a compact manifold with critical values $(a_i)$, the intervals are always half-open, of type $[a_i,a_j) = \lgroup a_i^-,a_j^-\rgroup$.
See section~\ref{subsec:snapping}.
%
%This is because $\Hgr(\Xx_\sub)$ is constant over each interval $[a_i,a_{i+1})$, thanks to the fact that $X^{a_i} \to X^{b}$ is a homotopy equivalence whenever $a_i \leq b < a_{i+1}$.
\end{remark}

%--------------------------------------------------
\subsection{Quiver calculations}
\label{subsec:quiver}

We now set up the notation and algebraic tools for handling persistence modules over a finite index set.

A persistence module $\Vv$ indexed over a finite subset
\[
\Tt: \quad a_1 < a_2 < \dots < a_n
\]
of the real line can be thought of as a diagram of $n$~vector spaces and $n-1$~linear maps:
\[
\Vv: \quad
V_{a_1}
\to 
V_{a_2}
\to 
\ldots
\to 
V_{a_n}
\]
Such a diagram is a \textbf{representation} of the following \textbf{quiver}:
\[
\bullet \longrightarrow \bullet \longrightarrow \dots \longrightarrow \bullet
\]

We have seen (Theorem~\ref{thm:gabriel+}) that $\Vv$ decomposes as a finite sum of interval modules $\Ii{[a_i, a_j]}$. When $n$ is small, we can represent these interval modules pictorially. The following example illustrates how.

\begin{example}
\label{ex:int-notation1}
Let $a < b < c$. There are six interval modules over $\{a,b,c\}$, namely:
\begin{align*}
\Ii{[a,a]} &= \qon{a}\qem\qoff{b}\qem\qoff{c}&
\Ii{[a,b]} &= \qon{a}\qem\qon{b}\qem\qoff{c}&
\Ii{[a,c]} &= \qon{a}\qem\qon{b}\qem\qon{c}&
\\
\Ii{[b,b]} &= \qoff{a}\qem\qon{b}\qem\qoff{c}&
\Ii{[b,c]} &= \qoff{a}\qem\qon{b}\qem\qon{c}&
\\
\Ii{[c,c]} &= \qoff{a}\qem\qoff{b}\qem\qon{c}
\end{align*}
In the notation we use filled circles~$\bullet$ to indicate where the module has rank~1, and clear circles~$\circ$ to indicate where the module has rank~0. The connecting maps have full rank.
\end{example}

Now let $\Vv$ be a persistence module indexed over~$\Rr$. For any finite set of indices
\[
\Tt: \quad a_1 < a_2 < \dots < a_n
\]
and any interval $[a_i,a_j] \subseteq \Tt$, we define the multiplicity of $[a_i,a_j]$ in~$\Vv_\Tt$ to be the number of copies of $\Ii{[a_i,a_j]}$ to occur in the interval decomposition of $\Vv_\Tt$. This takes values in the set $\{0, 1, 2, \dots, \infty \}$. (We do not distinguish different infinite cardinals.)

It is useful to have notation for these multiplicities. Again, we define by example.

\begin{example}
We write
\[
\langle [b,c] \mid \Vv_{a,b,c} \rangle
\quad \text{or} \quad
\langle\, \qoff{a} \qem \qon{b} \qem \qon{c} \mid \Vv \,\rangle
\]
for the multiplicity of $\qoff{a}\qem\qon{b}\qem\qon{c}$ in the 3-term module
\[
\Vv_{a,b,c} = (V_a \to V_b \to V_c).
\]
When $\Vv$ is clear from the context, we may simply write
\[
\langle\, \qoff{a} \qem \qon{b} \qem \qon{c} \,\rangle.
\]
The abbreviation $\langle [b,c] \rangle$ is not permitted since it is ambiguous. For example, $\langle [b,c] \mid V_{b,c} \rangle$ and $\langle [b,c] \mid V_{a,b,c} \rangle$ are not generally the same. See Proposition~\ref{prop:restriction} and Example~\ref{ex:restriction}.
\end{example}

\begin{example}
The invariants of a single linear map $V_a \stackrel{v}{\to} V_b$ are:
\begin{align*}
\rank(v) &= \langle \qon{a}\qem\qon{b} \mid \Vv \rangle
\\
\nullity(v) &= \langle \qon{a}\qem\qoff{b} \mid \Vv \rangle
\\
\conullity(v) &= \langle \qoff{a}\qem\qon{b} \mid \Vv \rangle
\end{align*}
(The conullity is the dimension of the cokernel.)
\end{example}

\begin{proposition}[direct sums]
\label{prop:m-additivity}
Suppose a persistence module $\Vv$ can be written as a direct sum
\[
\Vv = \bigoplus_{\ell \in L} \Vv^\ell
\]
Then
\[
\langle [a_i, a_j] \mid \Vv_\Tt \rangle
=
\sum_{\ell \in L} \langle [a_i, a_j] \mid \Vv_\Tt^\ell \rangle
\]
for any index set $\Tt = \{a_1, a_2, \dots, a_n\}$ and interval $[a_i,a_j] \subseteq \Tt$.
\end{proposition}

\begin{proof}
Each summand $\Vv^\ell_\Tt$ can be decomposed separately into interval modules. Putting these together we get an interval decomposition of $\Vv_\Tt$. The number of summands of a given type in $\Vv_\Tt$ is then equal to the total number of summands of that type in all of the~$\Vv^\ell_\Tt$.
\end{proof}

Often we wish to compare multiplicities of intervals in different finite restrictions of $\Vv$. The principle is very simple:

\begin{proposition}[restriction principle]
\label{prop:restriction}
Let $\Ss, \Tt$ be finite index sets with $\Ss \subset \Tt$. Then
\[
\langle \Ii \mid \Vv_\Ss \rangle
=
\sum_\Jj \langle \Jj \mid \Vv_\Tt \rangle
\]
where the sum is over those intervals $\Jj \subseteq \Tt$ which restrict over $\Ss$ to $\Ii$.
\end{proposition}

\begin{proof}
Take an arbitrary interval decomposition of $\Vv_\Tt$. This induces an interval decomposition of $\Vv_\Ss$. Summands of $\Vv_\Ss$ of type $\Ii$ arise precisely from those summands of $\Vv_\Tt$ of types $\Jj$ as above.
\end{proof}

\begin{example}
\label{ex:restriction}
Let $a < b < c$. Consider new indices $p$ and $q$, arranged $a < p < b < q < c$.
Then
\[
\langle \qoff{a}\qem\qno\qem\qon{b}\qem\qno\qem\qon{c} \rangle 
=
\langle \qoff{a}\qem\qno\qem\qon{b}\qem\qon{q}\qem\qon{c} \rangle
\mathbin{\phantom{+}}
\phantom{%
\langle \qoff{a}\qem\qon{p}\qem\qon{b}\qem\qno\qem\qon{c} \rangle.}
\]
whereas
\[
\langle \qoff{a}\qem\qno\qem\qon{b}\qem\qno\qem\qon{c} \rangle 
=
\langle \qoff{a}\qem\qoff{p}\qem\qon{b}\qem\qno\qem\qon{c} \rangle 
+
\langle \qoff{a}\qem\qon{p}\qem\qon{b}\qem\qno\qem\qon{c} \rangle.
\]
The extra term occurs when the inserted new index occurs between a clear node and a filled node, because then there are two possible intervals which restrict to the original interval.
\end{example}

We will make frequent use of the restriction principle. Here is a simple illustration, to serve as a template for similar arguments that we will encounter later on.

\begin{example}
\label{ex:monotonerank}
Consider the standard fact that $\rank(V_b \to V_c) \geq \rank(V_a \to V_d)$ when $a \leq b \leq c \leq d$. The proof using quiver notation runs as follows:
\begin{align*}
\rank(V_b \to V_c)
&= \langle \qno\qem\qon{b}\qem\qon{c}\qem\qno \rangle
\\
&= \langle \qon{a}\qem\qon{b}\qem\qon{c}\qem\qon{d} \rangle
	+ \text{three other terms}
\\
&\geq \langle \qon{a}\qem\qon{b}\qem\qon{c}\qem\qon{d} \rangle
\\
&= \langle \qon{a}\qem\qno\qem\qno\qem\qon{d} \rangle
\\
&= \rank(V_a \to V_d)
\end{align*}
The `three other terms' are
\[
\langle \qoff{a}\qem\qon{b}\qem\qon{c}\qem\qon{d} \rangle,
\quad
\langle \qon{a}\qem\qon{b}\qem\qon{c}\qem\qoff{d} \rangle,
\quad
\langle \qoff{a}\qem\qon{b}\qem\qon{c}\qem\qoff{d} \rangle
\]
as indicated by the restriction principle.
\end{example}

%-------------------------------------------------------------------
\section{Rectangle Measures}
\label{sec:measures}

For a decomposable $\Rr$-persistence module
\[
\Vv \cong \bigoplus_{\ell \in L} \Ii{\lgroup p_\ell^*, q_\ell^* \rgroup},
\]
we have defined the decorated persistence diagram to be the multiset
\[
\Dgm(\Vv) = {\{} (p_\ell^*, q_\ell^*) \mid \ell \in L {\}},
\]
and the undecorated persistence diagram to be the multiset
\[
\dgm(\Vv) = {\{} (p_\ell, q_\ell) \mid \ell \in L {\}}.
\]

If we don't know that $\Vv$ is decomposable, then we have to proceed differently. The rough idea is that if we know how many points of~$\Dgm$ are contained in each rectangle in the upper-half space, then we know $\Dgm$ itself. For persistence modules, counting points in rectangles turns out to be easy.

The language of measure theory is well suited to this argument. We will show that each persistence module defines an integer-valued measure on rectangles. If the module is tame then this measure is finite-valued, and therefore (Theorem~\ref{thm:equivalence}) it is concentrated at a discrete set of points: this is the persistence diagram.

In the decomposable case, we will see that this agrees with the definition above. And when the module is not known to be decomposable, we can proceed regardless.

The persistence measures that we construct are not true measures on subsets of~$\Rr^2$: they are additive in the sense of tiling rather than in the usual sense of disjoint set union. The discrepancy arises when we split a rectangle into two: what happens to the points on the common edge? To which rectangle do they belong? In resolving this, one is naturally led to the notion of decorated points. This fits perfectly with our earlier use of decorations to distinguish open and closed interval ends.

%--------------------------------------------------
\subsection{The persistence measure}

Let $\Vv$ be a persistence module. The {persistence measure} of~$\Vv$ is the function \[
\mu_\Vv(R)
  = \langle \qoff{a}\qem\qon{b}\qem\qon{c}\qem\qoff{d} \mid \Vv \rangle
\]
defined on rectangles $R = [a,b] \times [c,d]$ in the plane with $a < b \leq c < d$.

For a decomposable persistence module, there is a clear relationship between $\mu_\Vv$  and the interval summands of~$\Vv$.
Let us first consider the case of an interval module.

\begin{proposition}
\label{prop:Rmembership}
Let $\Vv = \Ii^J$ where $J = \lgroup p^*, q^* \rgroup$ is a real interval. Let $R = [a,b] \times [c,d]$ where $a <  b \leq c < d$. Then
\[
\mu_V(R)
=
\left\{
\begin{array}{ll}
1 \quad&\text{\rm if $[b,c] \subseteq J \subseteq (a,d)$}
\\
0 & \text{\rm otherwise}
\end{array}
\right.
\]
\end{proposition}

\begin{proof}
It is clear that $\Ii^J$ restricted to $\{a,b,c,d\}$ is an interval or is zero. 
Thus, $\mu_\Vv(R) \leq 1$. Moreover $\mu_\Vv(R) = 1$ precisely when
\[
\Ii^J_{a,b,c,d} = \qoff{a}\qem\qon{b}\qem\qon{c}\qem\qoff{d},
\]
which happens if and only if $b,c \in J$ and $a,d \not\in J$. This is equivalent to the condition $[b,c] \subseteq J \subseteq (a,d)$.
\end{proof}

Proposition~\ref{prop:Rmembership} has a graphical interpretation. Represent the interval $J \subseteq \Rr$ as a decorated point in the extended plane. The following picture indicates exactly which decorated points $\lgroup p^*,q^* \rgroup$ are detected by $\mu_\Vv(R)$:
\begin{center}
\includegraphics[scale=0.75]{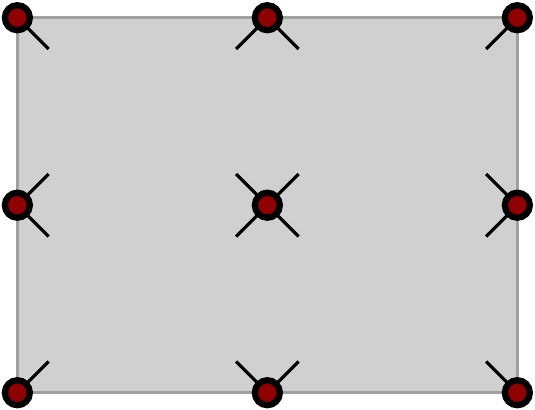}
\end{center}
If $(p,q)$ is in the interior of~$R$ then $\lgroup p^*,q^* \rgroup$ is always detected regardless of the decoration. If $(p,q)$ is on the boundary, then $\lgroup p^*,q^* \rgroup$ is detected if the tick is directed inwards.

We formalise this by defining a membership relation between decorated points and rectangles.

\begin{definition*}
Let $R = [a,b] \times [c,d]$ where $a < b \leq c < d$, and consider a decorated point $(p^*, q^*)$ with $p^* < q^*$.
We write $(p^*, q^*) \in R$ if 

\quad (i) the interval~$J = \lgroup p^*, q^* \rgroup$ satisfies $[b,c] \subseteq J \subseteq (a,d)$;

or equivalently

\quad (ii) the point $(p, q)$ and its decoration tick are contained in the closed rectangle~$R$.
\end{definition*}

We use the notation
\[
R\rintsup = \{ (p^*, q^*) \in R \}
\]
when we wish to explicitly name the set of decorated points contained in~$R$. However, we prefer to write $(p^*, q^*) \in R$ instead of $(p^*, q^*) \in R\rintsup$ when we are simply indicating the relationship between a decorated point and a rectangle.

With this understanding we state the following counting result.

\begin{corollary}
\label{cor:Rcount}
Suppose $\Vv$ is a decomposable persistence module over~$\Rr$:
\[
\Vv = \bigoplus_{\ell \in L} \Ii{\lgroup p_\ell^*, q_\ell^* \rgroup}
\]
Then:
\[
\mu_\Vv(R)
  = \card \left( \Dgm(\Vv)|_R \right)
\tag{*}\label{eq:R-count}
\]
\end{corollary}

\begin{proof}
This follows immediately from Proposition~\ref{prop:Rmembership} and Proposition~\ref{prop:m-additivity} (direct sums).
\end{proof}

We can now articulate our strategy for defining the persistence diagram without assuming that the module~$\Vv$ is decomposable:
\begin{vlist}
\item
construct the persistence measure $\mu_\Vv$;

\item
let $\Dgm(\Vv)$ be a multiset in the half-plane such that \eqref{eq:R-count} holds for all rectangles~$R$.
\end{vlist}
To make this work, we need to know that such a multiset exists and is unique. This is the content of Theorem~\ref{thm:equivalence}, under the hypothesis that $\mu_\Vv$ is finite and additive. The result is a sort of `Riesz' representation theorem for measures on rectangles.

When $\Vv$ is decomposable, Corollary~\ref{cor:Rcount} confirms that our new definition agrees with the old.

%--------------------------------------------------
\subsection{The persistence measure (continued)}
\label{subsec:additivity}

We call $\mu_\Vv$ a measure because it is additive with respect to splitting a rectangle into two rectangles.
We prove this shortly. First, we give a new proof of an `alternating sum' formula for $\mu_\Vv(R)$ that appears in~\cite{CohenSteiner_E_H_2007}.

\begin{proposition}
\label{prop:rankformula}
Let $\Vv$ be a persistence module, and let $a < b \leq c < d$. If the spaces $V_a$, $V_b$, $V_c$, $V_d$ are finite-dimensional, or less stringently if $\rk_b^c < \infty$, then
\[
\langle \qoff{a}\qem\qon{b}\qem\qon{c}\qem\qoff{d} \mid \Vv \rangle
=
\rk_b^c - \rk_a^c - \rk_b^d + \rk_a^d.
\]
\end{proposition}
(Here as before $\rk_s^t = \rank(v_s^t : V_s \to V_t)$.)

\begin{proof}
Decompose the 4-term module $\Vv_{a,b,c,d}$ into intervals. The left-hand side counts intervals of type $[b,c]$. By the restriction principle, the four terms on the right-hand side evaluate as follows:
\begin{align*}
\rk_b^c &= 
\makebox[7.5em][c]{%
  $\langle \qoff{a}\qem\qon{b}\qem\qon{c}\qem\qoff{d} \rangle$}
\makebox[2em][c]{+}
\makebox[7.5em][c]{%
  $\langle \qon{a}\qem\qon{b}\qem\qon{c}\qem\qoff{d} \rangle$}
\makebox[2em][c]{+}
\makebox[7.5em][c]{%
  $\langle \qoff{a}\qem\qon{b}\qem\qon{c}\qem\qon{d} \rangle$}
\makebox[2em][c]{+}
\makebox[7.5em][c]{%
  $\langle \qon{a}\qem\qon{b}\qem\qon{c}\qem\qon{d} \rangle$}
\\
\rk_a^c &= 
\makebox[7.5em][c]{}
\makebox[2em][c]{}
\makebox[7.5em][c]{%
  $\langle \qon{a}\qem\qon{b}\qem\qon{c}\qem\qoff{d} \rangle$}
\makebox[2em][c]{}
\makebox[7.5em][c]{}
\makebox[2em][c]{+}
\makebox[7.5em][c]{%
  $\langle \qon{a}\qem\qon{b}\qem\qon{c}\qem\qon{d} \rangle$}
\\
\rk_b^d &= 
\makebox[7.5em][c]{}
\makebox[2em][c]{}
\makebox[7.5em][c]{}
\makebox[2em][c]{}
\makebox[7.5em][c]{%
  $\langle \qoff{a}\qem\qon{b}\qem\qon{c}\qem\qon{d} \rangle$}
\makebox[2em][c]{+}
\makebox[7.5em][c]{%
  $\langle \qon{a}\qem\qon{b}\qem\qon{c}\qem\qon{d} \rangle$}
\\
\rk_a^d &= 
\makebox[7.5em][c]{}
\makebox[2em][c]{}
\makebox[7.5em][c]{}
\makebox[2em][c]{}
\makebox[7.5em][c]{}
\makebox[2em][c]{}
\makebox[7.5em][c]{%
  $\langle \qon{a}\qem\qon{b}\qem\qon{c}\qem\qon{d} \rangle$}
\end{align*}
These expressions are all finite: the hypothesis $\rk_b^c < \infty$ implies that the other three ranks are finite too (Example~\ref{ex:monotonerank}). We can legitimately take the alternating sum, whereupon all terms cancel except for the $\langle \qoff{a}\qem\qon{b}\qem\qon{c}\qem\qoff{d} \rangle$.
\end{proof}

We give three proofs of additivity. The first is completely general, whereas the other two work under restricted settings but are illuminating in their own way.

\begin{proposition}
\label{prop:splitting}
$\mu_\Vv$ is additive under vertical and horizontal splitting, meaning that
\begin{align*}
\mu_\Vv([a,b]\times[c,d])
  &=  \mu_\Vv([a,p]\times[c,d]) + \mu_\Vv([p,b]\times[c,d])
\\
\mu_\Vv([a,b]\times[c,d])
  &= \mu_\Vv([a,b]\times[c,q]) + \mu_\Vv([a,b]\times[q,d])
\end{align*}
whenever $a < p < b \leq c < q < d$.
\end{proposition}

This additivity property is illustrated by the following figure
\begin{center}
\includegraphics[scale=0.75]{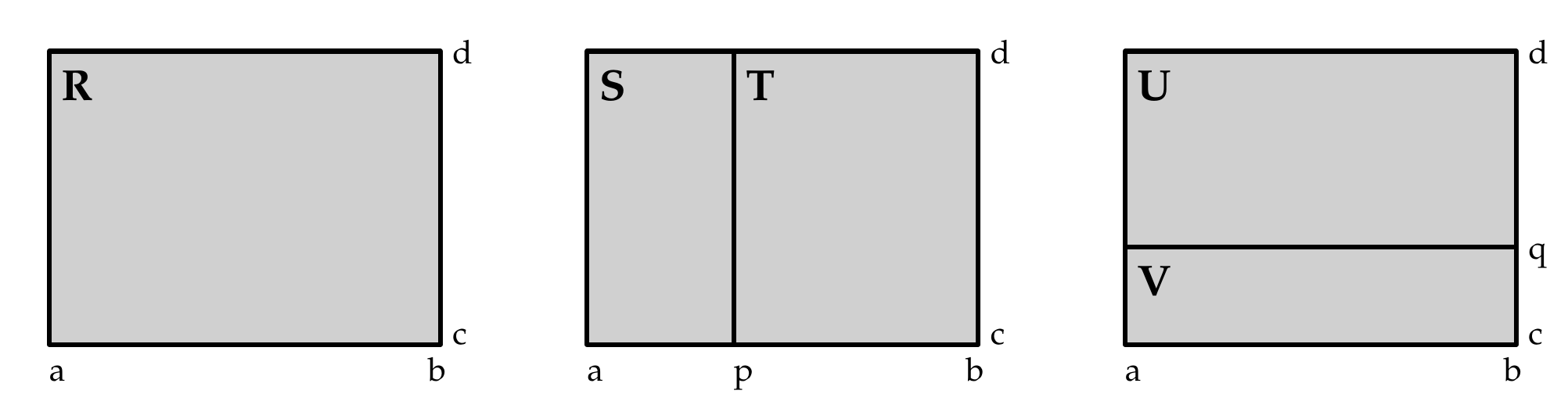}
\end{center}
where the claim is that $\mu_\Vv(R) = \mu_\Vv(S) + \mu_\Vv(T) = \mu_\Vv(U) + \mu_\Vv(V)$.

\begin{proof}[First proof]
Let $a < p < b \leq c < q < d$. Then we calculate
\begin{align*}
\mu_\Vv([a,b] \times [c,d])
&=
\langle\,
  \qoff{a}\qem\qno\qem\qon{b}\qem\qon{c}\qem\qoff{d}
\,\rangle
\\
&=
\langle\,
  \qoff{a}\qem\qon{p}\qem\qon{b}\qem\qon{c}\qem\qoff{d}
\,\rangle
+
\langle\,
  \qoff{a}\qem\qoff{p}\qem\qon{b}\qem\qon{c}\qem\qoff{d}
\,\rangle
\\
&=
\langle\,
  \qoff{a}\qem\qon{p}\qem\qno\qem\qon{c}\qem\qoff{d}
\,\rangle
+
\langle\,
  \qno\qem\qoff{p}\qem\qon{b}\qem\qon{c}\qem\qoff{d}
\,\rangle
\\
&=
\mu_\Vv([a,p]\times[c,d]) + \mu_\Vv([p,b]\times[c,d])
\end{align*}
for additivity with respect to a horizontal split, and
\begin{align*}
\mu_\Vv([a,b] \times [c,d])
&=
\langle\,
  \qoff{a}\qem\qon{b}\qem\qon{c}\qem\qno\qem\qoff{d}
\,\rangle
\\
&=
\langle\,
  \qoff{a}\qem\qon{b}\qem\qon{c}\qem\qoff{q}\qem\qoff{d}
\,\rangle
+
\langle\,
  \qoff{a}\qem\qon{b}\qem\qon{c}\qem\qon{q}\qem\qoff{d}
\,\rangle
\\
&=
\langle\,
  \qoff{a}\qem\qon{b}\qem\qon{c}\qem\qoff{q}\qem\qno
\,\rangle
+
\langle\,
  \qoff{a}\qem\qon{b}\qem\qno\qem\qon{q}\qem\qoff{d}
\,\rangle
\\
&=
\mu_\Vv([a,b] \times [c,q]) + \mu_\Vv([a,b] \times [q,d])
\end{align*}
for additivity with respect to a vertical split.
\end{proof}

\begin{proof}[Second proof, assuming $\rk_b^c < \infty$.]
The alternating sum formula (Proposition~\ref{prop:rankformula}) gives
\[
\rk_b^c - \rk_a^c - \rk_b^d + \rk_a^d
=
(\rk_p^c - \rk_a^c - \rk_p^d + \rk_a^d)
+
(\rk_b^c - \rk_p^c - \rk_b^d + \rk_p^d)
\]
and
\[
\rk_b^c - \rk_a^c - \rk_b^d + \rk_a^d
=
(\rk_b^c - \rk_a^c - \rk_b^q + \rk_a^q)
+
(\rk_b^q - \rk_a^q - \rk_b^d + \rk_a^d)
\]
as required. Note that $\rk_b^c < \infty$ implies that $\rk_p^c, \rk_b^q < \infty$, so the formula is valid for all the rectangles involved.
\end{proof}

This second proof is particularly transparent when drawn geometrically in the plane: the $+$ and $-$ signs at the corners of the rectangles cancel in a pleasant way.

\begin{proof}[Third proof, assuming $\Vv$ is decomposable.]
By Corollary~\ref{cor:Rcount}, the measure of a rectangle is equal to the number of interval summands whose corresponding decorated points lie in the rectangle. Additivity now follows from the elementary observation that a decorated point in~$R$ belongs to exactly one of $S$ and~$T$, and to exactly one of $U$ and~$V$.
\end{proof}

We finish this section with two further descriptions of $\mu_\Vv([a,b]\times[c,d])$.

\begin{proposition}
\label{prop:localize}
We have the following formulae:
\begin{align*}
\langle \qoff{a}\qem\qon{b}\qem\qon{c}\qem\qoff{d} \mid \Vv \rangle
&=
\dim \left[ 
\frac{\img(v_b^c) \cap \ker(v_c^d)}{\img(v_a^c) \cap \ker(v_c^d)}
\right]
\\
&=
\dim \left[ 
\frac{\ker(v_b^d)}{\ker(v_b^c) + \img(v_a^b) \cap \ker(v_b^d)}
\right]
\end{align*}
\end{proposition}

\begin{proof} This is covered, for instance, in the localisation discussion in section~{5.1} of~\cite{Carlsson_deSilva_2010}. The two formulae are obtained by localising at $c,b$ respectively.
\end{proof}

Proposition~\ref{prop:localize} expresses the measure of a rectangle as the dimension of a vector space constructed functorially from~$\Vv$. (Ostensibly there are two vector spaces, one for each formula, but the map $v_b^c$ induces a natural isomorphism between them.)
The functoriality has its uses, but in other regards this characterisation is quite hard to use. For instance, additivity is not at all obvious in this formulation.

%--------------------------------------------------
\subsection{Abstract r-measures}
\label{subsec:abstract-R}

We now consider rectangle measures more abstractly. Persistence measures are of course our primary example, but the general formulation allows for many other situations.

For ease of exposition, we initially work in the plane $\Rr^2$ rather than the extended plane $\RR^2$. The picture is completed in section~\ref{subsec:infinity} when we discuss the points at infinity.

\begin{definition*}
Let $\Dd$ be a subset of $\Rr^2$. Define
\[
\rect(\Dd) =
\{ [a,b]\times[c,d] \subset \Dd \mid a < b \;\text{and}\; c < d \}
\]
(the set of closed rectangles contained in~$\Dd$).
A {\bf rectangle measure} or {\bf r-measure} on $\Dd$ is a function
\[
\mu : \rect(\Dd) \to \left\{0, 1, 2, \dots \right\} \cup \{\infty\}
\]
which is additive under vertical and horizontal splitting (as in Proposition~\ref{prop:splitting}).
\end{definition*}

\begin{proposition}
\label{prop:mu-properties}
Let $\mu$ be an r-measure on $\Dd \subseteq \Rr^2$. Then $\mu$ is:

(Finitely additive)
If $R \in \rect(\Dd)$ can be written as a union $R = R_1 \cup \dots \cup R_k$ of rectangles with disjoint interiors, then
$\mu(R) = \mu(R_1) + \dots + \mu(R_k)$.

(Monotone)
If $R \subseteq S$ then $\mu(R) \leq \mu(S)$.
\end{proposition}

\begin{proof} (Finitely additive)
Let $R = [a,b] \times [c,d]$.
By induction and the vertical splitting property, it follows that finite additivity holds for decompositions of the form
\[
R = \bigcup_i  R_i
\]
where $R_i = [a_i,a_{i+1}] \times [c,d]$ with $a = a_1 < a_2 < \dots < a_m = b$.

By induction and the horizontal splitting property, it then follows that finite additivity holds for `product' decompositions
\[
R = [a,b] \times [c,d] = \bigcup_{i,j} R_{ij}
\]
where $R_{ij} = [a_i, a_{i+1}] \times [c_j, c_{j+1}]$ with
$a = a_1 < a_2 < \dots < a_m = b$ and $c = c_1 < c_2 < \dots < c_n = d$.

For an arbitrary decomposition $R = R_1 \cup \dots \cup R_k$, the result follows by considering a product decomposition of~$R$ by which each $R_i$ is itself product-decomposed.

(Monotone)
Decompose $S$ into a collection of interior disjoint rectangles $R, R_1, \dots, R_{k-1}$, one of which is $R$. (This can be done with at most 9 rectangles using a product decomposition.) Then
\begin{eqnarray*}
\mu(S)
&=&  \mu(R) + \mu(R_1) + \dots + \mu(R_{k-1})
\\
&\geq& \mu(R)
\end{eqnarray*}
by finite additivity and the fact that $\mu \geq 0$.
\end{proof}

Here is one more plausible-and-true statement about abstract r-measures.

\begin{proposition}[Subadditivity]
\label{prop:subadditive}
Let $\mu$ be an r-measure on $\Dd \subseteq \Rr^2$. If a rectangle $R \in \rect(\Dd)$ is contained in a finite union
\[
R \subseteq R_1 \cup \dots \cup R_k
\]
of rectangles $R_i \in \rect(\Dd)$, then
\[
\mu(R) \leq \mu(R_1) + \dots + \mu(R_k).
\]
\end{proposition}

\begin{proof} Let
\[
a_1 < a_2 < \dots < a_m
\]
include all the $x$-coordinates of the corners of all the rectangles, and let
\[
c_1 < c_2 < \dots < c_n
\]
include all the $y$-coordinates. Each rectangle is then tiled as a union of pieces
\[
[a_i,a_{i+1}] \times [c_j,c_{j+1}]
\]
with disjoint interiors, and the measure of the rectangle is the sum of the measures of its tiles, by additivity. 
Since each tile belonging to~$R$ must also belong to one or more of the~$R_i$, the inequality follows.
\end{proof}

%--------------------------------------------------
\subsection{Equivalence of measures and diagrams}
\label{subsec:equivalence}
We wish to establish a correspondence between r-measures and decorated  diagrams. The task of defining a continuous persistence diagram can then be replaced by the simpler task of defining an r-measure. For this to work, the measure has to be finite.

A useful notion is the {\bf r-interior} of a region $\Dd \subseteq \Rr^2$, defined
\[
\drint = 
\left\{
(p^*,q^*) \mid
\text{there exists $R \in \rect(\Dd)$ such that $(p^*,q^*) \in R$}
\right\}.
\]
This is the set of decorated points that are `seen' by the rectangles in~$\Dd$. The decorated diagram will be a multiset in $\drint$. Clearly an r-measure cannot tell us what happens outside~$\drint$.

The {\bf interior} of~$\Dd$ in the classical sense is written $\dint$. It may be defined
\[
\dint =
\left\{
(p,q) \mid
\text{there exists $R \in \rect(\Dd)$ such that $(p,q) \in R\intsup$}
\right\},
\]
supposing that we already agree that the interior of a closed rectangle $R = [a,b]\times[c,d]$ is the open rectangle $R\intsup = (a,b)\times(c,d)$.
%
%An equivalent definition is
%%
%\[
%\dint =
%\left\{
%(p,q) \mid
%\text{all four decorated points $(p^*, q^*)$ belong to $\drint$}
%\right\}.
%\]
%%
The undecorated diagram will be a multiset in~$\dint$.

\begin{theorem}[The equivalence theorem]
\label{thm:equivalence}
Let $\Dd \subseteq \Rr^2$.
There is a bijective correspondence between:
\begin{vlist}
\item
Finite r-measures $\mu$ on~$\Dd$. Here `finite' means that $\mu(R) < \infty$ for every $R \in \rect(\Dd)$.

\item
Locally finite multisets $\Aa$ in~$\drint$. Here `locally finite' means that $\card(\Aa|_R) < \infty$ for every $R \in \rect(\Dd)$.
\end{vlist}
The measure $\mu$ corresponding to a multiset~$\Aa$ is related to it by the formula
\begin{equation} \label{eq:sum1}
\mu(R) = \card(\Aa|_R)
\end{equation}
for every $R \in \rect(\Dd)$.
\end{theorem}

\begin{remark}
We can write equation~\eqref{eq:sum1} equivalently as
\begin{equation}
\label{eq:sum2}
\mu(R) = \sum_{(p^*,q^*) \in R} \mult(p^*,q^*),
\end{equation}
where
\[
\mult: \drint \to \{0, 1, 2, \dots \}
\]
is the multiplicity function for~$\Aa$.
\end{remark}

The theorem leads immediately to the following definitions. Let $\mu$ be a finite r-measure on a region~$\Dd \subset \Rr^2$.

(i)
The {\bf decorated diagram} of~$\mu$ is the unique locally finite multiset $\Dgm(\mu)$ in~$\drint$ such that
\[
\mu(R) = \card(\Dgm(\mu)|_R)
\]
for every $R \in \rect(\Dd)$.

(ii)
The {\bf undecorated diagram} of~$\mu$ is the locally finite multiset in $\dint$
\[
\dgm(\mu) = \left\{ (p,q) \mid (p^*, q^*) \in \Dgm(\mu) \right\}
\cap \dint
\]
obtained by forgetting the decorations on the points and restricting to the interior.

\begin{remark}
Note that $\dgm$ is locally finite in~$\dint$, but not necessarily
%
%because any point in~$\dint$ is contained in the interior of a rectangle $R \subset \dint$, and this interior contains at most $\mu(R)$ points of $\dgm$. However, $\dgm$ need not be 
%
locally finite in~$\Rr^2$---it may have accumulation points on the boundary of~$\Dd$.
\end{remark}

\begin{proof}[Proof of Theorem~\ref{thm:equivalence}]
One direction of the correspondence is easy. If $\Aa$ is a multiset on $\drint$ then the function $\mu(R)$ on rectangles defined by equation~\eqref{eq:sum1} is indeed an r-measure. It is finite if $\Aa$ is locally finite. To verify additivity, suppose that a rectangle $R$ is split vertically or horizontally into two rectangles $R_1, R_2$. Notice that every decorated point $(p^*,q^*) \in R$ belongs to exactly one of $R_1, R_2$. It follows that
\[
\mu(R)
= \card(\Aa|_R)
= \card(\Aa|_{R_1}) + \card(\Aa|_{R_2})
= \mu(R_1) + \mu(R_2),
\]
as required.

The reverse direction takes more work. Given an r-measure $\mu$ we will {(1)}~construct a multiset $\Aa$ in~$\drint$, (2)~show that $\mu$ and $\Aa$ are related by equation~\eqref{eq:sum1}, and (3)~show that $\Aa$ is unique.

In practice we will work with the multiplicity function $\mult$ and equation~\eqref{eq:sum2}, rather than referring to $\Aa$ directly.

{\bf\small Step 1.}
Let $\mu$ be a finite r-measure on $\Dd$. For $(p^*,q^*)$ in $\drint$, define
\begin{equation}
\mult(p^*,q^*) =
\min \left\{ \mu(R) \mid R \in \rect(\Dd),\, (p^*, q^*) \in R \right\}.
\label{def:mult}
\end{equation}
Note that the minimum is attained because the set is nonempty and $\mu$ takes values in the natural numbers.

Here is an alternative characterisation. Rather than minimising over all rectangles, we can take the limit through a decreasing sequence of rectangles:

\begin{lemma}
\label{lemma:mult-2}
Let $(\xi_i)$ and $(\eta_i)$ be non-increasing sequences of positive real numbers which tend to zero as $i \to \infty$. Then
\[
\mult(p^+,q^+) =
\lim_{i \to \infty} \mu([p,p+\xi_i] \times [q,q+\eta_i]),
\]
and similarly
\begin{align*}
\mult(p^+,q^-) &=
\lim_{i \to \infty} \mu([p,p+\xi_i] \times [q-\eta_i,q]),
\\
\mult(p^-,q^+) &=
\lim_{i \to \infty} \mu([p-\xi_i,p] \times [q,q+\eta_i]),
\\
\mult(p^-,q^-) &=
\lim_{i \to \infty} \mu([p-\xi_i,p] \times [q-\eta_i,q]).
\end{align*}
\end{lemma}

\begin{proof}
The key observation is that the sequence of rectangles $R_i = [p,p+\xi_i] \times [q,q+\eta_i]$ is {cofinal} in the set of rectangles $R$ containing $(p^+,q^+)$. In other words, for any such $R$ we have $R_i \subseteq R$ for all sufficiently large~$i$.

By monotonicity, the sequence of nonnegative integers $\mu(R_i)$ is non-increasing, and hence eventually stabilises to a limit. Then
\[
\mult(p^+, q^+)
\leq \min_i \mu(R_i)
= \lim_{i\to\infty} \mu(R_i)
\leq \mu(R)
\]
for any $R$ containing $(p^+, q^+)$. Taking the minimum over all~$R$, the right-hand side becomes $\mult(p^+,q^+)$ and hence by squeezing
\[
\mult(p^+,q^+) = \lim_{i\to\infty} \mu(R_i).
\]
The other three cases of the lemma are similar.
\end{proof}

We return to the main proof.

{\bf\small Step 2.} Having defined $\mult(p^*,q^*)$, we now show that this is the `correct' definition, meaning that equation~\eqref{eq:sum2} is satisfied. We have seen already that $\mult$ corresponds to an r-measure
\begin{equation}
\nu(R) = \sum_{(p^*,q^*) \in R} \mult(p^*,q^*),
\end{equation}
and it remains to show (for this step) that $\nu = \mu$.
We prove this by induction on $k = \mu(R)$.

{\bf\small Base case.} $\mu(R) = 0$. Then for every $(p^*,q^*) \in R$ we have
\[
0 \leq \mult(p^*,q^*) \leq \mu(R) = 0
\]
so $\nu(R) = 0$. 

{\bf\small Inductive step.} Suppose $\mu(R) = \nu(R)$ for every rectangle~$R$ with $\mu(R) < k$. Consider a rectangle $R_0$ with $\mu(R_0) = k$. We must show that $\nu(R_0) = k$.

Split the rectangle into four equal quadrants $S_1, S_2, S_3, S_4$. Certainly
\begin{align*}
\mu(R_0) &= \mu(S_1) + \mu(S_2) + \mu(S_3) + \mu(S_4)
\\ 
\nu(R_0) &= \nu(S_1) + \nu(S_2) + \nu(S_3) + \nu(S_4)
\end{align*}
by finite additivity (Proposition~\ref{prop:mu-properties}). If every quadrant satisfies $\mu(S_i) < k$, then by induction we deduce that $\mu(R_0) = \nu(R_0)$.
Otherwise, one of the quadrants has $\mu = k$ and the other three quadrants satisfy $\mu = 0$ (and hence $\nu = 0$). Let $R_1$ be the distinguished quadrant, so $\mu(R_1) = k$. It is now enough to show that $\nu(R_1) = k$.

We repeat the argument. Subdivide $R_i$ into four equal quadrants. Either all four quadrants satisfy the inductive hypothesis $\mu < k$, in which case we are done. Otherwise we find a quadrant $R_{i+1}$ with $\mu(R_{i+1}) = k$, and we are reduced to showing that $\nu(R_{i+1}) = k$.

In the worst case---the remaining unresolved case---this iteration never terminates and we obtain a sequence of closed rectangles
\[
R_0 \supset R_1 \supset R_2 \supset \dots
\]
each being a quadrant of the previous one, with $\mu(R_i) = k$.
Since the diameters of the rectangles tend to zero, their intersection $\bigcap_i R_i$ contains a single point $(r,s)$.

We are now in a position to show that $\nu(R_0)=k$, by evaluating the sum explicitly over all decorated points in~$R_0$.

First of all, consider decorated points that eventually leave the sequence $(R_i)$. Specifically, suppose that $(p^*,q^*) \in R_0$ but $(p^*,q^*) \in R_{i-1} - R_i$ for some~$i$. This means that $(p^*,q^*)$ belongs to one of the three quadrants of $R_{i-1}$ for which $\mu=0$. It follows immediately that $\mult(p^*,q^*) = 0$.

Thus, the only contribution to $\nu(R_0)$ comes from decorated points $(p^*,q^*)$ which belong to every rectangle in the sequence $(R_i)$. Clearly these must be decorated versions $(r^*,s^*)$ of the intersection point $(r,s)$. There are 4, 2 or 1 of them depending on how the nested sequence of rectangles converges to its limit. Here we illustrate the three cases:

\smallskip
\centerline{
\includegraphics[scale=0.5]{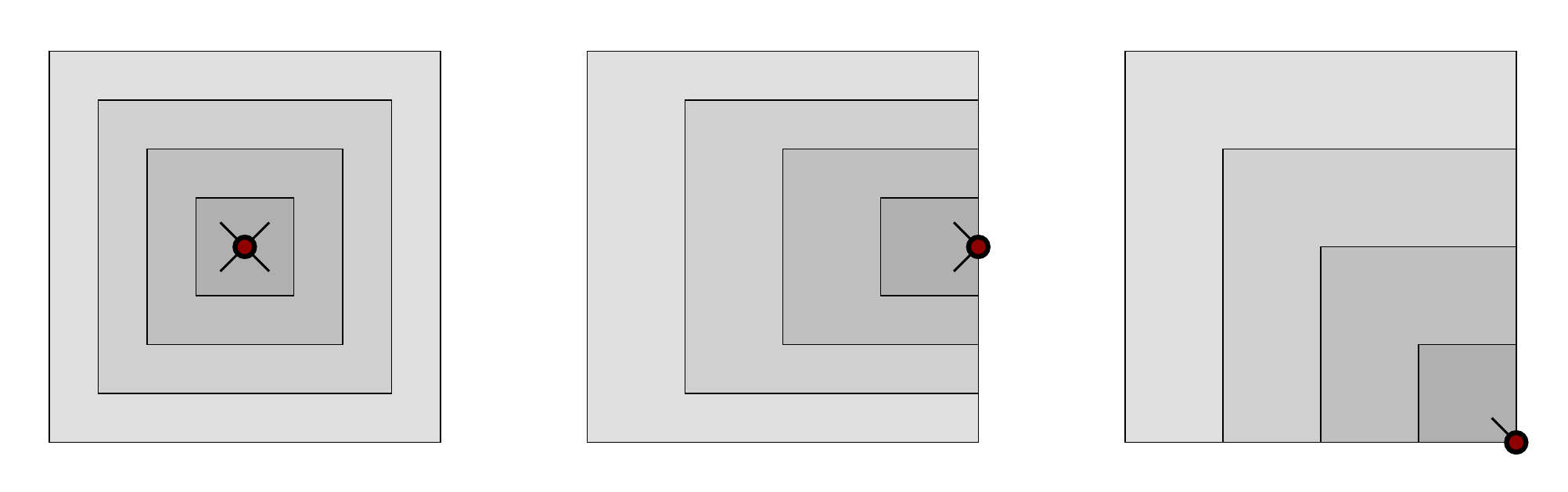}
}

Suppose first that $(r,s)$ lies in the interior of every rectangle $R_i$, so that all four decorated points $(r^+,s^+)$, $(r^+,s^-)$, $(r^-,s^+)$, $(r^-,s^-)$ belong to every $R_i$. Divide each $R_i$ into 4 subrectangles $R_i^{++}$, $R_i^{+-}$, $R_i^{-+}$, $R_i^{--}$, which share a common corner at $(r,s)$ so that each of the four decorated points $(r^*,s^*)$ belongs to one of the subrectangles in the obvious notation. By Lemma~\ref{lemma:mult-2},
\begin{alignat*}{2}
\mult(r^+,s^+) &= \lim_{i\to\infty} \mu(R_i^{++}),
\qquad
&\mult(r^+,s^-) &= \lim_{i\to\infty} \mu(R_i^{+-}),
\\
\mult(r^-,s^+) &= \lim_{i\to\infty} \mu(R_i^{-+}),
&\mult(r^-,s^-) &= \lim_{i\to\infty} \mu(R_i^{--}),
\end{alignat*}
and moreover each of these decreasing integer sequences eventually stabilises at its limiting value.
Thus, for sufficiently large~$i$,
\begin{alignat*}{6}
\nu(R_0)
  &= \mult(r^+,s^+) &&+ \mult(r^+,s^-) &&+ \mult(r^-,s^+) &&+ \mult(r^-,s^-) 
\\
  &= \mu(R_i^{++}) &&+ \mu(R_i^{+-}) &&+ \mu(R_i^{-+}) &&+ \mu(R_i^{--})
%\\
  &&= \mu(R_i)
%\\
  &&= k
\end{alignat*}
as required.

A similar argument (with fewer terms) can be made in the cases where only 2 or 1 of the decorated points $(r^*,s^*)$ belong to every $R_i$.
For instance, if $(r,s)$ lies on the interior of the right-hand edge of the rectangles $(R_i)$ for all sufficiently large~$i$, we split each rectangle into two parts $R_i^{-+}$ and $R_i^{--}$ and obtain
\[
\nu(R_0)
= \mult(r^-,s^+) + \mult(r^-,s^-) 
= \mu(R_i^{-+}) + \mu(R_i^{--})
= \mu(R_i)
= k
\]
in the same way. In this case $(r^+,s^+)$ and $(r^+,s^-)$ eventually leave (or were never in) the sequence $(R_i)$ and therefore do not contribute to $\nu(R_0)$. We omit the details of the remaining cases, which are equally straightforward.

This completes the inductive step. Thus $\mu(R) = \nu(R)$ for every $R \in \rect(\Dd)$.
%
%This confirms the existence of a multiset $\Dgm(\mu)$ for which \eqref{eq:sum1} holds.
%

{\small\bf Step 3.}
Suppose $\mult'(p^*, q^*)$ is some other multiplicity function on~$\drint$
whose associated r-measure
\[
\nu'(R) = \sum_{(p^*,q^*) \in R} \mult'(p^*,q^*)
\]
satisfies $\mu = \nu'$. We must show that $\mult = \mult'$.

Consider an arbitrary decorated point $(p^*,q^*) \in \drint$. Let $R$ be a rectangle which contains $(p^*,q^*)$ at its corner. Since
\[
\nu(R) = \nu'(R) = \mu(R) < \infty,
\]
there are only finitely many other decorated points $(r^*,s^*) \in R$ with positive multiplicity in $\mult$ or~$\mult'$. By making~$R$ smaller, we can therefore assume that $(p^*, q^*)$ is the only decorated point in~$R$ with positive multiplicity in either measure.
Then
\[
\mult(p^*,q^*) = \nu(R) = \mu(R) = \nu'(R) = \mult'(p^*,q^*).
\]
Since $(p^*,q^*)$ was arbitrary it follows that $\mult = \mult'$.

This completes the proof of Theorem~\ref{thm:equivalence}.
\end{proof}

%--------------------------------------------------
\subsection{Non-finite measures}
\label{subsec:non-finite}

If a measure is not everywhere finite, we restrict our attention to the parts of the plane where it is finite. Define the {\bf finite r-interior} of an r-measure~$\mu$ to be the set of decorated points
\[
\frint(\mu) = 
\left\{
(p^*,q^*) \mid
\text{there exists $R \in \rect(\Dd)$ such that $(p^*,q^*) \in R$ and $\mu(R) < \infty$}
\right\}.
\]
The {\bf  finite interior} is
\[
\fint(\mu) =
\left\{
(p,q) \mid
\text{there exists $R \in \rect(\Dd)$ such that $(p,q) \in R\intsup$ and $\mu(R) < \infty$}
\right\}.
\]
%\[
%\fint(\mu) =
%\left\{
%(p,q) \mid
%\text{all four decorated points $(p^*, q^*)$ belong to $\frint(\mu)$}
%\right\}.
%\]
This is an open subset of the plane, being a union of open rectangles. It is easy to see that $(p,q) \in \fint(\mu)$ if and only if $(p^*,q^*) \in \frint(\mu)$ for all possible decorations.

We can apply Theorem~\ref{thm:equivalence} to each rectangle~$R$ of finite measure to obtain a decorated diagram in~$R\rintsup$. These diagrams will agree, by uniqueness, on the common intersection of any two such rectangles. Therefore we can combine these local definitions to get a well-defined multiset $\Dgm(\mu)$ in the whole of $\frint(\mu)$.

This multiset has the property that $\mu(R) = \card(\Dgm(\mu) |_R)$ for any rectangle $R \in \rect(\Dd)$ with finite measure.
There is one subtlety to address.

\begin{proposition}
\label{prop:frint-subtle}
Let $R \in \rect(\Dd)$. If $R\rintsup \subseteq \frint(\mu)$ then $\mu(R) < \infty$. 
\end{proposition}

%In other words, if every decorated point in~$R$ is contained in some finite-measure rectangle, then $R$ itself has finite measure.

\begin{proof}
We show that each $(p,q) \in R$ is contained in the relative interior of a rectangle $S \subseteq R$ of finite measure. Then $R$, being compact, is the union of finitely many of these rectangles, and by subadditivity (Proposition~\ref{prop:subadditive}) it must have finite measure.

If $(p,q)$ lies in the interior of~$R$, then each of the four decorated points $(p^*,q^*)$ belongs to $\frint(\mu)$ so we can find four finite-measure rectangles containing them. The union of these rectangles contains a neighbourhood of $(p,q)$, and we can take $S \subseteq R$ to be a rectangle contained in this union with $(p,q)$ in its interior. It has finite measure, by subadditivity.

If $(p,q)$ lies on the interior of an edge, we take two finite-measure rectangles containing a relative neighbourhood of~$(p,q)$; and if $(p,q)$ is a corner point we take just one rectangle.
\end{proof}

\begin{corollary}
\label{cor:f-diagram}
Let $\mu$ be an r-measure on $\Dd \subseteq \Rr^2$.
Then there is a uniquely defined locally finite multiset $\Dgm(\mu)$ in $\frint(\mu)$ such that
\[
\mu(R) = \card(\Dgm(\mu)|_R)
\]
for every $R \in \rect(\Dd)$ with $R\rintsup \subseteq \frint(\mu)$.
\qed
\end{corollary}

Similarly, or as an easy consequence of Proposition~\ref{prop:frint-subtle}, we have:

\begin{proposition}
\label{prop:fint-subtle}
Let $R \in \rect(\Dd)$. If $R \subseteq \fint(\mu)$ then $\mu(R) < \infty$.
\qed
\end{proposition}

Now we can define the diagrams of a general r-measure.

\begin{vlist}
\item
The {\bf decorated diagram} of an r-measure~$\mu$ is the pair $(\Dgm(\mu), \frint(\mu))$, where $\Dgm(\mu)$ is the multiset in~$\frint(\mu)$ described in the corollary.

\item
The {\bf undecorated diagram} is the pair $(\dgm(\mu), \fint(\mu))$, where
\[
\dgm(\mu) = \left\{ (p,q) \mid (p^*, q^*) \in \Dgm(\mu) \right\}
\cap \fint(\mu)
\]
is the locally finite%
\footnote{%
As before, $\dgm$ is locally finite in $\fint$, but may have accumulation points on the boundary of~$\fint$.
}
multiset in $\fint(\mu)$ obtained by forgetting the decorations in $\Dgm(\mu)$ and restricting to the finite interior.
\end{vlist}

\begin{remark}
The point of Proposition~\ref{prop:frint-subtle} is that there is no `hidden information' in~$\mu$ beyond what is recorded in the diagram. The measure of any rectangle~$R$ with $R\rintsup \subseteq \frint(\mu)$ is recovered by counting decorated points of $\Dgm(\mu)$, and any other rectangle has infinite measure.
\end{remark}

We can bring these definitions into accord with the previous one (for finite r-measures) by agreeing that $\Dgm(\mu)$ and $\dgm(\mu)$ are abbreviations for $(\Dgm(\mu), \drint)$ and $(\dgm(\mu), \dint)$, when $\Dd$ is clear from the context.

It is sometimes useful to adopt the following {\bf extension convention}. An r-measure defined on a subset $\Dd \subset \Rr^2$ can be interpreted as an r-measure on the whole plane~$\Rr^2$, by agreeing that $\mu(R) = \infty$ for any rectangle that meets $\Rr^2 - \Dd$. The extension has the same diagram as the original r-measure.
%

%--------------------------------------------------
\subsection{The diagram at infinity}
\label{subsec:infinity}

We now discuss r-measures in the extended plane~$\RR^2$. This requires infinite rectangles and decorated points at infinity.

\begin{vlist}
\item
For $\Dd \subseteq \RR^2$, the set $\rect(\Dd)$ consists of all rectangles $R \subseteq \Dd$ of the form
\[
R = [a,b] \times [c,d]
\]
where now $-\infty \leq a < b \leq +\infty$ and $-\infty \leq c < d \leq +\infty$. 

\item
Decorated points in~$\RR^2$ are pairs $(p^*,q^*)$ where $p^*$ and $q^*$ are decorated real numbers or $-\infty^+$ or $+\infty^-$. The symbols $-\infty^+, +\infty^-$ may be  abbreviated to $-\infty, +\infty$.

\end{vlist}

With these preliminaries, several other concepts are formally unchanged. The r-interior and the interior of~$\Dd$ are
\begin{align*}
\drint
&= 
\left\{ (p^*,q^*) \mid
\text{there exists $R \in \rect(\Dd)$ such that $(p^*,q^*) \in R$} \right\},
\\
\dint
&=
\left\{ (p,q) \mid
\text{there exists $R \in \rect(\Dd)$ such that $(p,q) \in R\intsup$} \right\}.
\end{align*}
One subtlety is that we take $R\intsup$ to be the relative interior of $R$ as a subspace of~$\RR^2$. For instance, if $R = [-\infty,b]\times[c,d]$ where $b,c,d$ are finite, then $R\intsup = [-\infty,b)\times(c,d)$.

The concept of an r-measure goes through as before: it is a function
\[
\mu : \rect(\Dd) \to \{0, 1, 2, \dots, \} \cup \{ \infty \}
\]
which is additive with respect to horizontal and vertical splittings. An r-measure has a finite r-interior and a finite interior:
\begin{align*}
\frint(\mu)
&= 
\left\{
(p^*,q^*) \mid
\text{there exists $R \in \rect(\Dd)$ such that $(p^*,q^*) \in R$ and $\mu(R) < \infty$}
\right\},
\\
\fint(\mu)
&=
\left\{
(p,q) \mid
\text{there exists $R \in \rect(\Dd)$ such that $(p,q) \in R\intsup$ and $\mu(R) < \infty$}
\right\}.
\end{align*}

We sometimes write $\mathcal{S}(\mu) = \RR^2 - \fint(\mu)$, the \textbf{singular support} of~$\mu$.

{\bf Claim.}
Corollary~\ref{cor:f-diagram} is valid for r-measures on $\Dd \subseteq \RR^2$.

\begin{proof}
The statements (and indeed the proofs) of Theorem~\ref{thm:equivalence} and
Corollary~\ref{cor:f-diagram} are invariant under reparametrisations of the plane of the form
\[
x' = f(x), \quad y' = g(y),
\]
where $f,g$ are homeomorphic embeddings. We can view $\RR^2$ as a rectangle in~$\Rr^2$ via a transformation of this type; for instance
\[
x'  = \arctan(x), \quad y' = \arctan(y)
\]
identifies $\RR^2$ with the rectangle $[-\pi/2,\pi/2] \times [-\pi/2, \pi/2]$ in~$\Rr^2$. 
The theorem and its corollary now apply directly.
\end{proof}

The decorated diagram $(\Dgm(\mu), \frint(\mu))$ and the undecorated diagram $(\dgm(\mu), \fint(\mu))$ are defined as before.

Let $\Vv$ be a persistence module.
The persistence measure $\mu_\Vv$, previously defined by
\[
\mu_\Vv([a,b]\times[c,d]) =
\langle \qoff{a}\qem\qon{b}\qem\qon{c}\qem\qoff{d} \mid \Vv \rangle
\]
for $a < b \leq c < d$, extends easily to infinite rectangles. We allow for the possibility that $a = -\infty$ or $d = +\infty$ by setting
\[
V_{-\infty} = 0,
\quad
V_{+\infty} = 0.
\]
In these cases the alternating sum formula of Proposition~\ref{prop:rankformula} becomes:
\newcommand{\muabcd}[2]{
\mu_\Vv \left(
\makebox[3.3em][r]{$#1$}\times\makebox[3.3em][l]{$#2$}
\right)
}
\begin{alignat*}{3}
\muabcd{[-\infty,b]}{[c,+\infty]}
&\;=\;
  \langle\,
  \qon{b}\qem\qon{c}
  \mid \Vv
  \,\rangle
&&\quad= \rk_b^c
\\
\muabcd{[a,b]}{[c,+\infty]}
&\;=\;
  \langle\,
  \qoff{a}\qem\qon{b}\qem\qon{c}
  \mid \Vv
  \,\rangle
&&\quad= \rk_b^c - \rk_a^c
&&\quad\text{(if $\rk_a^c < \infty$)}
\\
\muabcd{[-\infty,b]}{[c,d]}
&\;=\;
  \langle\,
  \qon{b}\qem\qon{c}\qem\qoff{d}
  \mid \Vv
  \,\rangle
&&\quad= \rk_b^c - \rk_b^d
&&\quad\text{(if $\rk_b^d < \infty$)}
\end{alignat*}
The first of these corresponds to the `$k$-triangle lemma' of~\cite{CohenSteiner_E_H_2007}.

So $\mu_\Vv$ can be thought of as an r-measure on the extended half-plane
\[
\Upper = \{ (p,q) \mid -\infty \leq p \leq q \leq +\infty \}
\]
and its diagram $\Dgm(\mu_\Vv)$ is defined on the subset of the r-interior
\[
\UpperRint = \{ (p^*, q^*) \mid -\infty^+ \leq p^* < q^* \leq +\infty^- \}
\]
on which $\mu_\Vv$ is finite.

Notice that there is a 1--1 correspondence between nonempty real intervals and elements of~$\UpperRint$.
The interval corresponding to $(p^*,q^*)$ is written $\lgroup p^*, q^* \rgroup$ as before. Proposition~\ref{prop:Rmembership} and Corollary~\ref{cor:Rcount} extend straightforwardly to this setting. In particular:
\begin{corollary}
\label{cor:Rcount2}
If $\Vv$ is decomposable into interval modules then $\mu_\Vv(R)$ counts the interval summands corresponding to decorated points which lie in~$R$.
\qed
\end{corollary}

We finish this section by defining the `measures at infinity' induced by an r-measure $\mu$. The extended plane has 4 lines at infinity
\[
(-\infty, \Rr),\quad
(+\infty, \Rr),\quad
(\Rr,-\infty),\quad
(\Rr,+\infty)
\]
and 4 points at infinity
\[
(-\infty,-\infty),\quad
(+\infty,-\infty),\quad
(-\infty,+\infty),\quad
(+\infty,+\infty)
\]
and, depending on where $\mu$ is finite, there are measures defined on each of these. On the lines at infinity these are `interval measures' (the 1-dimensional analogue of r-measures), and at the points at infinity these are simply numbers.

We write out the three cases of direct relevance to persistence modules. The other cases are similar.

\begin{vlist}
\item
{\bf\small the line $(-\infty, \Rr)$}:
\[
\mu(-\infty, [c,d])
	= \lim_{b \to -\infty} \mu([-\infty,b]\times[c,d])
	= \min_b \, \mu([-\infty,b]\times[c,d]) 
\]
for any interval $[c,d] \subseteq \Rr$.

\item
{\bf\small the line $(\Rr, +\infty)$}:
\[
\mu([a,b], +\infty)
	=  \lim_{c \to +\infty} \mu([a,b]\times[c,+\infty])
	= \min_c \, \mu([a,b]\times[c,+\infty]) 
\]
for any interval $[a,b] \subseteq \Rr$.

\item
{\bf\small the point $(-\infty, +\infty)$}:
\[
\mu(-\infty, +\infty)
	= \lim_{e \to +\infty} \mu([-\infty, -e]\times[e,+\infty])
	= \min_e \, \mu([-\infty, -e]\times[e,+\infty])
\]
\end{vlist}

Monotonicity of~$\mu$ guarantees that each limit exists, {except} when none of the rectangles in the expression belongs to~$\rect(\Dd)$. When that happens we adopt the extension convention (see section~\ref{subsec:non-finite}), and regard these rectangles has having infinite measure and the limit as being~$\infty$.

In order for the limit to be finite, we need at least one rectangle in the limiting expression to have finite measure. By considering the decorated diagram in that rectangle, it becomes clear what the `measure at infinity' measures: it counts the decorated points of $\Dgm(\mu)$ that lie in the segment, or corner, at infinity.

For example, when $\mu(-\infty, [c,d])$ is finite, it counts the decorated points of $\Dgm(\mu)$ of the form $(-\infty^+, q^*)$ where $q^* \in [c,d]$. When it is infinite, it means that some such $(-\infty^+, q^*)$ does not belong to $\frint(\mu)$, either because it does not belong to~$\drint$ or because all rectangles containing it have infinite mass.
%
%Indeed, when $b$ is small enough there are no other decorated points in $[-\infty,b]\times[c,d]$ but these.

%--------------------------------------------------
\subsection{Diagrams of persistence modules.}
\label{subsec:pd+}

We now have two competing definitions of the decorated diagram of a persistence module.
\begin{vlist}
\item
If $\Vv = \bigoplus_{\ell \in L} \Ii{\lgroup p^*_\ell, q^*_\ell \rgroup}$ is decomposable into intervals, then define
\begin{align*}
\Dgm(\Vv)
&= \Int(\Vv)
 = \left\{ (p_\ell^*, q_\ell^*) \mid \ell \in L \right\}.
\end{align*}
This is a multiset in $\UpperRint$.

\item
Let $\mu_\Vv$ be the persistence measure of~$\Vv$. Then 
\[
\Dgm(\Vv) = \Dgm(\mu_\Vv).
\]
This is a multiset in $\frint(\mu)$, the finite r-interior of~$\mu = \mu_\Vv$.

\end{vlist}

The overloading of the term `$\Dgm(\Vv)$' is partially excused by the following fact.

\begin{proposition}
If $\Vv$ is decomposable into intervals, then $\Int(\Vv)$ agrees with $\Dgm(\mu_\Vv)$ where the latter is defined, that is, on $\frint(\mu)$.
\end{proposition}

\begin{proof}
By Corollary~\ref{cor:Rcount2} we have
\[
\card ( \Int(\Vv) |_R ) = \mu_\Vv(R)
\]
for all rectangles. On the other hand, we have
\[
\card ( \Dgm(\mu_\Vv) |_R ) = \mu_\Vv(R)
\]
for all rectangles with $\mu_\Vv(R) < \infty$. By uniqueness, it follows that $\Int(\Vv)$ and $\Dgm(\mu_\Vv)$ must be the same multiset when restricted to $\frint(\mu_\Vv)$.
\end{proof}

Neither definition strictly outperforms the other, as the following examples show.

\begin{example}
Let
\[
\Vv = \bigoplus_{\ell \in L} \Ii{\lgroup p_\ell^*, q_\ell^* \rgroup}
\]
where the undecorated pairs $(p_\ell, q_\ell)$ form a dense subset of the half-plane $\Upper$. Then $\Int(\Vv)$ is defined; but $\mu_\Vv(R) = \infty$ for every rectangle, so $\frint(\mu_\Vv)$ is the empty set and $\Dgm(\mu_\Vv)$ is nowhere defined.
\end{example}

\begin{example}
\label{ex:webb2}
Adapting the example of Webb~\cite{Webb_1985} to the real line:
\begin{alignat*}{2}
W_{t} &= 0 && \quad\text{for $t > 0$}
\\
W_{0} &= \{ \text{sequences $(x_1, x_2, x_3, \dots)$ of real numbers} \}
\\
W_{t} &= \{ \text{sequences with $x_n = 0$ for all $n \leq |t|$} \}
	&& \quad\text{for $t < 0$}
\end{alignat*}
As before, this is not decomposable into intervals. On the other hand, it is easy to see that
\[
\langle \qoff{a}\qem\qon{b} \mid \Ww \rangle
= \conullity(W_a \to W_b) < \infty
\]
except when $a = -\infty$, and
\[
\langle \qon{c}\qem\qoff{d} \mid \Ww \rangle
= \nullity(W_c \to W_d) < \infty
\]
except when $c \leq 0 < d$. Each of these terms dominates
\[
\langle \qoff{a}\qem\qon{b}\qem\qon{c}\qem\qoff{d} \mid \Ww \rangle,
\]
which is therefore finite for all rectangles which do not contain $(-\infty, 0^+)$. 
It is not difficult to complete the argument that
\[
\frint(\mu) = \UpperRint - \{ (-\infty, 0^+) \}
\]
and to determine the persistence diagram
\[
\Dgm(\mu_\Ww)
= \left\{ (-n^{+}, 0^+) \mid n = 1, 2, 3, \dots \right\}
\]
over $\frint(\mu)$.
\begin{figure}
\centerline{
\includegraphics[height=2.25in]{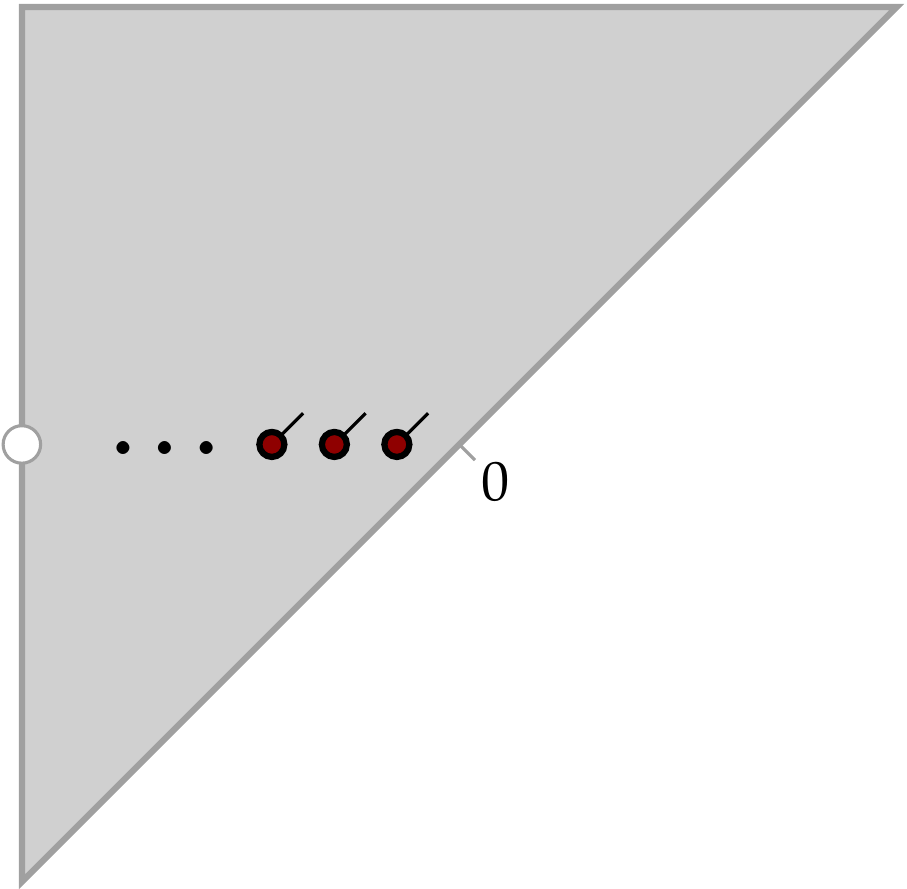}
}
\caption{The persistence diagram $\Dgm(\Ww) = \Dgm(\mu_\Ww)$ for the example of Webb. This is defined everywhere in the extended half-plane except $(-\infty, 0^+)$.}
\label{fig:webb2}
\end{figure}
See Figure~\ref{fig:webb2}.
In other words, although this example is not decomposable into intervals, its persistence measure is well-behaved away from the singular support, which is the point~$(-\infty, 0^+)$.
\end{example}

%--------------------------------------------------
\subsection{Tameness conditions}
\label{subsec:tameness}

We now describe several different levels of `tameness' for a persistence module, starting with the tamest.

A persistence module $\Vv$ is of {\bf finite type}, if it is a finite direct sum of interval modules.

It is {\bf locally finite}, if it is a direct sum of interval modules such that every $t \in \Rr$ has a neighbourhood which meets only finitely many of the intervals. Since closed bounded intervals are compact, this is equivalent to the stronger assertion that every bounded subset of~$\Rr$ meets only finitely many of the intervals.

\begin{proposition}
\label{prop:locally-finite}
A persistence module $\Vv$ is locally finite if and only if (i)~each $V_t$ is finite-dimensional and (ii)~there is a locally finite set $\Ss \subset \Rr$ such that $v_b^c$ is an isomorphism for every pair $b < c$ with $[b,c] \cap \Ss = \emptyset$.
\end{proposition}

In less formal language, condition~(ii) asserts that $\Vv$ is constant over each interval of the open set $\Rr - \Ss$.

\begin{proof}
Suppose conditions (i) and (ii) hold.
Select a  locally finite countable set
\[
\Tt:\quad
\dots < t_{-n} < \dots < t_{-1} < t_0 < t_1 < \dots < t_n < \dots
\]
unbounded in both directions, where the even-numbered points include the entire set~$\Ss$. By the result of Webb used in Theorem~\ref{thm:gabriel+}, we can decompose $\Vv_\Tt$ into intervals. Since $\Vv$ is constant over each $(t_{2n}, t_{2n+2})$,
this extends uniquely to a decomposition of~$\Vv$ into interval modules.
Any bounded subset of~$\Rr$ meets only finitely many of the $t_{2n}$ and $(t_{2n}, t_{2n+2})$. Applying condition~(i) we find that it meets only finitely many of the intervals.

The converse direction is immediate: take $\Ss$ to be the set of endpoints of the intervals in the decomposition of~$\Vv$.
\end{proof}

For a finite or locally finite persistent module, it is easy to see that $\frint(\mu) = \Upper$, so both persistence diagrams are defined on all of $\Upper$, including on the diagonal and at infinity.
Diagonal points must be decorated $(p^-,p^+)$, since only these belong to rectangles in $\Upper$.

We introduce four more kinds of tameness. The assumptions here are about finiteness of $\mu_\Vv$ over different types of rectangle (see Figure~\ref{fig:spqr}).
\begin{figure}
\centerline{
\includegraphics[height=2.25in]{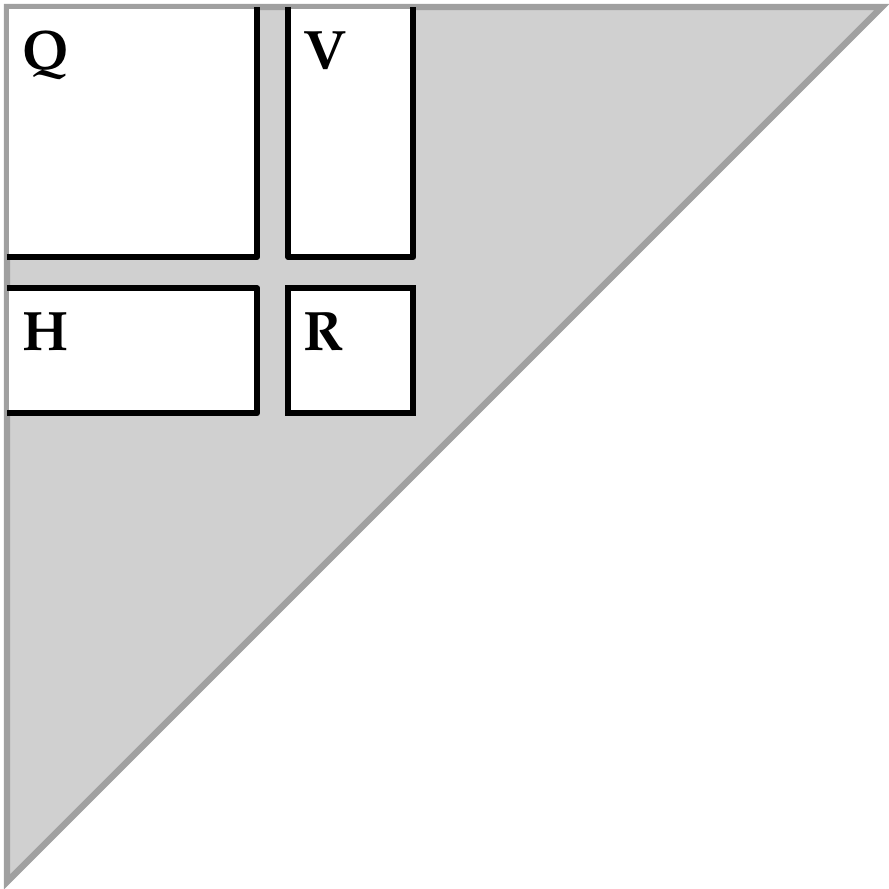}
}
\caption{A quadrant, horizontal strip, vertical strip, and finite rectangle in $\Upper$.}
\label{fig:spqr}
\end{figure}
Each condition guarantees the existence of the persistence diagram over a certain subset of the extended half-plane. The finite part of the plane (except the diagonal) is always included; it is at infinity that the four conditions differ.

\newcommand{\dgmplot}[1]{%
\raisebox{-1.8ex}{\includegraphics[height=5ex]{figures/#1}}
}
(i) We say that $\Vv$ is {\bf q-tame}, if $\mu_\Vv(Q) < \infty$ for every quadrant~$Q$ not touching the diagonal. In other words
\[
\langle \qon{b}\qem\qon{c} \mid \Vv \rangle < \infty
\]
(that is, $\rk_b^c < \infty$) for all $b < c$. 
The persistence diagram $\Dgm(\mu_\Vv)$ is defined over the set:
\[
\{ (p^*, q^*) \mid -\infty \leq p < q \leq +\infty \}
=
\dgmplot{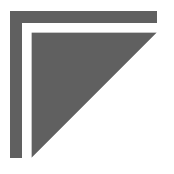}
\]

(ii) We say that $\Vv$ is {\bf h-tame}, if $\mu_\Vv(H) < \infty$ for every horizontally infinite strip~$H$ not touching the diagonal. In other words,
\[
\langle \qon{b}\qem\qon{c}\qem\qoff{d} \mid \Vv \rangle < \infty
\]
for all $b < c < d$.
The persistence diagram $\Dgm(\mu_\Vv)$ is defined over the set:
\[
\{ (p^*, q^*) \mid -\infty \leq p < q < +\infty \}
=
\dgmplot{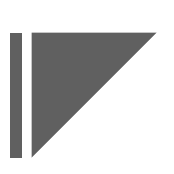}
\]

(iii) We say that $\Vv$ is {\bf v-tame}, if $\mu_\Vv(V) < \infty$ for every vertically infinite strip~$V$ not touching the diagonal. In other words,
\[
\langle \qoff{a}\qem\qon{b}\qem\qon{c} \mid \Vv \rangle < \infty
\]
for all $a < b < c$.
The persistence diagram $\Dgm(\mu_\Vv)$ is defined over the set:
\[
\{ (p^*, q^*) \mid -\infty < p < q \leq +\infty \}
=
\dgmplot{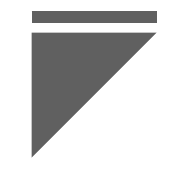}
\]

(iv) We say that $\Vv$ is {\bf r-tame}, if $\mu_\Vv(R) < \infty$ for every finite rectangle~$R$ not touching the diagonal. In other words,
\[
\langle \qoff{a}\qem\qon{b}\qem\qon{c}\qem\qoff{d}
\mid \Vv \rangle < \infty
\]
for all $a < b < c < d$.
The persistence diagram $\Dgm(\mu_\Vv)$ is defined over the set:
\[
\{ (p^*, q^*) \mid -\infty < p < q < +\infty \}
=
\dgmplot{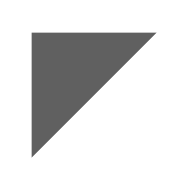}
\]

Here is the diagram of inclusions between the different classes of tame persistence modules:
\[
\begin{diagram}
\divide\dgARROWLENGTH by3
\node[4]{\text{v-tame}}
  \arrow{se}
\\
\node{\text{finite}}
  \arrow{e}
\node{\text{locally finite}}
  \arrow{e}
\node{\text{q-tame}}
  \arrow{ne}
  \arrow{se}
\node[2]{\text{r-tame}}
\\
\node[4]{\text{h-tame}}
  \arrow{ne}
\end{diagram}
\]
One can show that all the inclusions are strict, and also
\[
\text{q-tame} \subsetneq \text{h-tame} \cap \text{v-tame},
\quad
\text{h-tame} \cup \text{v-tame} \subsetneq \text{r-tame}.
\]
Examples for these last two assertions are suggested by the diagrams in Figure~\ref{fig:r-hv-q}.
\begin{figure}
\centerline{
\hfill
\includegraphics[scale=0.5]{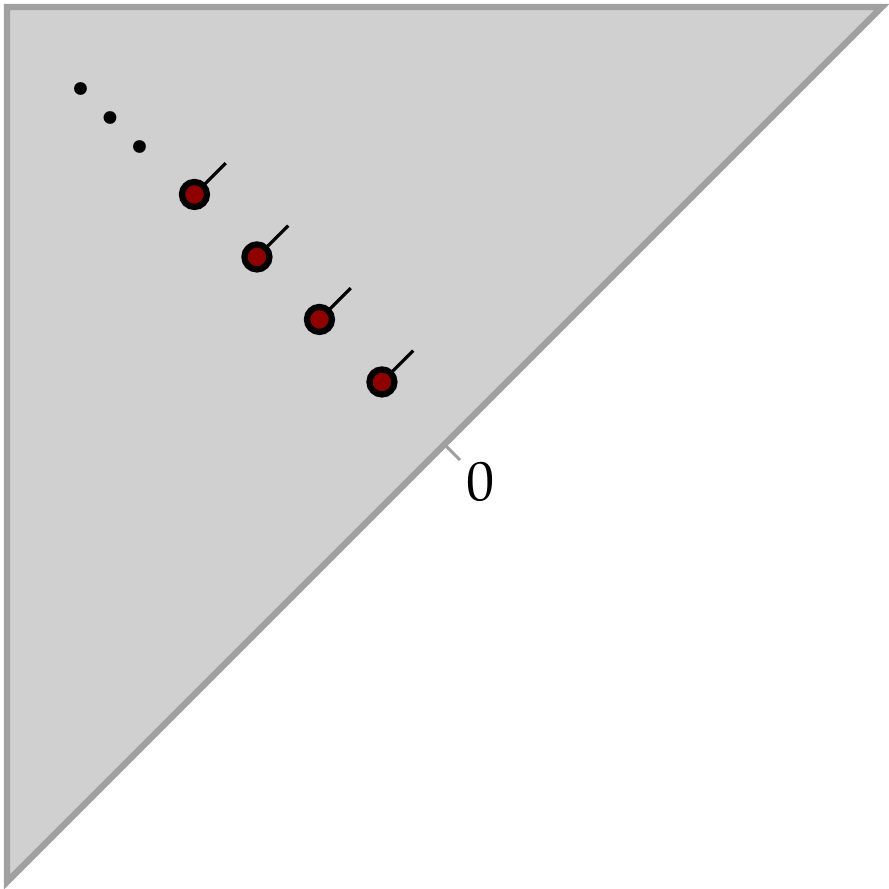}
\hfill
\includegraphics[scale=0.5]{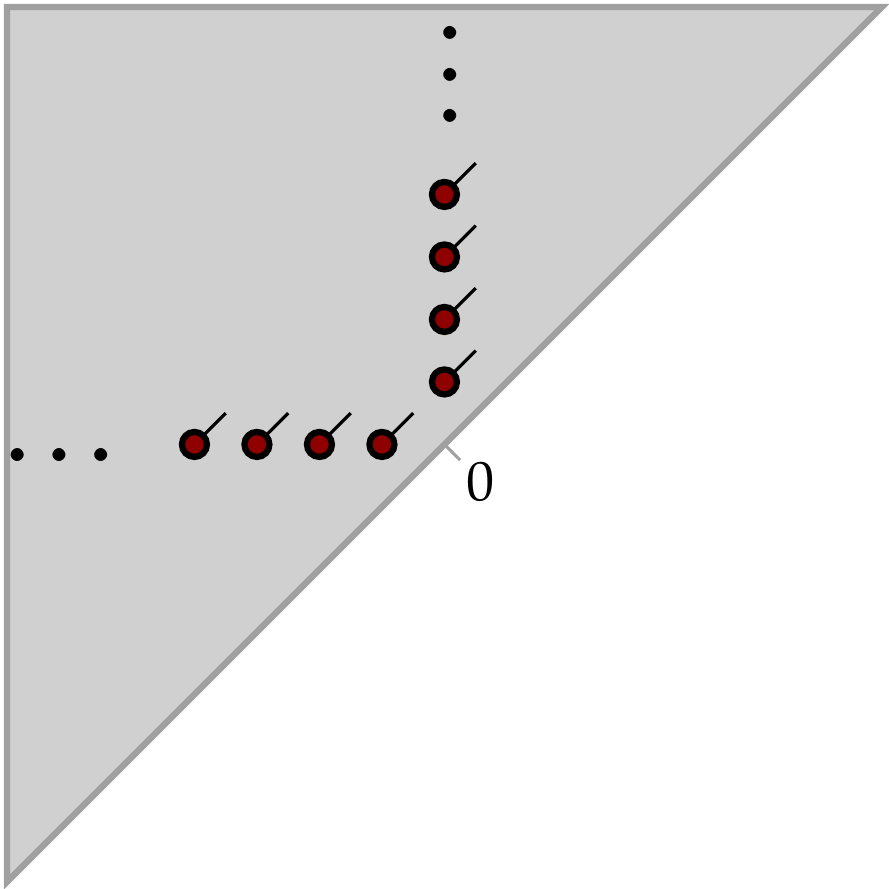}
\hfill
}
\caption{Diagrams of persistence modules which are: (left) h-tame and v-tame but not q-tame; (right) r-tame but not h-tame or v-tame.}
\label{fig:r-hv-q}
\end{figure}

We have not determined whether
\[
\text{h-tame} \oplus \text{v-tame} = \text{r-tame}.
\]
If an r-tame module is decomposable then it can certainly be written as the direct sum of an h-tame module and a v-tame module, by partitioning the intervals appropriately. The general situation seems more subtle.

Later we will show that `q-tame' can be thought of as the completion of `locally finite': a persistence module is q-tame if and only if it can be approximated arbitrarily well by locally finite modules.
See Theorem~\ref{thm:qtame-char}.

Many naturally occurring persistence modules are q-tame. Here is a typical example.

\begin{theorem}
\label{thm:poly-q-tame}
Let $X$ be a finite polyhedron,\footnote{%
By `polyhedron' we mean the realisation of a simplicial complex as a topological space.
}
and let $f: X \to \Rr$ be a continuous function. Then persistent homology
$
\Hgr(\Xx_\sub)
$
of the sublevelset filtration of $(X,f)$ is q-tame.
\end{theorem}

\begin{proof}
For any $b < c$ we must show that
\[
\Hgr(X^b) \to \Hgr(X^c)
\]
has finite rank. Begin with any triangulation of~$X$, and subdivide it repeatedly until no simplex meets both $f^{-1}(b)$ and $f^{-1}(c)$. If we define $Y$ to be the union of the closed simplices which meet $X^b$, then we have
\[
X^b \subseteq Y \subseteq X^c
\]
and hence the factorisation
\[
\Hgr(X^b) \to \Hgr(Y) \to \Hgr(X^c).
\]
Since $Y$ is a finite polyhedron, $\Hgr(Y)$ is finite dimensional, so $\Hgr(X^b) \to \Hgr(X^c)$ has finite rank.
\end{proof}

%\begin{remark}
We emphasise that
individual vector spaces $\Hgr(X^b)$ in the persistence module may be infinite-dimensional. For instance let $X$ be a closed disk in the plane, and let $f$ be a non-negative function whose zero set is a Hawaiian earring; then $\Hgr_1(X^0)$ is uncountable-dimensional.
%
%The theorem shows that such pathological behaviour is ephemeral; it does not persist over any nontrivial interval of parameter values.
%\end{remark}
%
%
The theorem nonetheless applies.

With a little more work we can show:

\begin{theorem}
\label{thm:poly-hv-tame}
Let $X$ be a locally compact polyhedron, and let $f: X \to \Rr$ be a proper continuous function. Then persistent homology
$
\Hgr(\Xx_\sub)
$
of the sublevelset filtration of $(X,f)$ is h-tame and v-tame (and r-tame).
\end{theorem}

The theorem essentially asserts that the only possible bad behaviour of $\Hgr(\Xx_\sub)$ occurs at $(-\infty, +\infty)$. It is easy to construct examples which are definitely not q-tame. The simplest example is $X = \Zz$ (the integers), with $f(n) = n$. The 0-homology of any sublevelset is infinite dimensional, and all inclusions have infinite rank.

\begin{proof}
{\bf h-tameness.}
Let $b < c < d$. We must show that
\[
\langle \qon{b}\qem\qon{c}\qem\qoff{d}
\mid \Hgr(\Xx_\sub) \rangle
< \infty.
\tag{h-*}
\]
Begin with a triangulation of~$X$. Only finitely many simplices  meet the compact set $f^{-1}(b)$, so again after a finite number of subdivisions no simplex meets both $f^{-1}(b)$ and $f^{-1}(c)$.

Now let $Y$ be the union of the closed simplices which meet $X^b$, and let $Z$ be the union of the closed simplices which meet $X^d$. This gives a diagram of inclusions
\[
X^b \subseteq Y \subseteq X^c \subseteq X^d \subseteq Z.
\]
Note that the polyhedron~$Z$ differs from its subpolyhedron~$Y$ by the addition of only finitely many simplices, since each such simplex must meet the compact set $f^{-1}[b,d]$. Thus the relative homology $\Hgr(Z,Y)$ is finite-dimensional.

We now work with the induced homology diagram
\[
\Hgr(X^b) \to \Hgr(Y) \to \Hgr(X^c) \to \Hgr(X^d) \to \Hgr(Z).
\]
In the obvious notation,
\begin{align*}
\langle \qon{b}\qem\qno\qem\qon{c}\qem\qoff{d}\qem\qno \rangle
&=
\langle \qon{b}\qem\qon{y}\qem\qon{c}\qem\qoff{d}\qem\qoff{z} \rangle
\\
&\leq
\langle \qno\qem\qon{y}\qem\qno\qem\qno\qem\qoff{z} \rangle
\\
&= \dim[\ker(\Hgr(Y) \to \Hgr(Z))].
\end{align*}
By the homology long exact sequence for the pair $(Z,Y)$, we have
\[
\ker(\Hgr(Y) \to \Hgr(Z)) = \img(\Hgr(Z,Y) \to \Hgr(Y))
\]
which is finite-dimensional. This confirms~(h-*).

{\bf v-tameness.} Let $a < b < c$. We must show that
\[
\langle \qoff{a}\qem\qon{b}\qem\qon{c}
\mid \Hgr(\Xx_\sub) \rangle < \infty
\tag{v-*}
\]
Using a similar argument to the above, we construct a diagram of inclusions
\[
Y \subseteq X^a \subseteq X^b \subseteq Z \subseteq X^c
\]
where $Y, Z$ are polyhedra with $\Hgr(Z,Y)$ finite-dimensional. 
Working with the homology diagram
\[
\Hgr(Y) \to \Hgr(X^a) \to \Hgr(X^b) \to \Hgr(Z) \to \Hgr(X^c),
\]
we estimate
\begin{align*}
\langle \qno\qem\qoff{a}\qem\qon{b}\qem\qno\qem\qon{c} \rangle
&=
\langle \qoff{y}\qem\qoff{a}\qem\qon{b}\qem\qon{z}\qem\qon{c} \rangle
\\
&\leq
\langle \qoff{y}\qem\qno\qem\qno\qem\qon{z}\qem\qno \rangle
\\
&=
\dim[\coker(\Hgr(Y) \to \Hgr(Z))]
\end{align*}
By the homology long exact sequence of the pair $(Z,Y)$, we have
\[
\coker(\Hgr(Y) \to \Hgr(Z)) \cong \img(\Hgr(Z) \to \Hgr(Z,Y))
\]
which is finite-dimensional. This confirms~(v-*).
\end{proof}

\bigskip
\begin{remark}
We find that q-tame modules occur rather widely. For instance, it is shown in~\cite{Chazal_dS_O_2012} that the Vietoris--Rips and \v{C}ech complexes of a pre-compact metric space have q-tame persistent homology.
At the same time, such complexes can be very badly behaved when viewed non-persistently. In fact, there is a compact metric space whose Vietoris--Rips complex has uncountable dimension at uncountably many parameter values (indeed, over an entire interval). The construction is not at all pathological in appearance; see~\cite{Chazal_dS_O_2012}. The first examples of this type were shown to us by J.-M.~Droz; see~\cite{Droz_2012} for another.
These theorems and examples support our contention that q-tame persistence modules are the natural class to work with.
\end{remark}

%--------------------------------------------------
\subsection{Finite approximations}
\label{subsec:snapping}

Away from the finite r-interior, the r-measure gives a decidedly limited view of the structure of a persistence module. For example:

\begin{vlist}
\item
It is not possible to distinguish between the many nonisomorphic persistence modules $\Vv$ for which $\mu_\Vv$ is infinite on every rectangle.

\item
If the persistence diagram of~$\Vv$ contains a sequence of points $( p^*_n, q^*_n )$ with $p_n, q_n$ converging to $r$ from below and above, respectively, then there is no way to determine the multiplicity of $( r^-, r^+ )$ from the measure alone.
\end{vlist}

On the other hand, from $\mu_\Vv$ we do recover all information obtainable by restricting $\Vv$ to finite subsets $\Tt \subset \Rr$. We may call this the `finitely observable' part of~$\Vv$. Specifically, for any finite index set
\[
\Tt : \quad a_1 < a_2 < \dots < a_n
\]
we have
\begin{align*}
\langle [a_i, a_j] \mid \Vv_\Tt \rangle
&=
\langle
 \;
 \Qoff{a_{i-1}}\qem\Qon{a_{i}}
 \qem\Qon{a_{j}}\qem\,\,\,\Qoff{a_{j+1}}
\mid
  \Vv
\rangle
=
\mu_\Vv([a_{i-1}, a_i] \times [a_j, a_{j+1}])
\tag{$\dagger$}\label{eq:snap}
\end{align*}
(where $a_0, a_{n+1}$ are interpreted as $-\infty, +\infty$ respectively).

This can be interpreted as a sort of `snapping principle' (see~\cite{Chazal_CS_G_G_O_2009,Chazal_CS_G_G_O_2008} for the origin of this term). The left-hand side of~\eqref{eq:snap} gives the multiplicity of the interval $[a_i, a_{j+1})$ in the traditional definition of the barcode for $\Vv_\Tt$. For us, this is $\lgroup a_i^-, a_{j+1}^- \rgroup$. The right-hand side of~\eqref{eq:snap} counts the decorated points of $\Dgm(\mu_\Vv)$ in the rectangle $[a_{i-1}, a_i]\times[a_j,a_{j+1}]$.
\begin{figure}
\includegraphics[height=2.25in]{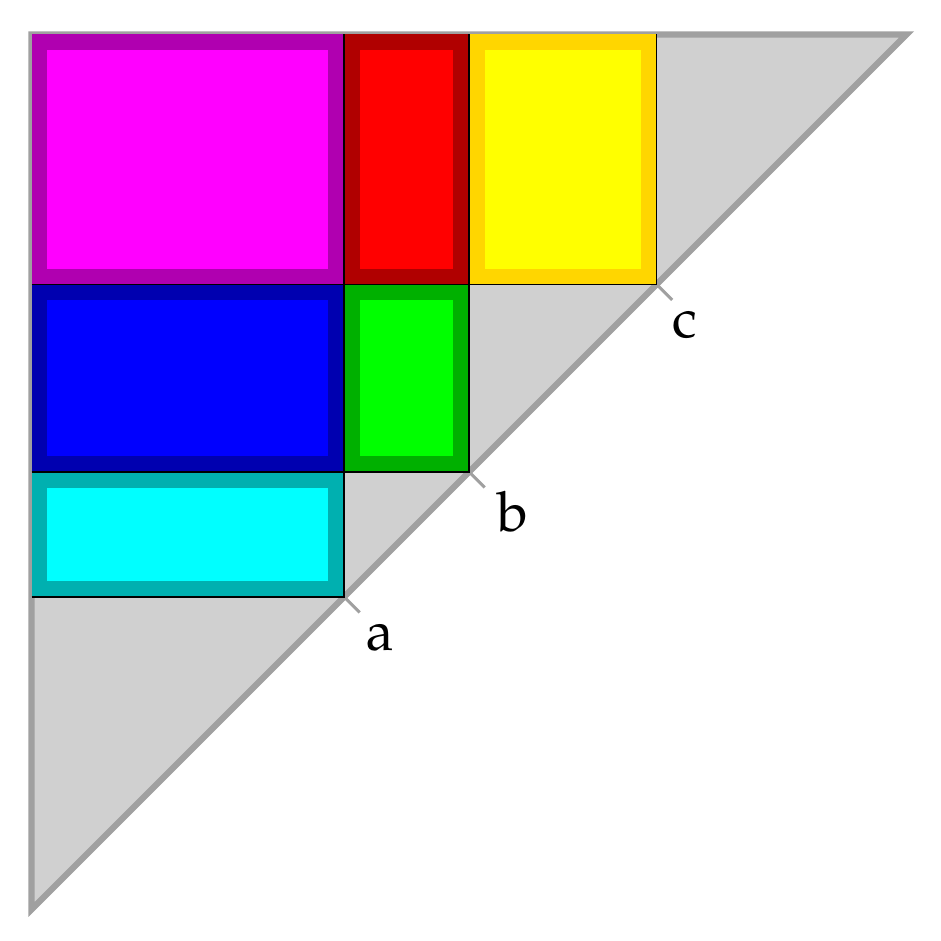}
\quad
\raisebox{1.1in}
	{$\stackrel{\text{snap}}{\longrightarrow}$}
\quad
\includegraphics[height=2.25in]{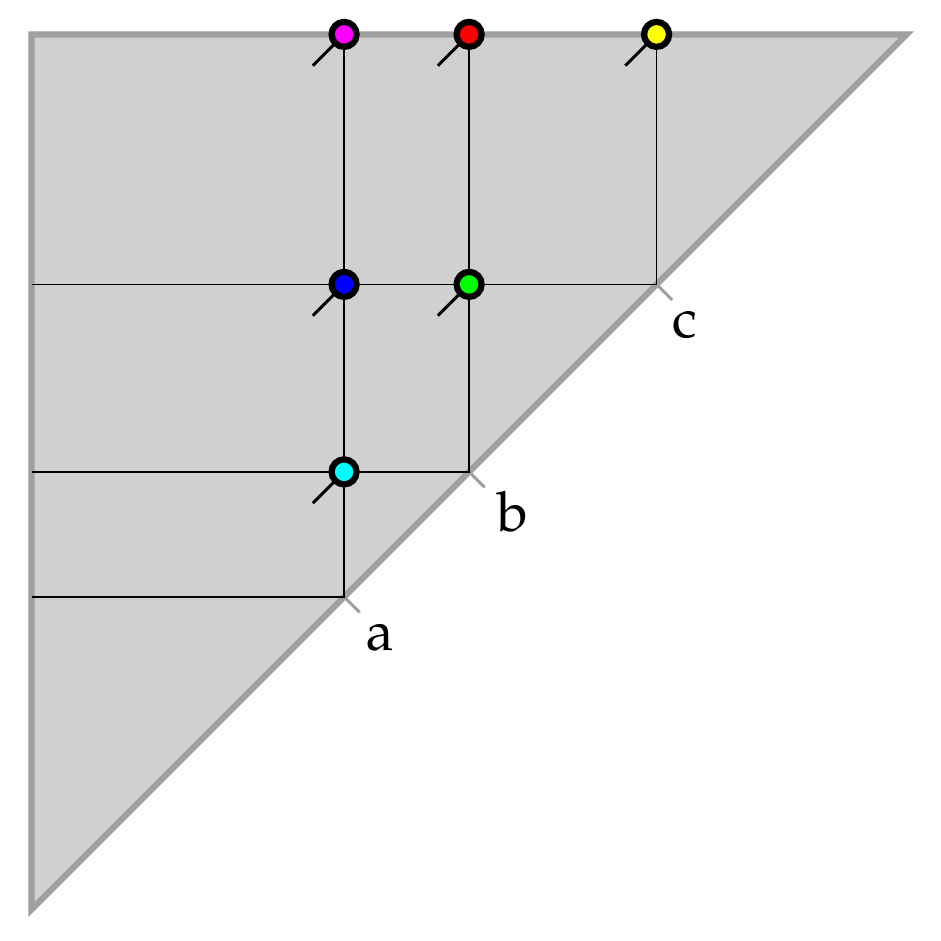}
\caption{A persistence module $\Vv$, discretised at $a, b, c$. 
The persistence diagram $\Dgm(\Vv_{a,b,c})$ is localised at six grid vertices. The multiplicity of each vertex is equal to the number of decorated points of $\Dgm(\Vv)$ in the rectangle immediately below and to the left of it. Decorated points of $\Dgm(\Vv)$ in the remaining triangular regions do not show up in $\Dgm(\Vv_{a,b,c})$.
}
\label{fig:snap}
\end{figure}
See Figure~\ref{fig:snap}.

There are some well known situations in which the entire structure of~$\Vv$ is finitely observable. For example, suppose
\begin{vlist}
\item
$X$ is a compact manifold and $f$ is a Morse function; or

\item
$X$ is a compact polyhedron and $f$ is piecewise linear.
\end{vlist}
In each case there is a finite set of critical points
\[
\Tt : \quad
a_1 < a_2 < \dots < a_n
\]
such that the inclusion of sublevelsets $X^s \subset X^t$ is a homotopy equivalence provided that $(s,t]$ does not meet any of the critical points. It follows that the sublevelset persistent homology $\Vv = \Hgr(\Xx_\sub)$ is constant on each of the intervals
\[
(-\infty, a_1),
\quad
[a_1, a_2),
\quad
[a_2, a_3),
\quad
\dots,
\quad
[a_{n-1}, a_n),
\quad
[a_n, +\infty).
\]
The structure of $\Vv$ is therefore determined by $\Vv_\Tt$, it is of finite type, and $\Dgm(\Vv)$ is localised at the points
\begin{alignat*}{1}
(a_i^-, a_j^-)
&\quad
\text{for $1\leq i < j \leq n$},
\\
(a_i^-, +\infty)
&\quad
\text{for $1\leq i \leq n$}.
\end{alignat*}

%-------------------------------------------------------------------
\section{Interleaving}
\label{sec:interleaving}

As with any category, two persistence modules $\Uu, \Vv$ are said to be isomorphic if there are maps
\[
\Phi \in \Hom(\Uu, \Vv),
\quad
\Psi \in \Hom(\Vv, \Uu),
\]
such that
\[
\Psi \Phi = 1_\Uu,
\quad
\Phi \Psi = 1_\Vv.
\]
This relation is too strong for situations where the data leading to the construction of a persistence module is obtained with some uncertainty or noise. The natural response is to consider a weaker relation, {\bf $\delta$-interleaving}, where $\delta \geq 0$ quantifies the  uncertainty.

In this section, we define the interleaving relation and study its elementary properties. We prove the nontrivial result (from~\cite{Chazal_CS_G_G_O_2009}) that if two persistence modules are $\delta$-interleaved, then they are connected in the space of persistence modules by a path of length~$\delta$. This `interpolation lemma' is a crucial step in the proof of the stability theorem in section~\ref{sec:isometry}.

In this section, all persistence modules are indexed by~$\Rr$ unless explicitly stated otherwise.

%--------------------------------------------------
\subsection{Shifted homomorphisms}

The first step is to consider homomorphisms which shift the value of the persistent index. Let $\Uu, \Vv$ be persistence modules over~$\Rr$, and let $\delta$ be any real number. A {\bf homomorphism of degree~$\delta$} is a collection $\Phi$ of linear maps
\[
\phi_t : U_t \to V_{t+\delta}
\]
for all $t \in \Rr$, 
such that the diagram
\[
\begin{diagram}
\node{U_s}
  \arrow{e,t}{u_s^{t}}
  \arrow{s,l}{\phi_s}
\node{U_{t}}
  \arrow{s,r}{\phi_{t}}
\\
\node{V_{s+\delta}}
  \arrow{e,t}{v_{s+\delta}^{t+\delta}}
\node{V_{t+\delta}}
\end{diagram}
\]
commutes whenever $s \leq t$.

We write
\begin{align*}
\Hom^\delta(\Uu, \Vv) &= \{ \text{homomorphisms $\Uu \to \Vv$ of degree~$\delta$} \},
\\
\End^\delta(\Vv) &= \{ \text{homomorphisms $\Vv \to \Vv$ of degree~$\delta$} \}.
\end{align*}
Composition gives a map
\[
\Hom^{\delta_2}(\Vv, \Ww) \times \Hom^{\delta_1}(\Uu, \Vv)
\to
\Hom^{\delta_1+\delta_2}(\Uu, \Ww).
\]

For $\delta \geq 0$, the most important degree-$\delta$ endomorphism is the shift map
\[
1_\Vv^\delta \in \End^\delta(\Vv),
\]
which is the collection of maps $( v_t^{t+\delta} )$ from the persistence structure on~$\Vv$.
If $\Phi$ is a homomorphism $\Uu \to \Vv$ of any degree, then by definition
\[
\Phi 1_\Uu^\delta = 1_\Vv^\delta \Phi
\]
for all $\delta \geq 0$.

%--------------------------------------------------
\subsection{Interleaving}
\label{subsec:interleaving}

Let $\delta \geq 0$. Two persistence modules $\Uu, \Vv$ are said to be {\bf $\delta$-interleaved} if there are maps
\[
\Phi \in \Hom^\delta(\Uu, \Vv),
\quad
\Psi \in \Hom^\delta(\Vv, \Uu)
\]
such that
\[
\Psi \Phi = 1_\Uu^{2\delta}, %\in \End^{2\delta}(\Uu)
\quad
\Phi \Psi = 1_\Vv^{2\delta}. %\in \End^{2\delta}(\Vv)
\]
More expansively (with many more indices), this means that there are maps
\[
\phi_t: U_t \to V_{t+\delta}
\quad \text{and} \quad
\psi_t: V_t \to U_{t+\delta}
\]
defined for all~$t$, such that the following diagrams
\begin{equation}
\label{eq:int_expansive}
\begin{array}{cc}
\begin{diagram}
\node{U_s}
  \arrow{e,t}{u_s^t}
  \arrow{s,l}{\phi_s}
\node{U_t}
  \arrow{s,r}{\phi_t}
\\
\node{V_{s+\delta}}
  \arrow{e,t}{v_{s+\delta}^{t+\delta}}
\node{V_{t+\delta}}
\end{diagram}
\quad&\quad
\begin{diagram}
\node{U_{s-\delta}}
  \arrow[2]{e,t}{u_{s-\delta}^{s+\delta}}
  \arrow{se,b}{\phi_{s-\delta}}
\node[2]{U_{s+\delta}}
\\
\node[2]{V_{s}}
  \arrow{ne,b}{\psi_{s}}
\end{diagram}
\\
\\
\begin{diagram}
\node{V_s}
  \arrow{e,t}{v_s^t}
  \arrow{s,l}{\psi_s}
\node{V_t}
  \arrow{s,r}{\psi_t}
\\
\node{U_{s+\delta}}
  \arrow{e,t}{u_{s+\delta}^{t+\delta}}
\node{U_{t+\delta}}
\end{diagram}
\quad&\quad
\begin{diagram}
\node{V_{s-\delta}}
  \arrow[2]{e,t}{v_{s-\delta}^{s+\delta}}
  \arrow{se,b}{\psi_{s-\delta}}
\node[2]{V_{s+\delta}}
\\
\node[2]{U_{s}}
  \arrow{ne,b}{\phi_{s}}
\end{diagram}
\end{array}
\end{equation}
commute for all eligible parameter values; that is, for all $s < t$.
Where possible, we will be concise rather than expansive.

Here is the classic example of a pair of interleaved persistence modules.

\begin{example}
Let $X$ be a topological space and let $f, g: X \to \Rr$. Suppose $\| f - g \|_\infty < \delta$. Then the persistence modules $\Hgr(\Xx^f_\sub)$, $\Hgr(\Xx^g_\sub)$ are $\delta$-interleaved. Indeed, there are inclusions
\begin{align*}
(X,f)^t  &\subseteq (X,g)^{t+\delta}
\\
(X,g)^t  &\subseteq (X,f)^{t+\delta}
\end{align*}
for all~$t$, which induce maps
\begin{align*}
\Phi &: \Hgr(\Xx^f_\sub) \to \Hgr(\Xx^g_\sub)
\\
\Psi &: \Hgr(\Xx^g_\sub) \to \Hgr(\Xx^f_\sub)
\end{align*}
of degree~$\delta$. Since all the maps are induced functorially from inclusion maps, the interleaving relations are automatically satisfied.
\end{example}

This is the situation for which the stability theorem of Cohen-Steiner, Edelsbrunner and Harer~\cite{CohenSteiner_E_H_2007} was originally stated: if two functions $f,g$ are close then the diagrams for their sublevelset persistent homology are close. Subsequently, stability was formulated as a theorem about the diagrams of interleaved persistence modules \cite{Chazal_CS_G_G_O_2009,Chazal_CS_G_G_O_2008}. In the present paper, we come to view stability as a theorem about r-measures.

%--------------------------------------------------
\subsection{Interleaving (continued)}
\label{subsec:interleaving2}

An interleaving between two persistence modules can itself be thought of  as a persistence module over a certain partially ordered set (poset). We develop this idea now.

Consider the standard partial order on the plane:
\[
(p_1, q_1) \leq (p_2, q_2)
\quad\Leftrightarrow\quad
p_1 \leq p_2
\;\text{and}\;
q_1 \leq q_2.
\]
For any real number~$x$, define the shifted diagonal
\[
\Delta_x = \{ (p,q) \mid q-p = 2x \} \subset \Rr^2.
\]
As a poset, this is isomorphic to the real line. Specifically, we identify $t \in \Rr$ with the point $(t-x, t+x) \in \Delta_x$.
Through this, we get a canonical identification between persistence modules over~$\Rr$ and persistence modules over~$\Delta_x$.

\begin{proposition}
\label{prop:poset-interleaving}
Let $x,y$ be real numbers. Persistence modules $\Uu, \Vv$ are $|y-x|$-interleaved if and only if there is a persistence module $\Ww$ over $\Delta_x \cup \Delta_y$ such that $\Ww|_{\Delta_x} = \Uu$ and $\Ww|_{\Delta_y} = \Vv$.
\end{proposition}

\begin{proof}
We may assume $x < y$.

We claim that (i) the extra information carried by $(y-x)$-interleaving maps $\Phi, \Psi$ is equivalent to (ii) the extra information carried by $\Ww$. Let us describe both, more carefully:

\medskip
\quad
(i) In addition to $\Uu, \Vv$ we have a system of maps $\Phi = (\phi_t)$, where
\[
\phi_t : U_t \to V_{t+y-x},
\]
and a system of maps $\Psi = (\psi_t)$, where
\[
\psi_t : V_t \to U_{t+y-x}.
\]
These are constrained by the relations (for all $\eta \geq 0$).
\[
\tag{$\ddagger$}
\label{eq:int-relations}
\Phi 1_\Uu^\eta = 1_\Vv^\eta \Phi,
\quad
\Psi 1_\Vv^\eta = 1_\Uu^\eta \Psi,
\quad
\Psi\Phi = 1_\Uu^{2y-2x},
\quad
\Phi\Psi = 1_\Vv^{2y-2x}.
\]
There are no other constraints.

\medskip
\quad
(ii) In addition to $\Uu, \Vv$ the persistence module~$\Ww$ carries maps between the two components $\Delta_x, \Delta_y$. These maps are constrained by the composition law
\[
w_R^T = w_S^T \circ w_R^S
\]
for all $R,S,T \in \Delta_x \cup \Delta_y$ with $R \leq S \leq T$.

First, observe that we recover the maps $\phi_t, \psi_t$ as vertical maps from $\Delta_x$ to~$\Delta_y$, and horizontal maps from $\Delta_y$ to~$\Delta_x$, respectively (see Figure~\ref{fig:PhiPsi}):
\begin{alignat*}{3}
U_t &= W_{t-x,t+x} &&\to W_{t-x,t+2y-x} &&= V_{t+y-x}
\\
V_t &= W_{t-y,t+y} &&\to W_{t+y-2x,t+y} &&= U_{t+y-x}
\end{alignat*}
\begin{figure}
\hfill
\includegraphics[scale=0.7]{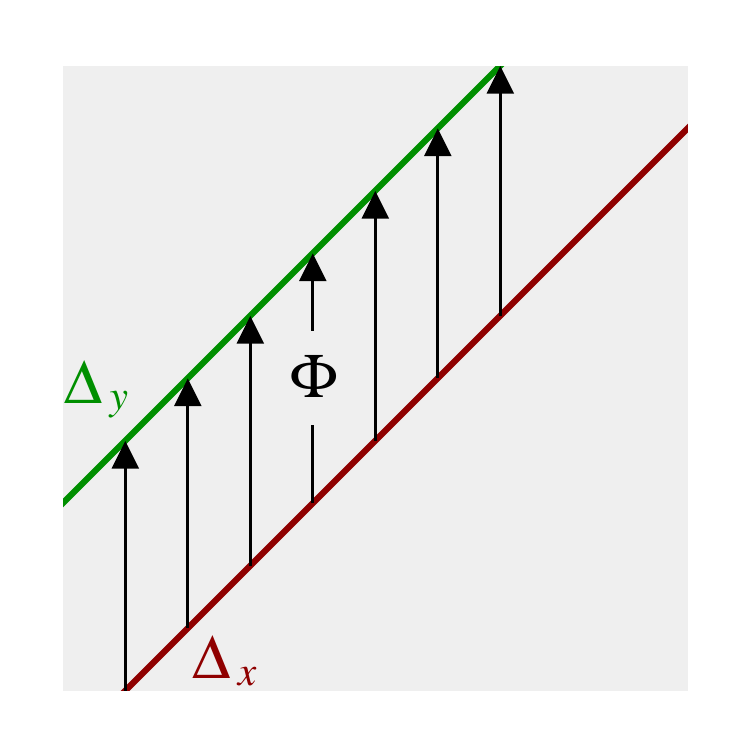}
\hfill
\includegraphics[scale=0.7]{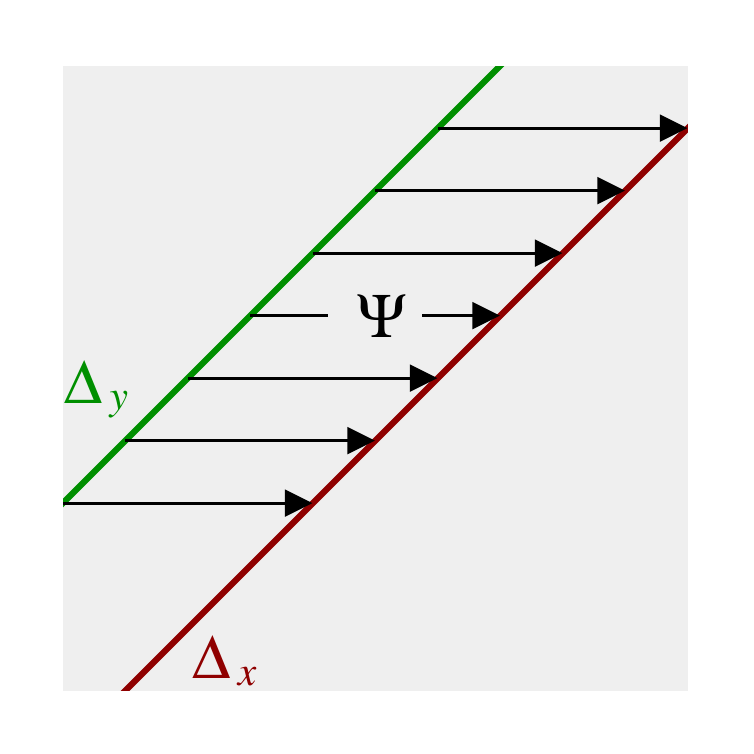}
\hfill{}
\caption{The maps $\Phi, \Psi$ recovered from the module~$\Ww$ over $\Delta_x \cup \Delta_y$.}
\label{fig:PhiPsi}
\end{figure}
Next, observe that the composition law implies all of the relations~\eqref{eq:int-relations}.

Finally, there is no additional information in~$\Ww$, beyond the interleaving maps and relations. Indeed, all remaining maps $w_S^T$, where $S \leq T$, can all be factored in the form:
\begin{alignat*}{2}
w_S^T &= v_{s+{y-x}}^{t} \circ \phi_s
	&&\quad\text{if $S \in \Delta_x$ and $T \in \Delta_y$,}
\\
w_S^T &= u_{s+{y-x}}^{t} \circ \psi_s
	&&\quad\text{if $S \in \Delta_y$ and $T \in \Delta_x$.}
\end{alignat*}
Thus each map in $\Ww$ is an instance of one of
\begin{alignat*}{2}
&
1_\Uu^\eta
&&\quad\text{from $\Delta_x$ to $\Delta_x$,}
\\
&
1_\Vv^\eta
&&\quad\text{from $\Delta_y$ to $\Delta_y$,}
\\
&
1_\Vv^\eta\Phi
&&\quad\text{from $\Delta_x$ to $\Delta_y$,}
\\
&
1_\Uu^\eta\Psi
&&\quad\text{from $\Delta_y$ to $\Delta_x$.}
\end{alignat*}
It is a simple matter to verify that the composition law is satisfied for each composable pair of maps. For instance
\[
(1_\Vv^\eta\Phi) (1_\Uu^\zeta \Psi)
=
1_\Vv^\eta\Phi 1_\Uu^\zeta \Psi
=
1_\Vv^\eta  1_\Vv^\zeta \Phi \Psi
=
1_\Vv^{\eta+\zeta} 1_\Vv^{2y-2x}
=
1_\Vv^{\eta+\zeta+2y-2x}.
\]
This can be done using only the known relations, so there are no further constraints on the~$w_S^T$.
\end{proof}

\begin{remark}
This characterisation makes it clear (or, depends on the fact) that all composable combinations of the maps $u, v, \phi, \psi$ from a given domain to a given codomain must be equal: indeed, they must agree with the appropriate map $w_S^T$ of $\Ww$.
\end{remark}

%--------------------------------------------------
\subsection{The interpolation lemma}
\label{subsec:interpolation}

In this section we give a proof of the following theorem, which first appeared in ~\cite{Chazal_CS_G_G_O_2009}. The result strikes us as somewhat surprising.

\begin{lemma}[Interpolation lemma]
\label{lem:interpolation}
Suppose $\Uu$, $\Vv$ are a $\delta$-interleaved pair of persistence modules. Then there exists a 1-parameter family of persistence modules
$
\left( \Uu_x \mid {x \in [0,\delta]} \right)
$
such that $\Uu_{0}, \Uu_{\delta}$ are equal to $\Uu, \Vv$ respectively, and $\Uu_x, \Uu_y$ are $|y-x|$-interleaved for all $x, y \in [0,\delta]$.
\end{lemma}

We can make a sharper statement about what is proved. Given a specific pair of interleaving maps
\begin{align*}
\Phi &\in \Hom^\delta(\Uu,\Vv)
\\
\Psi &\in \Hom^\delta(\Vv,\Uu)
\end{align*}
the construction explicitly provides, for each $x < y$, a pair of interleaving maps
\begin{align*}
\Phi_x^y &\in \Hom^{y-x}(\Uu_x,\Uu_y)
\\
\Psi_y^x &\in \Hom^{y-x}(\Uu_y,\Uu_x)
\end{align*}
such that $\Phi_{0}^{\delta}, \Psi_{\delta}^{0}$ are equal to $\Phi, \Psi$ respectively, and moreover
\begin{align*}
\Phi_x^z &= \Phi_y^z  \Phi_x^y
\\
\Psi_z^x &= \Psi_y^x \Psi_z^y
\end{align*}
for all $x < y < z$.

In view of Proposition~\ref{prop:poset-interleaving}, this sharp form of the interpolation lemma can be restated as follows. (We have also replaced the interval $[0,\delta]$ with the more general $[x_0, x_1]$.)

\begin{theorem}
\label{thm:interpolation2}
Any persistence module $\Ww$ over $\Delta_{0} \cup \Delta_{\delta}$ extends to a persistence module $\overline\Ww$ over the diagonal strip
\[
\Delta_{[0,\delta]}
=
\left\{ (p,q) \mid 0 \leq q-p \leq 2\delta \right\}
\subset
\Rr^2.
\]
\end{theorem}

\begin{remark}
The extension is by no means unique.
%For instance, one can add a module that is supported in the interior of the strip.
\end{remark}

Let us clarify how Theorem~\ref{thm:interpolation2} implies Lemma~\ref{lem:interpolation}. If $\Uu$, $\Vv$  are $\delta$-interleaved, then there exists a persistence module $\Ww$ over $\Delta_{0} \cup \Delta_{\delta}$ such that $\Ww|_{\Delta_{0}} = \Uu$ and $\Ww|_{\Delta_{\delta}} = \Vv$. By Theorem~\ref{thm:interpolation2}, this extends to $\overline\Ww$ over the strip $\Delta_{[0,\delta]}$. If we define a 1-parameter family $\Uu_x = \overline\Ww|_{\Delta_x}$, then $\Uu_x, \Uu_y$ are $|x-y|$-interleaved for all $x,y \in [0,\delta]$.

The equivalence between Theorem~\ref{thm:interpolation2} and the sharp form of Lemma~\ref{lem:interpolation} is left to the reader.

\begin{proof}[Proof of Theorem~\ref{thm:interpolation2}]
In order to express the proof more symmetrically, it is convenient to replace the interval $[0,\delta]$ by the interval $[-1,1]$. This can be done by a rescaling and translation of the plane.
Accordingly, suppose we are given a persistence module $\Ww$ over $\Delta_{-1} \cup \Delta_{1}$.

Our strategy is to construct two persistence modules over the strip $\Delta_{[-1, 1]}$ and a module map between them. The image of this map is also a persistence module over the strip, and will be the required extension.

By Proposition~\ref{prop:poset-interleaving}, $\Ww$ provides $\Uu = \Ww|_{\Delta_{-1}}$ and $\Vv = \Ww|_{\Delta_{1}}$, which we can view as persistence modules over~$\Rr$,
\begin{align*}
U_t &= W_{(t+1,t-1)},
\\
V_t &= W_{(t-1, t+1)},
\end{align*}
as well as interleaving maps $\Phi \in \Hom^2(\Uu, \Vv)$ and $\Psi \in \Hom^2(\Vv, \Uu)$ of degree~2.

From $\Uu, \Vv$ we construct four persistence modules over $\Rr^2$:
\begin{alignat*}{3}
\aaa &= \Uu{[p-1]}
&
\quad \text{defined by} \quad A_{(p,q)} &= U_{p-1}
&
\quad \text{and} \quad a_{(p,q)}^{(r,s)} &= u_{p-1}^{r-1}
\\
\bbb &= \Vv{[q-1]}
&
\quad \text{defined by} \quad B_{(p,q)} &= V_{q-1}
&
\quad \text{and} \quad b_{(p,q)}^{(r,s)} &= v_{q-1}^{s-1}
\\
\ccc &= \Uu{[q+1]}
&
\quad \text{defined by} \quad C_{(p,q)} &= U_{q+1}
&
\quad \text{and} \quad c_{(p,q)}^{(r,s)} &= u_{q+1}^{s+1}
\\
\ddd &= \Vv{[p+1]}
&
\quad \text{defined by} \quad
D_{(p,q)} &= V_{p+1}
&
\quad \text{and} \quad d_{(p,q)}^{(r,s)} &= v_{p+1}^{r+1}
\end{alignat*}

Next, we construct four module maps.
\begin{alignat*}{2}
1_\Uu &: \aaa \to \ccc
&\quad \text{defined at $(p,q)$ to be} \quad
u_{p-1}^{q+1} &: U_{p-1} \to U_{q+1}
\\
\Phi &: \aaa \to \ddd
&\quad \text{defined at $(p,q)$ to be} \quad
\phi_{p-1} &: U_{p-1} \to V_{p+1}
\\
\Psi &: \bbb \to \ccc
&\quad \text{defined at $(p,q)$ to be} \quad
\psi_{q-1} &: V_{q-1} \to U_{q+1}
\\
1_\Vv &: \bbb \to \ddd
&\quad \text{defined at $(p,q)$ to be} \quad
v_{q-1}^{p+1} &: V_{q-1} \to V_{p+1}
\end{alignat*}
The maps $\Phi, \Psi$ are defined over the whole plane, whereas the map $1_\Uu$ is defined only where $p-1\leq q+1$, and the map $1_\Vv$ is defined only where $q-1 \leq p+1$. 
It follows that all four maps are defined in the region where
\[
-2 \leq q-p \leq 2,
\]
which is precisely the strip~$\Delta_{[-1,1]}$. Henceforth, we restrict to that strip.

It is easy to confirm that these are module maps. The required commutation relations involve composable combinations of the maps $u, v, \phi, \psi$, which always agree by the remark following Proposition~\ref{prop:poset-interleaving}.

Define $\Omega \in \Hom(\aaa\oplus\bbb, \ccc\oplus\ddd)$ by the 2-by-2 matrix
\[
\left[ \begin{array}{cc} 1_\Uu & \Psi \\ \Phi & 1_\Vv \end{array} \right]
\]
of module maps. Our claim is that $\overline\Ww = \img(\Omega)$ is the required extension. We may equivalently say that $\overline\Ww = \coimg(\Omega)$ is the required extension. (The difference is whether we think of $\overline\Ww$ as a submodule of $\ccc\oplus\ddd$ or as quotient of $\aaa\oplus\bbb$.)

{\small\bf Step 1:} $\overline\Ww|_{\Delta_{-1}}$ is isomorphic to~$\Uu$.

\begin{proof}
On $\Delta_{-1} = \{ (t+1,t-1) \}$, we have
\begin{align*}
(\aaa \oplus \bbb)_t &= U_t \oplus V_{t-2},
\\
(\ccc \oplus \ddd)_t &= U_t \oplus V_{t+2},
\end{align*}
and the homomorphism $\Omega|_{\Delta_{-1}}$ is given by
\[
\omega_t
=
\left[ \begin{array}{cc} u_t^t & \psi_{t-2} \\ \phi_{t} & v_{t-2}^{t+2} \end{array} \right].
\]
The key point is that this factorises (in a block-matrix sense it `has rank 1'). The factorisation can be written in matrix form
\[
\omega_t
%=
%\left[ \begin{array}{cc} u_t^t & \psi_{t-2} \\ \phi_{t} & v_{t-2}^{t+2} \end{array} \right]
=
\left[ \begin{array}{c} u_t^t \\ \phi_t \end{array} \right]	
\left[ \begin{array}{cc} u_t^t & \psi_{t-2} \end{array} \right]
\]
or as a diagram of $\Rr$-module maps 
\[
\begin{diagram}
\node{\Uu \oplus \Vv[t-2]}
   \arrow{e,t}{\Omega_1}
\node{\Uu}
   \arrow{e,t}{\Omega_2}
\node{\Uu \oplus \Vv[t+2]}
\end{diagram}
\]
where
\[
\Omega_1(\mathbf{u} \oplus \mathbf{v}) = \mathbf{u} + \Psi(\mathbf{v})
\quad\text{and}\quad
\Omega_2(\mathbf{u}) = \mathbf{u} \oplus \Phi(\mathbf{u}).
\]

Since $\Omega_1$ is surjective and $\Omega_2$ is injective (at all indices~$t$), it follows that
\[
\coimg(\Omega)
\cong
\Uu
\cong
\img(\Omega)
\]
where the isomorphisms are given by $\Omega_1, \Omega_2$ respectively.
\end{proof}

{\small\bf Step 2.} $\overline\Ww|_{\Delta_{1}}$ is isomorphic to~$\Vv$.

\begin{proof}
On $\Delta_{1} = \{ (t-1,t+1) \}$, we have
\begin{align*}
(\aaa \oplus \bbb)_t &= U_{t-2} \oplus V_{t},
\\
(\ccc \oplus \ddd)_t &= U_{t+2} \oplus V_{t},
\end{align*}
and
\[
\omega_t
=
\left[ \begin{array}{cc}
		u_{t-2}^{t+2} & \psi_{t} \\
		\phi_{t-2} & v_{t}^{t}
	\end{array} \right].
\]
Again $\Omega$ factorises. The factorisation can be written in matrix form
\[
\omega_t
%=
%\left[ \begin{array}{cc} u_t^t & \psi_{t-2} \\ \phi_{t} & v_{t-2}^{t+2} \end{array} \right]
=
\left[ \begin{array}{c} \psi_t \\ v_t^t \end{array} \right]	
\left[ \begin{array}{cc} \phi_{t-2} & v_t^t \end{array} \right]
\]
or as a diagram of $\Rr$-module maps 
\[
\begin{diagram}
\node{\Uu[t-2] \oplus \Vv}
   \arrow{e,t}{\Omega_3}
\node{\Vv}
   \arrow{e,t}{\Omega_4}
\node{\Uu[t+2] \oplus \Vv}
\end{diagram}
\]
where
\[
\Omega_3(\mathbf{u} \oplus \mathbf{v}) = \Phi(\mathbf{u}) + \mathbf{v}
\quad\text{and}\quad
\Omega_4(\mathbf{v}) = \Psi(\mathbf{v}) \oplus \mathbf{v}.
\]
Since $\Omega_3$ is surjective and $\Omega_4$ is injective, we have
\[
\coimg(\Omega)
\cong
\Vv
\cong
\img(\Omega)
\]
where the isomorphisms are given by $\Omega_3, \Omega_4$ respectively.
\end{proof}

\medskip
{\small\bf Step 3.} Under the isomorphisms above, the cross maps of $\overline\Ww$ are precisely $\Phi$ and~$\Psi$.

\begin{proof}
The cross maps for $\overline\Ww$ are induced by the cross maps for $\aaa\oplus\bbb$ and also by the cross maps for $\ccc\oplus\ddd$. (The result must be the same, because $\Omega$ is a module homomorphism.)

The following diagram shows the vertical maps for $\aaa\oplus\bbb$ (on the left) and $\ccc\oplus\ddd$ (on the right), as well as the factorisations of Steps 1 and~2.
\[
\begin{diagram}
\node{\Uu \oplus \Vv[t-2]}
   \arrow{e,t}{\Omega_1}
   \arrow{s,l}{1_\Uu \,\oplus\, 1_\Vv^4}
\node{\Uu}
   \arrow{e,t}{\Omega_2}
   \arrow{s,l}{\Phi}
\node{\Uu \oplus \Vv[t+2]}
   \arrow{s,r}{1_\Uu^4\,\oplus\,1_\Vv}
\\
\node{\Uu \oplus \Vv[t+2]}
   \arrow{e,t}{\Omega_3}
\node{\Vv[t+2]}
   \arrow{e,t}{\Omega_4}
\node{\Uu[t+4] \oplus \Vv[t+2]}
\end{diagram}
\]
The diagram, with $\Phi$ in the middle, is easy verified to commute using the explicit forms of the maps $\Omega_i$. It follows that $\Phi$ is the induced vertical cross map for~$\overline\Ww$.

A similar argument using the diagram
\[
\begin{diagram}
\node{\Uu[t-2] \oplus \Vv}
   \arrow{e,t}{\Omega_3}
   \arrow{s,l}{1_\Uu^4 \,\oplus\, 1_\Vv}
\node{\Vv}
   \arrow{e,t}{\Omega_4}
   \arrow{s,l}{\Psi}
\node{\Uu[t+2] \oplus \Vv}
   \arrow{s,r}{1_\Uu\,\oplus\,1_\Vv^4}
\\
\node{\Uu[t+2] \oplus \Vv}
   \arrow{e,t}{\Omega_1}
\node{\Uu[t+2]}
   \arrow{e,t}{\Omega_2}
\node{\Uu[t+2] \oplus \Vv[t+4]}
\end{diagram}
\]
shows that $\Psi$ is the induced horizontal cross map.
\end{proof}

\medskip
This completes the proof of Theorem~\ref{thm:interpolation2}.
\end{proof}

We note that Step~3 isn't necessary to deduce the interpolation lemma (\ref{lem:interpolation}), because it is sufficient to demonstrate that $\Uu, \Vv$ belong to some 1-parameter family of persistence modules. It is important for the interpolation lemma that the $|x-y|$-interleavings exist between each pair $\Uu_x, \Uu_y$, but it doesn't matter exactly what they are or how they relate to each other.

We can interpret the stronger Theorem~\ref{thm:interpolation2} as a functor extension theorem. If we regard the posets $\Delta_{x_0} \cup \Delta_{x_1}$ and $\Delta_{[x_0,x_1]}$ as categories, then persistence modules over these posets are the same as functors to the category of vector spaces. (See the remark near the beginning of section~\ref{subsec:11}.)
The theorem asserts the existence of an extension $\overline\Ww$
\[
\begin{diagram}
\node{\Delta_{[x_0,x_1]}}
   \arrow{se,t,..}{\overline\Ww}
\\
\node{\Delta_{x_0} \cup \Delta_{x_1}}
   \arrow{n}
   \arrow{e,t}{\Ww}
\node{\mathrm{Vect}}
\end{diagram}
\]
for any functor $\Ww$.

%--------------------------------------------------
\subsection{The interpolation lemma (continued)}
\label{subsec:interpolation2}

In this optional section, we study the interpolation process in greater depth.
Given two persistence modules $\Uu, \Vv$ and a $\delta$-interleaving between them, there are at least three natural ways to construct an interpolating family. We describe the three methods and find some relationships between them.

As in the proof of Theorem~\ref{thm:interpolation2}, we may suppose that $\delta = 2$ and that $\Uu, \Vv$ and their interleaving are represented as a module over $\Delta_{-1} \cup \Delta_{1}$ in the plane. To extend this module over the strip
\[
\Delta_{[-1,1]}
=
\left\{
(p,q) \in \Rr^2 \mid -2 \leq p-q \leq 2
\right\}
\]
we consider the sequence
\[
\tag{\S}\label{eqn:Omega}
\begin{diagram}
\node{\begin{array}{c} \Uu[q-3] \\\oplus\\ \Vv[p-3] \end{array}}
	\arrow{e,t}{\Omega'}
\node{\begin{array}{c} \Uu[p-1] \\\oplus\\ \Vv[q-1] \end{array}}
	\arrow{e,t}{\Omega}
\node{\begin{array}{c} \Uu[q+1] \\\oplus\\ \Vv[p+1] \end{array}}
	\arrow{e,t}{\Omega''}
\node{\begin{array}{c} \Uu[p+3] \\\oplus\\ \Vv[q+3] \end{array}}
\end{diagram}
\]
defined over $\Delta_{[-1,1]}$ with maps
\[
\Omega' = 
\left[ \begin{array}{rr} 1_\Uu & -\Psi \\ -\Phi & 1_\Vv \end{array} \right]
,
\quad
\Omega =
\left[ \begin{array}{cc} 1_\Uu & \Psi \\ \Phi & 1_\Vv \end{array} \right]
,
\quad
\Omega'' =
\left[ \begin{array}{rr} 1_\Uu & -\Psi \\ -\Phi & 1_\Vv \end{array} \right]
\]
taken with the appropriate shift values for $1_\Uu$ and $1_\Vv$.

Notice that $\Omega$, $\Omega'$ and $\Omega''$ are essentially the same map. Certainly $\Omega', \Omega''$ are formally identical, up to a translation $\tau$ of the strip. In fact, each of the modules in the sequence is related to the next by an isomorphism $\sigma$ which changes the sign of the $\Vv$-term and transforms indices by $(p,q) \mapsto (q+2,p+2)$. We have $\tau = \sigma^2$, and conjugacies $\Omega = \sigma \Omega' \sigma^{-1}$ and $\Omega'' = \sigma \Omega \sigma^{-1}$.

\begin{proposition}
Each of the three modules
\[
\coker(\Omega'),
\quad
\coimg(\Omega) = \img(\Omega),
\quad
\ker(\Omega'')
\]
over $\Delta_{[-1,1]}$ defines an interpolating family between $\Uu,\Vv$.
\end{proposition}

\begin{proof}
We already know this for $\coimg(\Omega) = \img(\Omega)$ from the proof of Theorem~\ref{thm:interpolation2}.
Now we outline the proof that $\coker(\Omega')$ and $\ker(\Omega'')$ restrict on $\Delta_{-1}$ to modules isomorphic to~$\Uu$.

The diagonal $\Delta_{-1}$ is defined by $(p,q) = (t+1,t-1)$ and so the sequence~\eqref{eqn:Omega} restricts to:
\[
\begin{diagram}
\node{\begin{array}{c} \Uu[t-4] \\\oplus\\ \Vv[t-2] \end{array}}
	\arrow{e,t}{\Omega'}
\node{\begin{array}{c} \Uu \\\oplus\\ \Vv[t-2] \end{array}}
	\arrow{e,t}{\Omega}
\node{\begin{array}{c} \Uu \\\oplus\\ \Vv[t+2] \end{array}}
	\arrow{e,t}{\Omega''}
\node{\begin{array}{c} \Uu[t+4] \\\oplus\\ \Vv[t+2] \end{array}}
\end{diagram}
\]
We have a factorisation
\[
\Omega'
\; \text{or} \;
\Omega''
=
\left[ \begin{array}{rr} 1^4_\Uu & -\Psi \\ -\Phi & 1_\Vv \end{array} \right]
=
\left[ \begin{array}{r} -\Psi \\ 1_\Vv \end{array} \right]
\left[ \begin{array}{rr} -\Phi & 1_\Vv \end{array} \right]
= 
\Omega'_1 \Omega'_2
\; \text{or} \;
\Omega''_1 \Omega''_2
\]
which reveals that $\img(\Omega') = \img(\Omega'_1)$ is a complementary submodule to $\Uu \oplus 0$ in $\Uu \oplus \Vv[t-2]$, and that $\ker(\Omega'') = \ker(\Omega''_2)$ is a complementary submodule to $0 \oplus \Vv[t+2]$ in $\Uu \oplus \Vv[t+2]$.
It follows that $\coker(\Omega')$ and $\ker(\Omega'')$ are each isomorphic to~$\Uu$.

By a symmetric argument, the restriction of each module to~$\Delta_1$  is isomorphic to~$\Vv$. This completes the proof that $\coker(\Omega')$ and $\ker(\Omega'')$ interpolate between $\Uu$ and~$\Vv$.
\end{proof}

Which of the three constructions should we prefer? Notice that $\Omega \Omega' = 0$ and $\Omega'' \Omega = 0$, meaning that~\eqref{eqn:Omega} is a chain complex. It follows that there is a natural projection and a natural inclusion
\[
\coker(\Omega') \twoheadrightarrow \coimg(\Omega)
=
\img(\Omega) \hookrightarrow \ker(\Omega'')
\]
by which we see that $\coimg(\Omega) = \img(\Omega)$ is structurally the simplest of the three.
%

%\begin{remark}
The surplus information in the other two interpolations may be measured as the kernel of the projection and the cokernel of the inclusion. These are precisely the homology at the second and third terms of~\eqref{eqn:Omega}. It follows from the conjugacies described above that the two homology modules are isomorphic upon translating the strip by~2 and interchanging $p$ and~$q$ (i.e.\ reversing the interpolation parameter).
%\end{remark}

We finish this section by using the `vineyard' technique of~\cite{CohenSteiner_E_M_2006} to visualise the 1-parameter family of persistence modules produced by each of the three constructions. We obtained the vineyards by sketching the supports of the eight module summands in~\eqref{eqn:Omega} and using the sketches to partition the interpolation parameter range $[-1,1]$ into suitable intervals for case splitting. It is perhaps easier done than described, so we invite the reader to conduct their own calculations and confirm that our vineyards are correct. As further corroboration, one verifies that the homology modules are isomorphic in the sense described above.

In Figure~\ref{fig:interpol2}, we consider the canonical 2-interleaving between interval modules $\Ii^{[0,4)}$ and $\Ii^{[1,6)}$. The thick black lines show how the points of the persistence diagram travel in the plane as we proceed along the interpolating family, for each of the three constructions. Each point travels with speed~1 and traverses a path of length~2 (in the $\ellinf$-metric). The cokernel interpolation has an extra `ghost' summand which emerges from the diagonal at $(3,3)$ at the beginning of the interpolation, and is reabsorbed by the diagonal at $(2,2)$ at the end.
\begin{figure}
{}
\hfill
\includegraphics[scale=0.4]{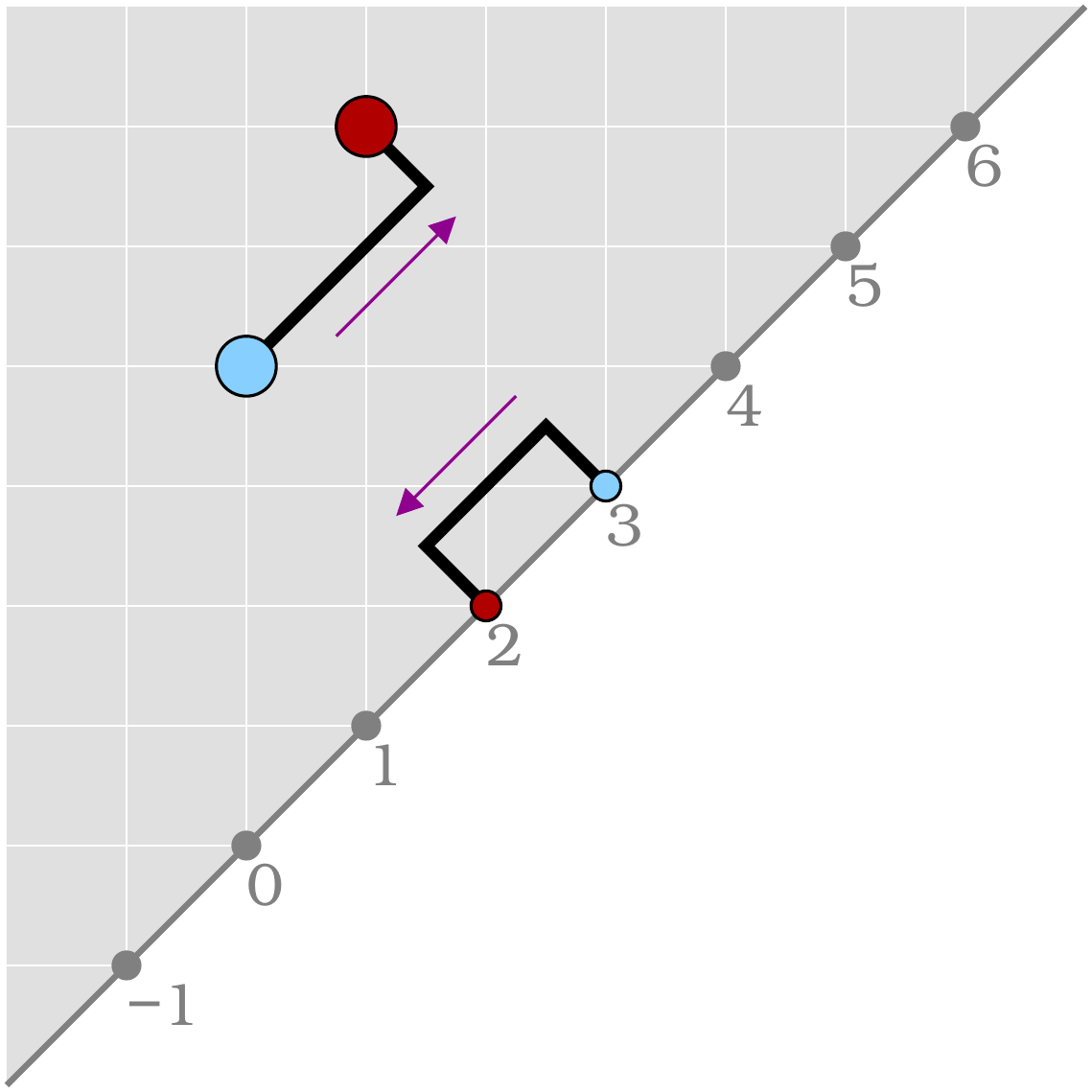}
\hfill
\includegraphics[scale=0.4]{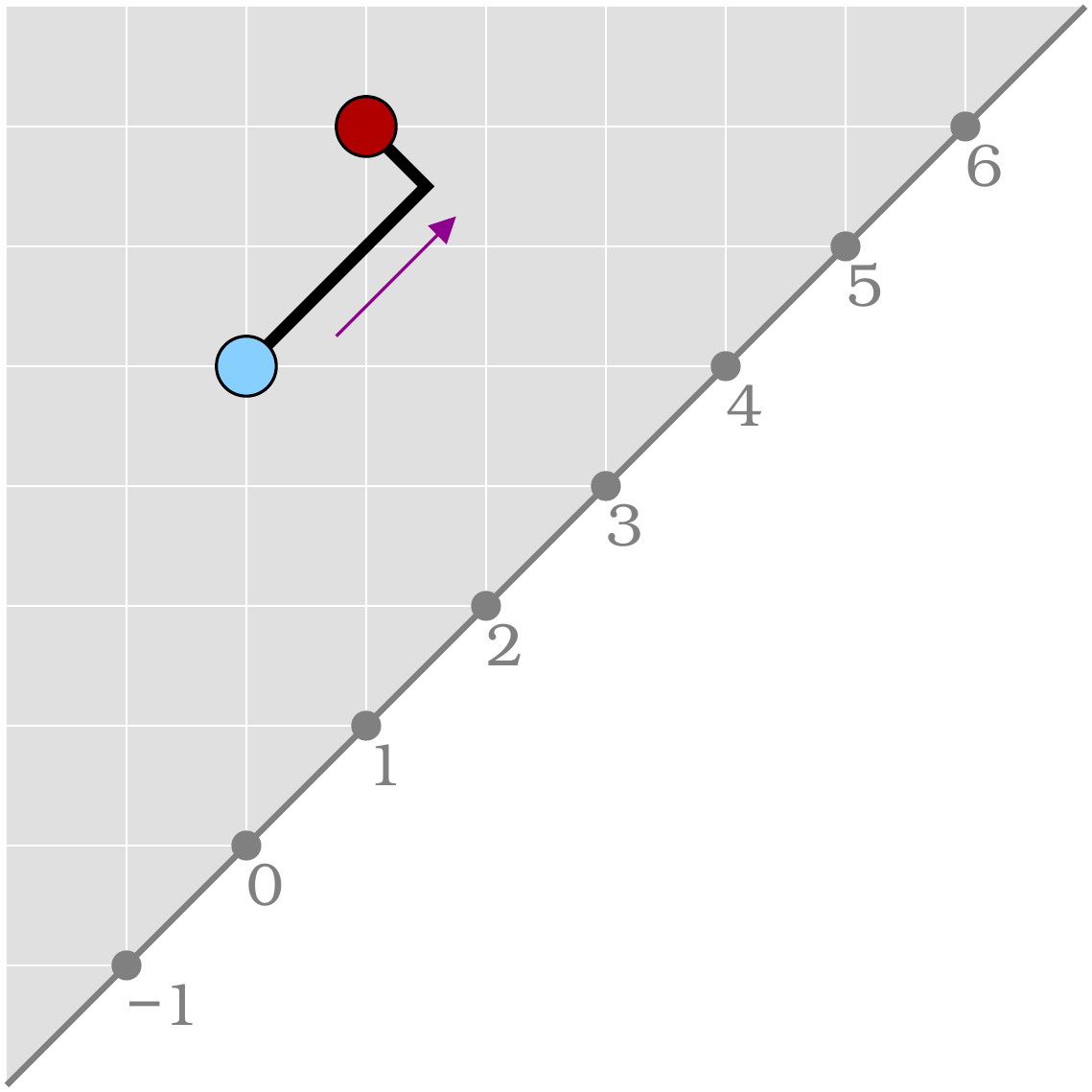}
\hfill
\includegraphics[scale=0.4]{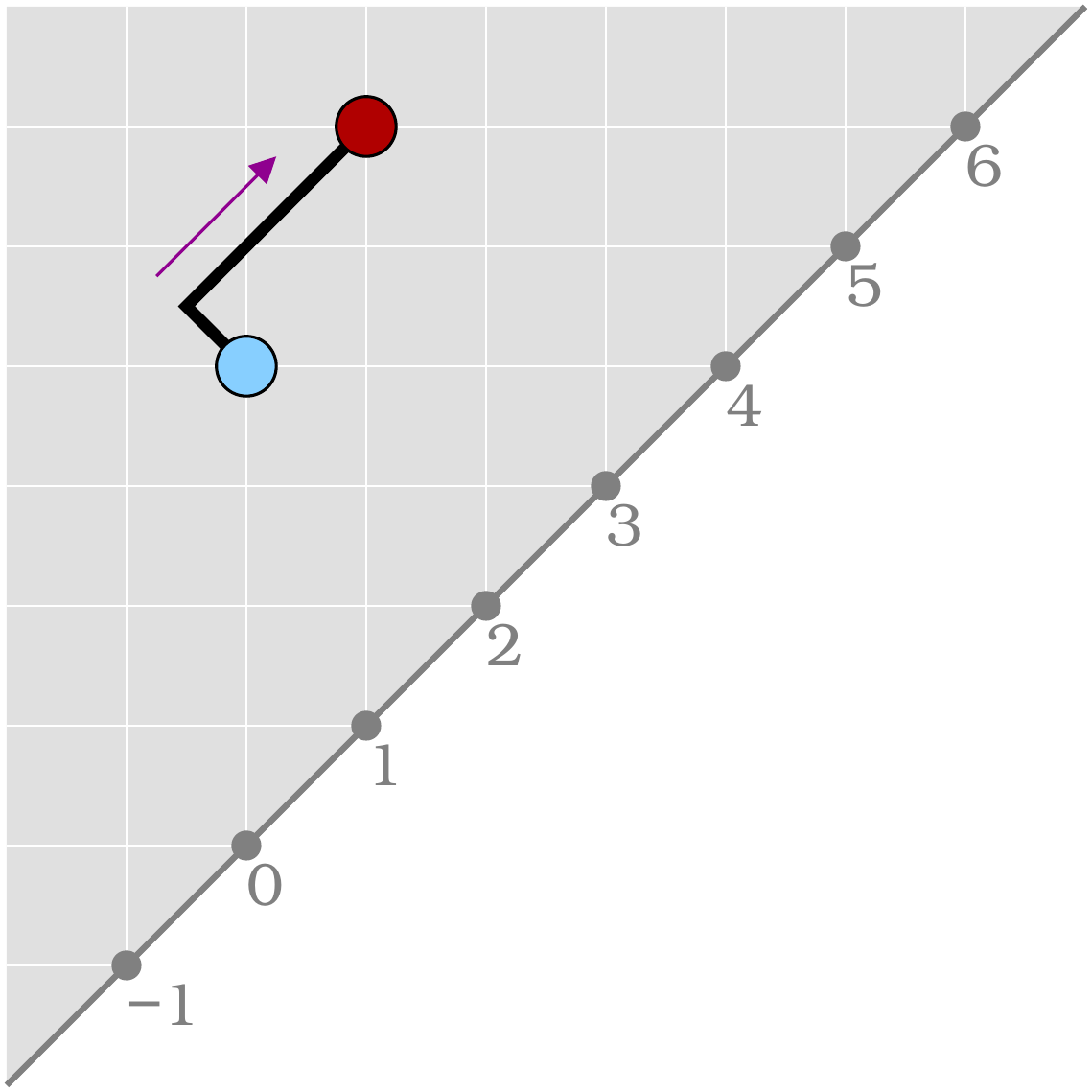}
\hfill
{}
\caption{Vineyards: cokernel (left), image (middle), and kernel (right) interpolations for the 2-interleaving between $\Ii^{[0,4)}$ and $\Ii^{[1,6)}$.}
\label{fig:interpol2}
\end{figure}

Figure~\ref{fig:interpol3} shows how the interpolation changes when we use a different interleaving parameter. For instance, the modules $\Ii^{[0,4)}$ and $\Ii^{[1,6)}$ are also 3-interleaved, since $3 \geq 2$. From the canonical 3-interleaving we obtain the vineyards shown in the figure. Now each point-trajectory has length~3. This time, both the cokernel and kernel interpolations produce `ghosts' at the diagonal.
\begin{figure}
{}
\hfill
\includegraphics[scale=0.4]{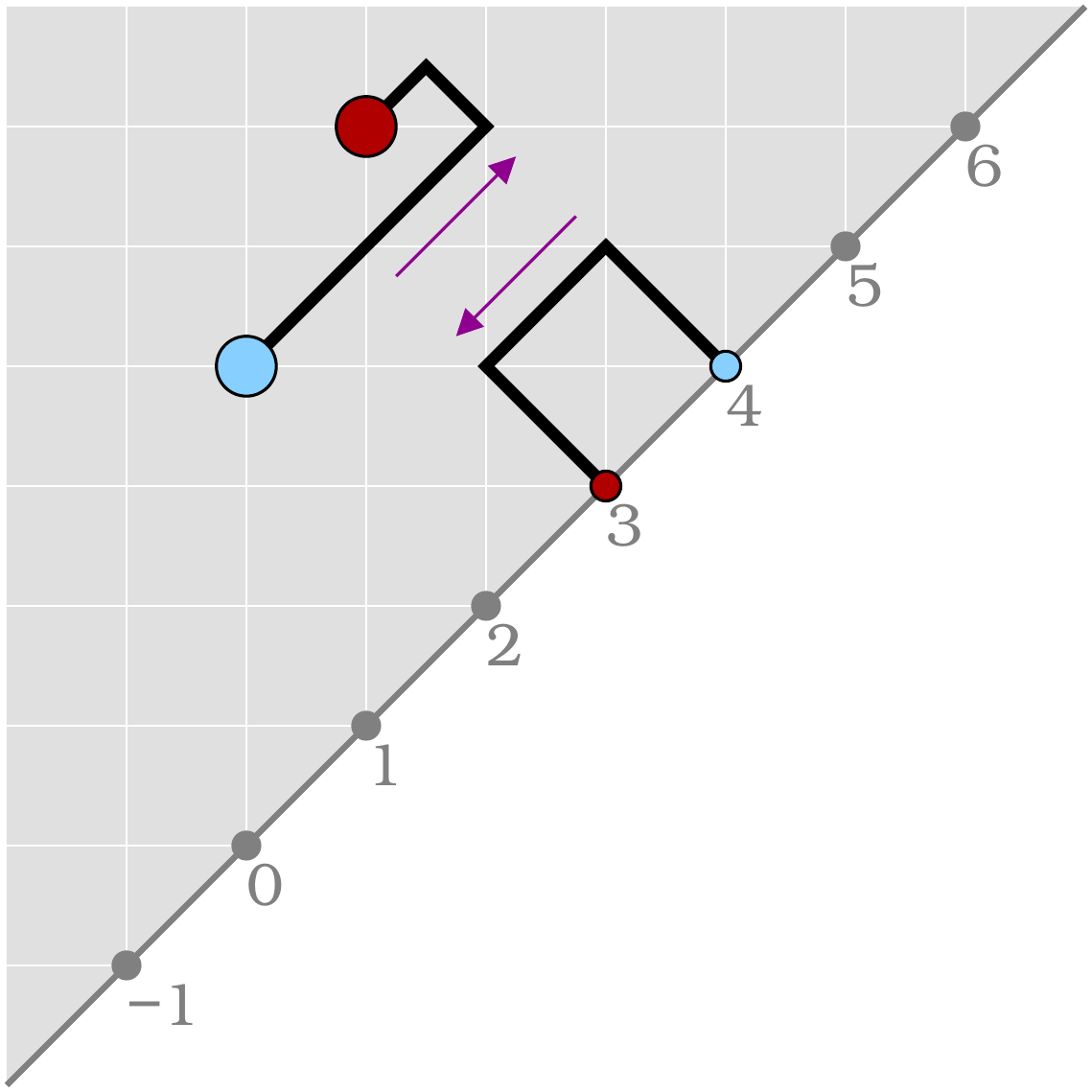}
\hfill
\includegraphics[scale=0.4]{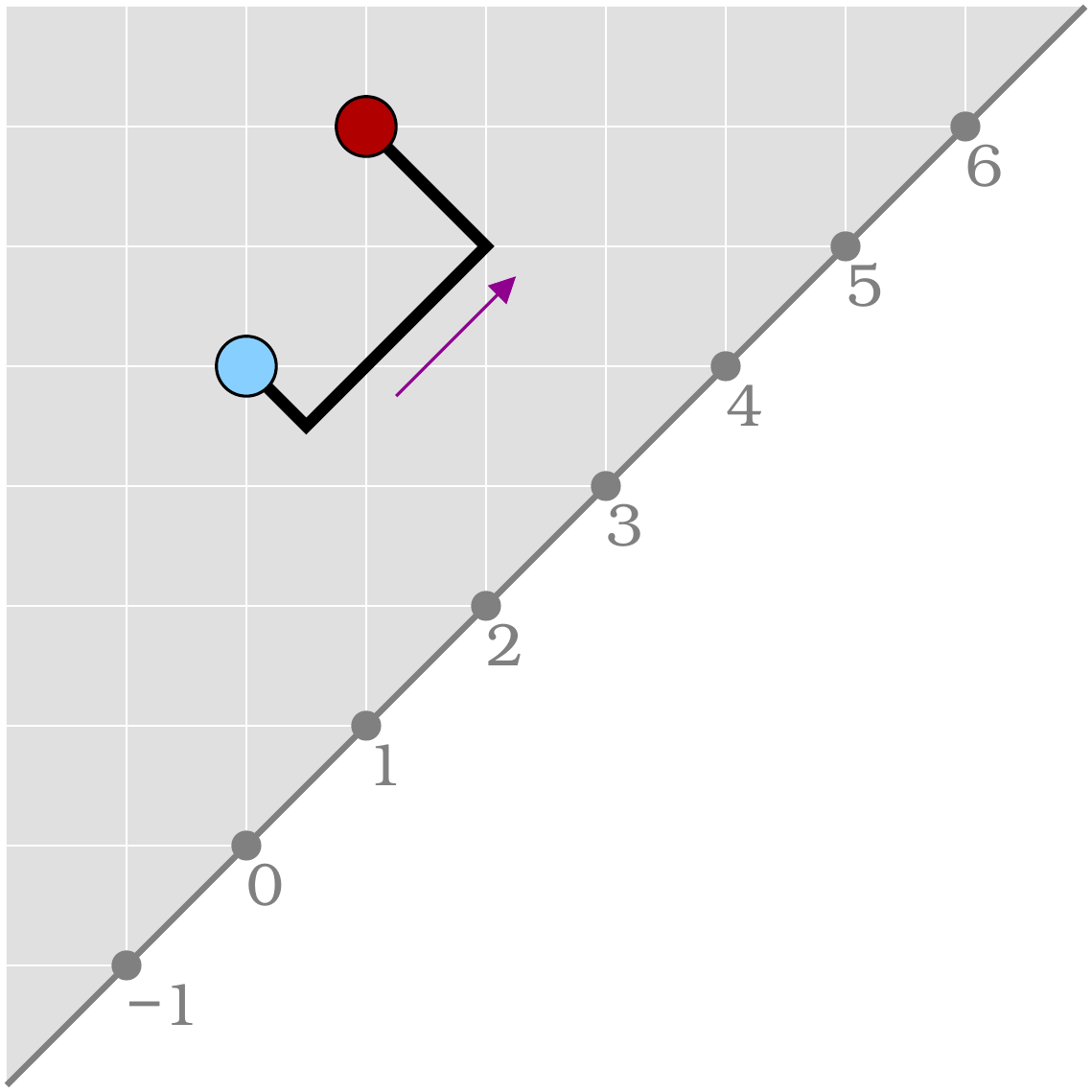}
\hfill
\includegraphics[scale=0.4]{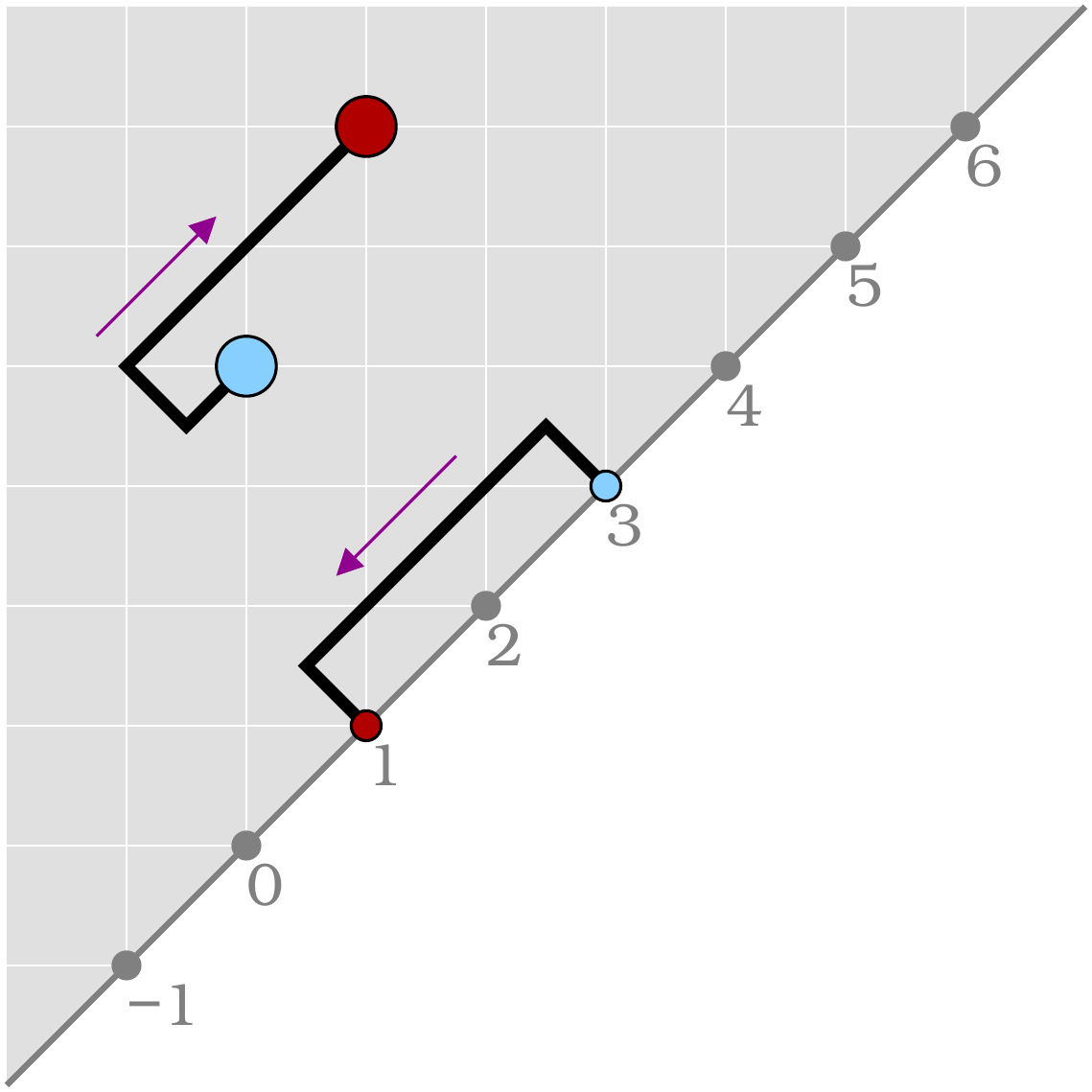}
\hfill
{}
\caption{Vineyards: cokernel (left), image (middle), and kernel (right) interpolations for the 3-interleaving between $\Ii^{[0,4)}$ and $\Ii^{[1,6)}$.}
\label{fig:interpol3}
\end{figure}

%-------------------------------------------------------------------
\section{The Isometry Theorem}
\label{sec:isometry}
%--------------------------------------------------

In this section we discuss the metric relationship between persistence modules and their persistence diagrams.
As in section~\ref{sec:interleaving}, all persistence modules are indexed by~$\Rr$ unless explicitly stated otherwise.

The principal result is the famous stability theorem of Cohen--Steiner, Edelsbrunner and Harer~\cite{CohenSteiner_E_H_2007}, in the generality established by~\cite{Chazal_CS_G_G_O_2008}. The main difference is that we emphasise persistence measures, deriving the standard theorem from a more general statement about measures. The structure of the proof remains the same as in~\cite{CohenSteiner_E_H_2007}.

The secondary result is the converse inequality, which together with the stability theorem implies that the space of q-tame persistence modules is isometric with the space of locally finite persistence diagrams. This isometry theorem appeared originally in the work of Lesnick~\cite{Lesnick_2011} for modules which satisfy $\dim(V_t) < \infty$ for all~$t$, and independently in the work of Bubenik and Scott~\cite{Bubenik_Scott_2012} for modules of finite type.

%--------------------------------------------------
\subsection{The interleaving distance}

In this section we define the interleaving distance between persistence modules. This was introduced in~\cite{Chazal_CS_G_G_O_2008}.

The first observation is that if $\Uu$ and $\Vv$ are $\delta$-interleaved, then they are $(\delta + \epsilon)$-interleaved for every $\epsilon > 0$. Indeed, the maps
\begin{align*}
\Phi'
&= \Phi 1_\Uu^\epsilon = 1_\Vv^\epsilon \Phi
\\
\Psi'
&= \Psi 1_\Vv^\epsilon = 1_\Uu^\epsilon \Psi
\end{align*}
provide the required interleaving.

The challenge, then, if two persistence modules are interleaved, is to make the interleaving parameter as small as possible. The minimum is not necessarily attained, so we introduce some additional terminology.

We say that two persistence modules $\Uu, \Vv$ are \textbf{$\delta^+$-interleaved} if they are $(\delta+\epsilon)$-interleaved for all $\epsilon > 0$. As we shall see, this does not imply that $\Uu, \Vv$ are $\delta$-interleaved.

\begin{example}
Two persistence modules are 0-interleaved if and only if they are isomorphic.
\end{example}

\begin{example}
A persistence module $\Vv$ is \textbf{ephemeral} if $v_s^t = 0$ for all $s < t$. The spaces $V_t$ can be completely arbitrary.
Let $\Uu$ and $\Vv$ be two non-isomorphic ephemeral modules. Then $\Uu, \Vv$ are $0^+$-interleaved but not $0$-interleaved. Indeed, $1_\Uu^{2\epsilon} = 0$ and $1_\Vv^{2\epsilon} = 0$ for all $\epsilon > 0$, so the zero maps
\begin{align*}
\Phi = 0 &\in \Hom^\epsilon(\Uu, \Vv)
\\
\Psi = 0 &\in \Hom^\epsilon(\Vv, \Uu)
\end{align*}
constitute an $\epsilon$-interleaving.
\end{example}

The \textbf{interleaving distance} between two persistence modules is defined as follows:
\begin{align*}
\inter(\Uu, \Vv)
  &= \inf \{ \delta \mid \text{$\Uu, \Vv$ are $\delta$-interleaved} \}
\\
  &= \min \{ \delta \mid \text{$\Uu, \Vv$ are $\delta^+$-interleaved} \}
\end{align*}
If there is no $\delta$-interleaving between $\Uu, \Vv$ for any value of~$\delta$, then $\inter(\Uu, \Vv) = \infty$.

\begin{proposition}
The interleaving distance satisfies the triangle inequality:
\[
\inter(\Uu, \Ww) \leq \inter(\Uu, \Vv) + \inter(\Vv, \Ww)
\]
for any three persistence modules $\Uu, \Vv, \Ww$.
\end{proposition}

\begin{proof}
Given a $\delta_1$-interleaving between $\Uu, \Vv$ and a $\delta_2$-interleaving between $\Vv, \Ww$ one can construct a $\delta = (\delta_1 + \delta_2)$-interleaving between $\Uu, \Ww$ by composing the interleaving maps:
\begin{alignat*}{2}
\Uu
\stackrel{\Phi_1}{\longrightarrow} &\Vv
\stackrel{\Phi_2}{\longrightarrow} &&\Ww
\\
\Uu
\stackrel{\Psi_1}{\longleftarrow} &\Vv
\stackrel{\Psi_2}{\longleftarrow} &&\Ww
\end{alignat*}
One easily verifies that $\Phi = \Phi_2\Phi_1$ and $\Psi = \Psi_1\Psi_2$ are interleaving maps. Explicitly:
\begin{alignat*}{5}
\Psi \Phi
	&= \Psi_1 \Psi_2 \Phi_2 \Phi_1
	&&= \Psi_1 1_\Vv^{2\delta_2} \Phi_1
	&&= \Psi_1 \Phi_1 1_\Uu^{2\delta_2}
	&&= 1_\Uu^{2\delta_1} 1_\Uu^{2\delta_2}
	&&= 1_\Uu^{2\delta}
\\
\Phi \Psi
	&= \Phi_2 \Phi_1 \Psi_1 \Psi_2
	&&= \Phi_2 1_\Vv^{2\delta_1} \Psi_2
	&&= \Phi_2 \Psi_2 1_\Ww^{2\delta_1}
	&&= 1_\Ww^{2\delta_2} 1_\Ww^{2\delta_2}
	&&= 1_\Ww^{2\delta}
\end{alignat*}
Now take the infimum over $\delta_1, \delta_2$.
\end{proof}

The proposition tells us that $\inter$ is a pseudometric. It is not a true metric because $\inter(\Uu,\Vv) = 0$ does not imply $\Uu \cong \Vv$, as we saw above. In fact, two q-tame persistence modules have interleaving distance~0 if and only if their undecorated persistence diagrams are the same. This is a consequence of the isometry theorem.

Here is the simplest instance. The direct proof is left to the reader (or see Proposition~\ref{prop:pqrs}).

\begin{example}
The four interval modules
\[
\Ii{[p,q]},\; \Ii{[p,q)},\; \Ii{(p,q]},\; \Ii{(p,q)}
\]
are $0^+$-interleaved but not isomorphic.
\end{example}

The following property of interleaving distance will be useful later.

\begin{proposition}
\label{prop:inter-sum}
Let $\Uu_1, \Uu_2, \Vv_1, \Vv_2$ be persistence modules. Then
\[
\inter(\Uu_1 \oplus \Uu_2, \Vv_1 \oplus \Vv_2)
\leq \max\left( \inter(\Uu_1, \Vv_1), \inter(\Uu_2, \Vv_2) \right)
\]
More generally, let $\left( \Uu_\ell \mid \ell \in L \right)$ and $\left( \Vv_\ell \mid \ell \in L \right)$ be families of persistence modules indexed by the same set~$L$, and let
\[
\Uu = \bigoplus_{\ell \in L} \Uu_\ell,
\quad
\Vv = \bigoplus_{\ell \in L} \Vv_\ell.
\]
Then
\[
\inter(\Uu, \Vv) \leq \sup \left( \inter(\Uu_\ell, \Vv_\ell) \mid \ell \in L \right).
\]
\end{proposition}

\begin{proof}
Given $\delta$-interleavings $\Phi_\ell, \Psi_\ell$ for each pair $\Uu_\ell, \Vv_\ell$, the direct sum maps $\Phi = \bigoplus \Phi_\ell$, $\Psi = \bigoplus \Psi_\ell$ constitute a $\delta$-interleaving of $\Uu, \Vv$. Thus any upper bound on the $\inter(\Uu_\ell, \Vv_\ell)$ is an upper bound for $\inter(\Uu, \Vv)$. In particular, this is true for the least upper bound, or $\sup$.
\end{proof}

%--------------------------------------------------
\subsection{The bottleneck distance}

Now we define the metric on the other side of the isometry theorem, namely the bottleneck distance between undecorated persistence diagrams.
For a q-tame persistence module~$\Vv$, every rectangle not touching the diagonal has finite $\mu_\Vv$-measure. This implies that the undecorated diagram
\[
\dgm(\Vv) = \dgm(\mu_\Vv)
\]
is a multiset in the extended open half-plane
\[
\UpperInt =
\{ (p,q) \mid -\infty \leq p < q \leq +\infty \}.
\]
In order to define the bottleneck distance, we need to specify the distance between any pair of points in~$\UpperInt$, as well as the distance between any point and the diagonal  (the boundary of the half-plane). These distance functions are not arbitrary; they are defined as they are because of the interleaving properties of interval modules.

\medskip
\quad
{\bf\small (point to point)}:
The first idea is that two undecorated diagrams are close if there is a bijection between them which doesn't move any point too far. We use the $\ell^\infty$-metric in the plane:
\[
\ellinf((p,q), (r,s)) = \max\left( |p-r|, |q-s| \right)
\]
Points at infinity are compared in the expected way:
\begin{align*}
\ellinf((-\infty, q), (-\infty, s)) &= |q-s|,
\\
\ellinf((p,+\infty), (r,+\infty)) &= |p-r|,
\end{align*}
and
\[
\ellinf((-\infty, +\infty), (-\infty, +\infty)) = 0.
\]
Distances between points in different strata (e.g.\ between $(p,q)$ and $(-\infty,s)$) are infinite.

The next result gives a relationship between the $\ell^\infty$-metric and the interleaving of interval modules.

\begin{proposition}
\label{prop:pqrs}
Let $\lgroup p^*,q^* \rgroup$ and $\lgroup r^*,s^* \rgroup$ be intervals (possibly infinite), and let
\[
\Uu = \Ii{\lgroup p^*,q^* \rgroup}
\quad\text{and}\quad
\Vv = \Ii{\lgroup r^*,s^* \rgroup}
\]
be the corresponding interval modules.
Then
\[
\inter(\Uu, \Vv) \leq \ellinf((p,q), (r,s)).
\]
\end{proposition}

The proof is postponed to the end of the section. We remark that equality holds provided that the intervals overlap sufficiently (the closure of each interval must meet the midpoint of the other), so the proposition is tight in that sense.

\medskip
\quad
{\bf\small (point to diagonal)}:
The next idea is that points which are close to the diagonal are allowed to be swallowed up by the diagonal. Again we use the $\ell^\infty$-metric:
\[
\ellinf((p,q), \Delta) = \textstyle\half(q-p)
\]

Again this is related to the behaviour of interval modules.

\begin{proposition}
\label{prop:pqD}
Let $\lgroup p^*,q^* \rgroup$ be an interval, let
\[
\Uu = \Ii{\lgroup p^*,q^* \rgroup},
\]
be the corresponding interval module, and let $0$ denote the zero persistence module.
Then
\[
\inter(\Uu, 0) = \textstyle\half(q-p).
\]
(This is infinite if the interval is infinite.)
\end{proposition}

\begin{proof}%[Proof of Proposition~\ref{prop:pqD}]
Let $\delta \geq 0$. When is there a $\delta$-interleaving? The interleaving maps must be zero (no other maps exist to or from the module~0), so the only condition that needs checking is $\Psi\Phi = 1_\Uu^{2\delta}$, which is really $0 = 1_\Uu^{2\delta}$. This holds when $\delta > \half(q-p)$ and fails when $\delta < \half(q-p)$.
\end{proof}

\medskip
We now use these two concepts to define the bottleneck distance between two multisets $\Aa, \Bb$ in the extended half-plane.

It is easier to work with sets rather than multisets. One way to do this is to attach labels to distinguish multiple instances of each repeated point. For instance, $\alpha$ with multiplicity~$k$ becomes $\alpha_1, \dots, \alpha_k$. Henceforth we will do this implicitly, without comment.

A {\bf partial matching} between $\Aa$ and~$\Bb$ is a collection of pairs
\[
\Mm \subset \Aa \times \Bb
\]
such that:

\begin{vlist}
\item
for every $\alpha \in \Aa$ there is at most one $\beta \in \Bb$ such that $(\alpha, \beta) \in \Mm$;

\item
for every $\beta \in \Bb$ there is at most one $\alpha \in \Aa$ such that $(\alpha, \beta) \in \Mm$.
\end{vlist}

We say that a partial matching $\Mm$ is a {\bf $\delta$-matching} if all of the following are true:
\begin{vlist}
\item
if $(\alpha, \beta) \in \Mm$ then $\ellinf(\alpha,\beta) \leq \delta$;
\item
if $\alpha \in \Aa$ is unmatched then $\ellinf(\alpha, \Delta) \leq \delta$;
\item
if $\beta \in \Bb$ is unmatched then $\ellinf(\beta, \Delta) \leq \delta$.
\end{vlist}

The {\bf bottleneck distance} between two multisets $\Aa, \Bb$ in the extended half-plane is
\[
\bottle(\Aa, \Bb)
=
\inf \left( \delta \mid \text{there exists a $\delta$-matching between $\Aa$ and $\Bb$} \right).
\]
In section~\ref{subsec:bottle-compact}, we will show that `$\inf$' can be replaced by `$\min$' if $\Aa, \Bb$ are locally finite.

\begin{remark}
In order for $\bottle(\Aa,\Bb) < \infty$, it is necessary that the cardinalities of $\Aa, \Bb$ agree over each of the three strata at infinity:
\begin{align*}
\card(\Aa|_{\{-\infty\}\times\Rr})
	&= \card(\Bb|_{\{-\infty\}\times\Rr})
\\
\card(\Aa|_{\Rr \times \{+\infty\}})
	&= \card(\Bb|_{\Rr \times \{+\infty\}})
\\
\card(\Aa|_{\{-\infty\}\times\{+\infty\}})
	&= \card(\Bb|_{\{-\infty\}\times\{+\infty\}})
\end{align*}
Indeed, these points have infinite distance from the diagonal and from points in the other strata, and therefore they must be bijectively matched within each stratum.
\end{remark}

\begin{proposition}
\label{prop:bottle-triangle}
The bottleneck distance satisfies the triangle inequality:
\[
\bottle(\Aa, \Cc)
\leq 
\bottle(\Aa, \Bb) + \bottle(\Bb, \Cc) 
\]
for any three multisets $\Aa, \Bb, \Cc$.
\end{proposition}

\begin{proof}
Suppose $\Mm_1$ is a  $\delta_1$-matching between $\Aa, \Bb$, and $\Mm_2$ is a $\delta_2$-matching between $\Bb, \Cc$. Let $\delta = \delta_1 + \delta_2$. We must show that there is a $\delta$-matching between $\Aa, \Cc$.

Define the composition of $\Mm_1, \Mm_2$ to be
\[
\Mm = \left( 
(\alpha,\gamma)
\mid
\text{there exists $\beta \in \Bb$ such that $(\alpha,\beta) \in \Mm_1$ and $(\beta,\gamma) \in \Mm_2$}
\right).
\]
We verify that $\Mm$ is the required $\delta$-matching:
\begin{vlist}
\item
If $(\alpha,\gamma) \in \Mm$ then
\[
\ellinf(\alpha,\gamma) \leq \ellinf(\alpha,\beta) + \ellinf(\beta,\gamma) \leq \delta_1 + \delta_2 = \delta
\]
where $\beta \in \Bb$ is the point linking $\alpha$ to~$\gamma$.

\item
If $\alpha$ is unmatched in $\Mm$ then there are two possibilities. Either $\alpha$ is unmatched in $\Mm_1$, in which case
\[
\ellinf(\alpha,\Delta) \leq \delta_1 \leq \delta.
\]
Or $\alpha$ is matched in $\Mm_1$, let's say $(\alpha,\beta) \in \Mm_1$. Then $\beta$~must be unmatched in $\Mm_2$, so
\[
\ellinf(\alpha,\Delta) \leq \ellinf(\alpha,\beta) + \ellinf(\beta,\Delta) \leq \delta_1 + \delta_2 = \delta.
\]

\item
If $\gamma$ is unmatched in $\Mm$, then a similar argument shows that
\[
\ellinf(\gamma, \Delta) \leq \delta.
\]
\end{vlist}

This completes the proof.
\end{proof}

\begin{remark}
Because $\Aa, \Bb, \Cc$ are in truth multisets rather than sets, the composition operation between matchings is not uniquely defined, but depends on how the matchings are realised when labels are added.
Figure~\ref{fig:M1M2M} illustrates what can happen when $\Bb$ has points of multiplicity greater than~1.
\begin{figure}
\centerline{
\hfill
\includegraphics[height=18ex]{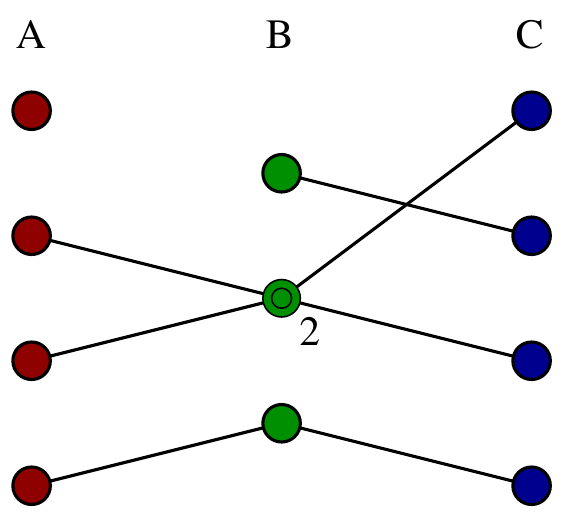}
\hfill
\raisebox{7ex}{gives}
\hfill
\includegraphics[height=18ex]{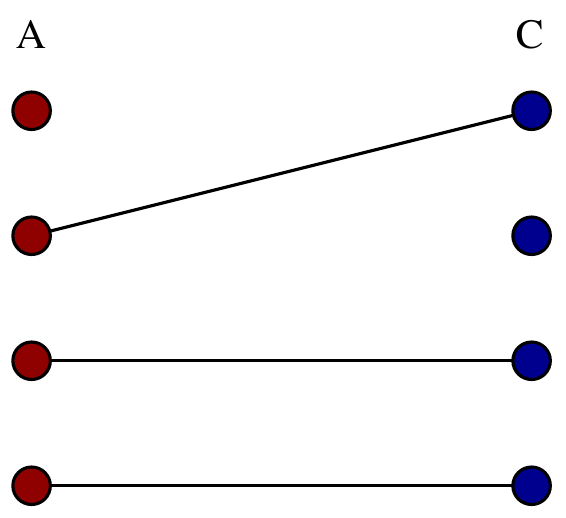}
\hfill
\raisebox{7ex}{or}
\hfill
\includegraphics[height=18ex]{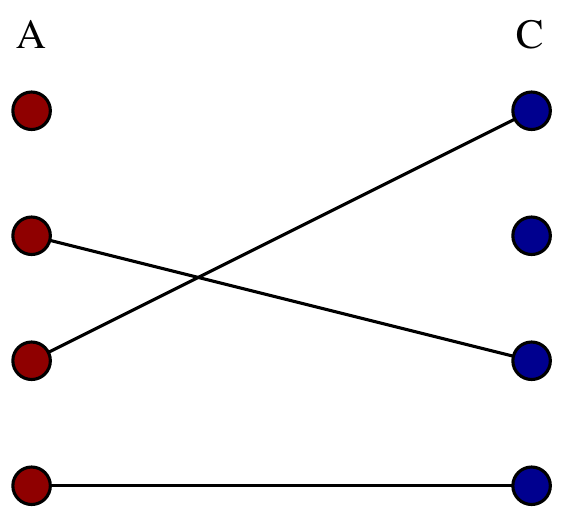}
\hfill
}
\caption{The partial matchings between $\Aa,\Bb$ and $\Bb, \Cc$ (left) compose
to give a partial matching between $\Aa, \Cc$ in two different ways (middle, right). The non-uniqueness arises from the point of multiplicity~2 in~$\Bb$.
}
\label{fig:M1M2M}
\end{figure}
Since we are concerned only with the existence of the composite matching, this ambiguity does not trouble us.
\end{remark}

Here is the first substantial-looking result comparing the interleaving and bottleneck  distances.

\begin{theorem}
\label{thm:pre-converse}
Let $\Uu, \Vv$ be decomposable persistence modules (i.e.\ direct sums of interval modules). Then
\[
\inter(\Uu, \Vv) \leq \bottle(\dgm(\Uu), \dgm(\Vv)).
\]
\end{theorem}

(We remind the reader that tameness is not required to define $\Dgm$ and $\dgm$ in this case: see section~\ref{subsec:pd1}.)

\begin{proof}
We show that whenever there exists a $\delta$-matching between $\dgm(\Uu)$ and $\dgm(\Vv)$, we have $\inter(\Uu, \Vv) \leq \delta$. The result follows by taking the infimum over all such~$\delta$.

Let $\Mm$ be a $\delta$-matching between the two diagrams. Since the points in each diagram correspond to the interval summands of the module, we can construct from~$\Mm$ a partial matching between the interval summands of $\Uu$ and $\Vv$. 

Re-write $\Uu$ and $\Vv$ in the form
\[
\Uu = \bigoplus_{\ell \in L} \Uu_\ell,
\quad
\Vv = \bigoplus_{\ell \in L} \Vv_\ell
\]
so that each pair $(\Uu_\ell, \Vv_\ell)$ is one of the following:
\begin{vlist}
\item
a pair of matched intervals;
\item
$\Uu_\ell$ is an unmatched interval, $\Vv_\ell = 0$;
\item
$\Vv_\ell$ is an unmatched interval, $\Uu_\ell = 0$.
\end{vlist}

In each case, by Propositions \ref{prop:pqrs} and~\ref{prop:pqD}, we have $\inter(\Uu_\ell, \Vv_\ell) \leq \delta$. It follows from Proposition~\ref{prop:inter-sum} that $\inter(\Uu, \Vv) \leq \delta$.
\end{proof}

We complete this section with the postponed proof.

\begin{proof}[Proof of Proposition~\ref{prop:pqrs}]
We treat the case where $p, q, r, s$ are all finite. We must show that if
\[
\delta > \max\left( |p-r|, |q-s| \right)
\]
then $\Uu, \Vv$ are $\delta$-interleaved. We define systems of linear maps
\begin{align*}
\Phi &= (\phi_t : U_t \to V_{t+\delta})
\\
\Psi &= (\psi_t : V_t \to U_{t+\delta})
\end{align*}
and then show that the interleaving relations
\[
\Phi 1_\Uu^\eta = 1_\Vv^\eta \Phi,
\quad
\Psi 1_\Vv^\eta = 1_\Uu^\eta \Psi,
\quad
\Psi\Phi = 1_\Uu^{2\delta},
\quad
\Phi\Psi = 1_\Vv^{2\delta}
\]
hold.

The definition of the maps $\phi_t, \psi_t$ is straightforward. Each vector space in $\Uu, \Vv$ is equal to zero or to the field $\kk$. If the domain and codomain equal~$\kk$, then the map is defined to be the identity $1 = 1_\kk$. Otherwise, the map is necessarily~0.

The first step is to show that the systems of maps $\Phi = (\phi_t)$, $\Psi = (\psi_t)$ are module homomorphisms. For $\Phi$ this entails verifying that the diagram
\[
\begin{diagram}
\dgARROWLENGTH=2em
\node{U_t}
   \arrow[2]{e}
   \arrow{se}
\node[2]{U_{t+\eta}}
   \arrow{se}
\\
\node[2]{V_{t+\delta}}
   \arrow[2]{e}
\node[2]{V_{t+\eta+\delta}}
\end{diagram}
\]
commutes for all $t$ and for all $\eta > 0$. 
Because of the special form of the vector spaces and maps, it is enough to show that the situation is not one of the following:
\[
\begin{diagram}
\dgARROWLENGTH=1.5em
\node{\bullet}
   \arrow[2]{e}
   \arrow{se}
\node[2]{\bullet}
   \arrow{se}
\\
\node[2]{\circ}
   \arrow[2]{e}
\node[2]{\bullet}
\end{diagram}
\qquad\text{or}\qquad
\begin{diagram}
\dgARROWLENGTH=1.5em
\node{\bullet}
   \arrow[2]{e}
   \arrow{se}
\node[2]{\circ}
   \arrow{se}
\\
\node[2]{\bullet}
   \arrow[2]{e}
\node[2]{\bullet}
\end{diagram}
\]
Here a filled circle~$\bullet$ indicates that the vector space is~$\kk$, and an open circle~$\circ$ denotes that it is zero.
For the first situation to occur, one must have
\[
p^* < t
\quad\text{and}\quad
t + \delta< r^*
\]
but $\delta > r-p$ so this is impossible. For the second situation to occur, one must have 
\[
q^* < t + \eta
\quad\text{and}\quad
t + \eta + \delta < s^*
\]
but $\delta > s - q$ so this too is impossible.
It follows that $\Phi$ is a module homomorphism. By symmetry, so is~$\Psi$.

The second step is to show that $\Psi\Phi = 1_\Uu^{2\delta}$ and $\Phi\Psi = 1_\Vv^{2\delta}$. For the first of these, we must verify that the diagram
\[
\begin{diagram}
\dgARROWLENGTH=1.5em
\node{U_t}
   \arrow[2]{e}
   \arrow{se}
\node[2]{U_{t+2\delta}}
\\
\node[2]{V_{t+\delta}}
   \arrow{ne}
\end{diagram}
\]
commutes for all~$t$.
This time the only forbidden configuration is:
\[
\begin{diagram}
\dgARROWLENGTH=1.5em
\node{\bullet}
   \arrow[2]{e}
   \arrow{se}
\node[2]{\bullet}
\\
\node[2]{\circ}
   \arrow{ne}
\end{diagram}
\]
If this occurs then the top row implies:
\[
p^* < t
\quad\text{and}\quad
t + 2\delta < q^*
\]
Since $\delta > r-p$ and $\delta > q-s$ we infer that
\[
r^* < (p+\delta)^* < t + \delta < (q-\delta)^* < s^*
\]
which implies that the circle on the bottom row is filled after all.
It follows that $\Psi\Phi = 1_\Uu^{2\delta}$. By symmetry, $\Phi\Psi = 1_\Vv^{2\delta}$.

This completes the proof when $p, q, r, s$ are finite. The infinite cases are similar.
\end{proof}

\begin{remark}
In the proof we have used and manipulated inequalities between decorated and undecorated real numbers. The definitions are the natural ones:
\[
p^- < p < p^+ < q^- < q < q^+
\]
whenever $p < q$ are undecorated real numbers. The proof can be written without this; for instance, the inequality $p^* < t$ can be replaced by the marginally weaker assertion $p \leq t$, and so on.
\end{remark}

%--------------------------------------------------
\subsection{The bottleneck distance (continued)}
\label{subsec:bottle-compact}
If $\Aa, \Bb$ are locally finite, it turns out that the `inf' is attained in the definition
\[
\bottle(\Aa, \Bb)
=
\inf \left( \delta \mid \text{there exists a $\delta$-matching between $\Aa$ and $\Bb$} \right),
\]
and can be replaced by `min'. This will allow us to make a tighter statement of the stability theorem~(\ref{thm:isometry}$''$) for q-tame modules. See Theorem~\ref{thm:stability1}.

\begin{theorem}
\label{thm:compact}
Let $\Aa, \Bb$ be locally finite multisets in the extended open half-plane $\UpperInt$.
Suppose for every $\eta > \delta$ there exists an $\eta$-matching between $\Aa, \Bb$. Then there exists a $\delta$-matching between $\Aa, \Bb$.
\end{theorem}

The assertion is obvious if $\Aa, \Bb$ are finite. The general case is proved using a compactness argument, as follows.

\begin{proof}
As usual we treat $\Aa, \Bb$ as sets rather than multisets.

For every integer $n \geq 1$,
let $\Mm_n$ be a $(\delta+\frac{1}{n})$-matching between $\Aa, \Bb$.
%; its existence is guaranteed by Theorem~\ref{thm:stability2}.
%
The plan is to construct a $\delta$-matching $\Mm$ from the sequence $(\Mm_n)$. 
In practice, we work with the indicator functions
\begin{align*}
\chi &: \Aa \times \Bb \to \{0, 1\}
\\
\chi_n &: \Aa \times \Bb \to \{0, 1\}
\end{align*}
of the partial matchings $\Mm, \Mm_n$.

The first step is to construct $\chi$ as a limit of the sequence $(\chi_n)$. Take a fixed enumeration
\[
\left( (\alpha_\ell, \beta_\ell) \mid \ell \geq 1 \right)
\]
of the countable set $\Aa \times \Bb$. We will inductively construct a descending sequence
\[
\Nn = \Nn_0 \supseteq \Nn_1 \supseteq \dots \supseteq \Nn_\ell \supseteq \dots
\]
of infinite subsets of the natural numbers, with the property that $\chi_n(\alpha_\ell, \beta_\ell)$ takes the same value for all $n \in \Nn_\ell$. Having done so, we define $\chi(\alpha_\ell, \beta_\ell)$ to be this common value.

The construction of $\Nn_\ell$ is straightforward: once $\Nn_{\ell-1}$ is defined, at least one of the two sets
\[
\{ n \in \Nn_{\ell-1} \mid \chi_n(\alpha_\ell, \beta_\ell) = 0 \}
\quad \text{and} \quad
\{ n \in \Nn_{\ell-1} \mid \chi_n(\alpha_\ell, \beta_\ell) = 1 \}
\]
has infinite cardinality, and that will be our $\Nn_\ell$. (If both, then either will do.) Repeat.

\begin{lemma*}
If $\Ff$ is any finite subset of $\Aa \times \Bb$, then there exists $\ell \geq 1$ such that
\[
\chi(\alpha, \beta) = \chi_n(\alpha, \beta)
\]
for all $(\alpha, \beta) \in \Ff$ and for all $n \in \Nn_\ell$.
\end{lemma*}

\begin{proof}
Indeed, select $\ell$ such that $(\alpha_1, \beta_1),\, \dots, (\alpha_\ell, \beta_\ell)$ include all of~$\Ff$.
\end{proof}

The second step is to verify that $\chi$ is the indicator function of a $\delta$-matching. There are several items to check.

\quad
$\bullet$
For $\alpha \in \Aa$ there is at most one $\beta \in \Bb$ such that $\chi(\alpha, \beta)=1$. 

\begin{quote}
{\it Proof.}\;
Suppose $\chi(\alpha, \beta) = \chi(\alpha, \beta') = 1$ for distinct elements $\beta, \beta' \in \Bb$. By the lemma, there exists~$n$ such that $\chi_n(\alpha, \beta) = \chi_n(\alpha, \beta') = 1$, which contradicts the fact that $\Mm_n$ is a partial matching.
\hfill
\qed
\end{quote}

\quad
$\bullet$
For $\alpha \in \Aa$ with $\ellinf(\alpha,\Delta) > \delta$, there is at least one $\beta \in \Bb$ such that $\chi(\alpha, \beta) = 1$.

\begin{quote}
{\it Proof.}\;
Select $N$ such that  $\ellinf(\alpha,\Delta) > \delta + \frac{1}{N}$.
Then the set
\[
\Ff_\alpha
=
\left\{ \beta \in \Bb \mid \ellinf(\alpha, \beta) \leq \delta + \textstyle\frac{1}{N} \right\}
\]
is finite, since $\Bb$ is locally finite and these points lie in a square bounded away from the diagonal.
By the lemma, there exists~$\ell$ such that
\[
\chi(\alpha, \beta) = \chi_n(\alpha, \beta)
\]
for all $\beta \in \Ff_\alpha$ and for all $n \in \Nn_\ell$. 
On the other hand, if $n \geq N$, then $\Mm_n$ matches $\alpha$ with some $\beta \in \Ff_\alpha$.
Combining these observations,
\[
\chi(\alpha, \beta) = \chi_n(\alpha, \beta) = 1
\]
for sufficiently large $n \in \Nn_\ell$ and for some $\beta \in \Ff_\alpha$.
\hfill
\qed
\end{quote}

By symmetry we have:

\quad
$\bullet$
For $\beta \in \Bb$ there is at most one $\alpha \in \Aa$ such that $\chi(\alpha, \beta)=1$. 

\quad
$\bullet$
For $\beta \in \Bb$ with $\ellinf(\beta,\Delta) > \delta$, there is at least one $\alpha \in \Aa$ such that $\chi(\alpha, \beta) = 1$.

Finally:

\quad
$\bullet$
If $\chi(\alpha, \beta) = 1$ then $\ellinf(\alpha, \beta) \leq \delta$.

\begin{quote}
{\it Proof.}\;
By the lemma, there are infinitely many $n$ for which $\chi_n(\alpha, \beta)=1$. Then
\[
\ellinf(\alpha, \beta) \leq \delta + \textstyle\frac{1}{n}
\]
for these~$n$. Since $n$ may be arbitrarily large, the result follows.
\hfill
\qed
\end{quote}

These five bullet points confirm that $\Mm$, defined by its indicator function $\chi$, is a $\delta$-matching between $\Aa, \Bb$.
\end{proof}

\begin{remark}
Although we have chosen to give a direct argument, we point out that Theorem~\ref{thm:compact} is an instance of the compactness theorem in first-order logic. The set of constraints that must be satisfied by an $\eta$-matching can be formulated as a theory $T_\eta$ on a collection of binary-valued variables $x_{\alpha\beta}$. An $\eta$-matching is precisely a model for that theory. The theory $T_\delta$ is seen to be logically equivalent to the union of the theories $( T_\eta \mid \eta > \delta)$. If each $T_\eta$ has a model, then any finite subtheory of this union has a model, therefore by compactness $T_\delta$ has a model.
The details are left as an exercise for the interested reader.
\end{remark}

%--------------------------------------------------
\subsection{The isometry theorem}
\label{subsec:isometry}

Having defined the interleaving distance and the bottleneck distance, we can now state the main theorem.

\begin{theorem}
\label{thm:isometry}
Let $\Uu, \Vv$ be q-tame persistence modules.
Then
\[
\inter(\Uu,\Vv) = \bottle(\dgm(\Uu), \dgm(\Vv))
\]
(Recall that $\dgm$ denotes the undecorated persistence diagram.)
\end{theorem}

The result naturally falls into two parts: the `stability theorem'~\cite{CohenSteiner_E_H_2007,Chazal_CS_G_G_O_2009}
\[
\inter(\Uu,\Vv) \geq \bottle(\dgm(\Uu), \dgm(\Vv)),
\tag{\ref{thm:isometry}$'$}
\]
and the `converse stability theorem'~\cite{Lesnick_2011}
\[
\inter(\Uu,\Vv) \leq \bottle(\dgm(\Uu), \dgm(\Vv)).
\tag{\ref{thm:isometry}$''$}
\]

The proof of the converse stability theorem~(\ref{thm:isometry}$''$) occupies section~\ref{subsec:converse}. We have already seen the result for decomposable modules, in 
Theorem~\ref{thm:pre-converse}, so it is a matter of extending the result to q-tame modules that are not known to be decomposable.

The proof of the stability theorem~(\ref{thm:isometry}$'$) is given in sections \ref{subsec:stability} and~\ref{subsec:mstability}.

%--------------------------------------------------
\subsection{The converse stability theorem}
\label{subsec:converse}

In this section we deduce the converse stability inequality~(\ref{thm:isometry}$''$) for q-tame modules from Theorem~\ref{thm:pre-converse}. A similar argument is given by Lesnick~\cite{Lesnick_2011}.

The key idea is that persistence modules can be approximated by better-behaved persistence modules, using a procedure we call `smoothing'.

\begin{definition}
Let $\Vv$ be a persistence module, and let $\epsilon > 0$. The {\bf $\epsilon$-smoothing} of~$\Vv$ is the persistence module $\Vv^\epsilon$ defined to be the image of the following map:
\[
1_\Vv^{2\epsilon} : \Vv[t-\epsilon] \to \Vv[t+\epsilon]
\]
In other words, $V^\epsilon_t$ is the image of the map
\[
v_{t-\epsilon}^{t+\epsilon} : V_{t-\epsilon} \to V_{t+\epsilon},
\]
and, for $s < t$, the shift map $(v^\epsilon)_s^t$ is the restriction of $v_{s+\epsilon}^{t+\epsilon}$.
\end{definition}

Thus we have a factorisation of $1_\Vv^{2\epsilon}$
\begin{equation}
\begin{diagram}
%\dgARROWLENGTH=1.5em
\node{\Vv[t-\epsilon]}
   \arrow{e}
\node{\Vv^\epsilon}
   \arrow{e}
\node{\Vv[t+\epsilon]}
\end{diagram}
\label{eq:smooth-inter}
\end{equation}
where the first map is surjective and the second map is injective (at all~$t$).
At a given index $t$ this is the sequence:
\[
\begin{diagram}
\node{V_{t-\epsilon}}
   \arrow{e,t}{v_{t-\epsilon}^{t+\epsilon}}
\node{V^\epsilon_{t}}
   \arrow{e,t}{1}
\node{V_{t+\epsilon}}
\end{diagram}
\]

\begin{proposition}
\label{prop:smooth-inter}
Let $\Vv$ be a persistence module. Then $\inter(\Vv, \Vv^\epsilon) \leq \epsilon$.
\end{proposition}

\begin{proof}
Indeed, the maps in~\eqref{eq:smooth-inter} give an $\epsilon$-interleaving. 
\end{proof}

Smoothing changes the persistence diagram in a predictable way. Here is the atomic example (which the reader can easily verify):

\begin{example}
\label{ex:smooth-interval}
Let $\Vv = \Ii{\lgroup p^*, q^* \rgroup}$. Then:
\[
\Vv^\epsilon =
\begin{cases}
\Ii{\lgroup (p+\epsilon)^*, (q-\epsilon)^* \rgroup}
	& \text{if $(p+\epsilon)^* < (q-\epsilon)^*$}
\\
0 & \text{otherwise}
\end{cases}
\]
In other words, $\epsilon$-smoothing shrinks the interval by $\epsilon$ at both ends.
\end{example}

\begin{proposition}
\label{prop:smooth-diagram}
The persistence diagram of~$\Vv^\epsilon$ is obtained from the persistence diagram of~$\Vv$ by applying the translation $T_\epsilon: (p,q) \mapsto (p+\epsilon, q-\epsilon)$ to the part of the extended half-plane that lies above the line $\Delta_\epsilon = \{ (t-\epsilon, t+\epsilon) \mid t \in \Rr \}$.
\end{proposition}

%\begin{remark}
In the case where $\Dgm$ is not everywhere defined, the proposition is understood to include the assertion that the finite r-interior of the persistence measure, and hence the region where $\Dgm$ is defined, is shifted by $T_\epsilon$.
%\end{remark}

Information on $\Dgm(\Vv)$ that lies below the line $\Delta_\epsilon$ is lost in $\Dgm(\Vv^\epsilon)$.

\begin{proof}
We consider three different cases. Case~(ii) is subsumed by case~(iii), but the proof is easier to digest.

{\small (i) \bf  $\Vv$ is decomposable.} The image of a direct sum of maps is the direct sum of the images of the maps; therefore $\epsilon$-smoothing commutes with direct sums:
\[
\Big[ \bigoplus_{\ell \in L} \Vv_\ell \Big]^\epsilon
=
\bigoplus_{\ell \in L} \Vv_\ell^\epsilon
\]
By Example~\ref{ex:smooth-interval}, the proposition is true for interval modules. It is therefore true for direct sums of interval modules.

{\small (ii) \bf  $\Vv$ is q-tame.} It is enough to show that the rank function of~$\Vv^\epsilon$ is equal to the rank function of $\Vv$ shifted by $T_\epsilon$, since this determines the persistence measure and hence the persistence diagram.
Specifically, for all $b < c$ we require:
\[
\rank[ V^\epsilon_b \to V^\epsilon_c]
=
\rank[ V_{b-\epsilon} \to V_{c+\epsilon} ]
\]
In fact, these maps are related by the sequence
\[
\begin{diagram}
\node{V_{b-\epsilon}}
   \arrow{e}
\node{V^\epsilon_b}
   \arrow{e}
\node{V^\epsilon_c}
   \arrow{e}
\node{V_{c+\epsilon}}
\end{diagram}
\]
where the first map is surjective and the third map is injective. Since the rank of a linear map is unchanged by pre-composing with a surjective map, or post-composing with an injective map, it follows that the rank of the middle map is equal to the rank of the composite. This is what we wished to prove.

{\small (iii) \bf general case.} We show that the persistence measure of $\Vv^\epsilon$ is equal to the persistence measure of $\Vv$ shifted by~$T_\epsilon$. Writing
\[
A = a-\epsilon, \quad
B = b-\epsilon, \quad
C = c+\epsilon, \quad
D = d+\epsilon
\]
this means showing that
\[
\langle
\Qoff{A}\qem\Qon{B}\qem\Qon{C}\qem\Qoff{D} \mid \Vv
\rangle
=
\langle
\Qoff{a}\qem\Qon{b}\qem\Qon{c}\qem\Qoff{d} \mid \Vv^\epsilon
\rangle
\]
for all $a < b \leq c < d$.

The proof is based on the following commutative diagram
\[
\dgARROWLENGTH=1.5em
\begin{diagram}
\node[2]{V^\epsilon_a}
   \arrow[2]{e}
\node[2]{V^\epsilon_b}
   \arrow{e}
\node{V^\epsilon_c}
   \arrow[2]{e}
   \arrow{se}
\node[2]{V^\epsilon_d}
   \arrow{se}
\\
\node{V_A}
   \arrow[2]{e}
   \arrow{ne}
\node[2]{V_B}
   \arrow[3]{e}
   \arrow{ne}
\node[3]{V_C}
   \arrow[2]{e}
\node[2]{V_D}
\end{diagram}
\]
in which the maps~$\nearrow$ are surjective and the maps~$\searrow$ are injective. We will construct various quiver diagrams from this. Surjectivity means that
\[
\langle \Qoff{A}\qem\qon{a} \rangle = 0
\quad\text{and}\quad
\langle \Qoff{B}\qem\Qon{b} \rangle = 0,
\]
and injectivity means that
\[
\langle \Qon{c}\qem\Qoff{C} \rangle = 0
\quad\text{and}\quad
\langle \Qon{d}\qem\Qoff{D} \rangle = 0.
\]
Moreover, by the restriction principle, interval types containing any of these `forbidden' configurations occur with multiplicity zero.

Then
\begin{align*}
\langle \Qoff{A}\qem\Qno\qem\Qon{b}\qem\Qon{c}\qem\Qno\qem\Qoff{D} \rangle
&=
\langle \Qoff{A}\qem\Qon{B}\qem\Qon{b}\qem\Qon{c}\qem\Qon{C}\qem\Qoff{D} \rangle
\\
&\qquad\text{+ three other terms}
\\
&=
\langle \Qoff{A}\qem\Qon{B}\qem\Qon{b}\qem\Qon{c}\qem\Qon{C}\qem\Qoff{D} \rangle
\\
&= \langle \Qoff{A}\qem\Qon{B}\qem\Qno\qem\Qno\qem\Qon{C}\qem\Qoff{D} \rangle
\\
&= \langle \Qoff{A}\qem\Qon{B}\qem\Qon{C}\qem\Qoff{D} \mid \Vv \rangle
\end{align*}
and at the same time
\begin{align*}
\langle \Qoff{A}\qem\Qno\qem\Qon{b}\qem\Qon{c}\qem\Qno\qem\Qoff{D} \rangle
&=
\langle \Qoff{A}\qem\Qoff{a}\qem\Qon{b}\qem\Qon{c}\qem\Qoff{d}\qem\Qoff{D} \rangle
\\
&\qquad\text{+ three other terms}
\\
&=
\langle \Qoff{A}\qem\Qoff{a}\qem\Qon{b}\qem\Qon{c}\qem\Qoff{d}\qem\Qoff{D} \rangle
\\
&=
\langle \Qno\qem\Qoff{a}\qem\Qon{b}\qem\Qon{c}\qem\Qoff{d}\qem\Qno \rangle
\\
&=
\langle \Qoff{a}\qem\Qon{b}\qem\Qon{c}\qem\Qoff{d} \mid \Vv^\epsilon \rangle
\end{align*}
so we get the required equality. The six `other terms' are all zero because they contain  forbidden configurations.
\end{proof}

\begin{corollary}
\label{cor:smooth-bottle}
Let $\Vv$ be a q-tame persistence module. Then $\bottle(\dgm(\Vv), \dgm(\Vv^\epsilon)) \leq \epsilon$.
\end{corollary}

\begin{proof}
Indeed, an $\epsilon$-matching is defined as follows:
\[
(p,q) \in \dgm(\Vv^\epsilon)
\quad\leftrightarrow\quad
(p-\epsilon, q+\epsilon) \in \dgm(\Vv)
\]
This is bijective except for the unmatched points of $\dgm(\Vv)$, which lie on or below the line $\Delta_\epsilon$, and therefore have distance at most $\epsilon$ from the diagonal.
\end{proof}

Our claim that smoothing makes a persistence module `better-behaved' amounts to the following fact.

\begin{proposition}
\label{prop:smooth-lf}
Let $\Vv$ be a q-tame persistence module over~$\Rr$. Then $\Vv^\epsilon$ is locally finite. In particular, $\Vv^\epsilon$ is decomposable into intervals.
\end{proposition}

\begin{proof}
We make use of the characterisation in Proposition~\ref{prop:locally-finite}.
Note first that
\[
\dim(V^\epsilon_t) = \rank[V_{t-\epsilon} \to V_{t+\epsilon}] < \infty
\]
so condition~(i) is satisfied.

For condition~(ii), we must show that there is a locally finite set $\Ss \subset \Rr$ of `singular values' with the property that $\Vv^\epsilon$ is constant over each interval of the open set $\Rr - \Ss$.
%
%In other words, the map $V^\epsilon_{b} \to V^\epsilon_{c}$ is an isomorphism for all $b < c$ such that $[b,c] \cap S = \emptyset$.
%

The set $\Ss$ is easily described: consider all points $(p^*, q^*) \in \Dgm(\Vv)$ such that $p^* + 2\epsilon < q^*$. Define $\Ss$ to be the union of the sets $\{ p + \epsilon, q - \epsilon \}$ over all such points of $\Dgm(\Vv)$.

In other words, we take all points of $\Dgm(\Vv)$ which lie on\footnote{
Points on the line must be decorated $((t - \epsilon)^-, (t + \epsilon)^+)$.
}
or above the line
\[
\Delta_\epsilon = \{ (t-\epsilon, t+\epsilon) \mid t \in \Rr \}
\]
and then project each point both vertically and horizontally onto~$\Delta_\epsilon$, abandoning the decoration. The resulting subset of $\Delta_\epsilon$ defines~$\Ss \subset \Rr$, using the identification $(t-\epsilon, t+\epsilon) \leftrightarrow t$.
See Figure~\ref{fig:rays} (left).

We show that $\Ss$ is locally finite. Indeed, for any $t$ and for any $\eta < \epsilon$,
\begin{align*}
\card(S \cap (t-\eta, t+\eta))
&\leq 2 \mu_\Vv([-\infty,t+\eta-\epsilon]\times[t-\eta+\epsilon,+\infty])
\\
&= 2 \rank[V_{t+\eta-\epsilon} \to V_{t-\eta+\epsilon}]
\end{align*}
which is finite.

We show that $\Vv$ is constant over each component of $\Rr - \Ss$.

{\bf\small Claim.}
The map $V^\epsilon_{b} \to V^\epsilon_{c}$ is an isomorphism if and only if $\Dgm(\Vv)$ does not meet the union of the rectangles
\[
[-\infty, b-\epsilon] \times [b+\epsilon, c+\epsilon]
\quad \text{and} \quad
[b-\epsilon, c-\epsilon] \times [c+\epsilon, +\infty].
\]
See Figure~\ref{fig:rays} (right).

\begin{figure}
\centerline{
\hfill
\includegraphics[scale=0.6]{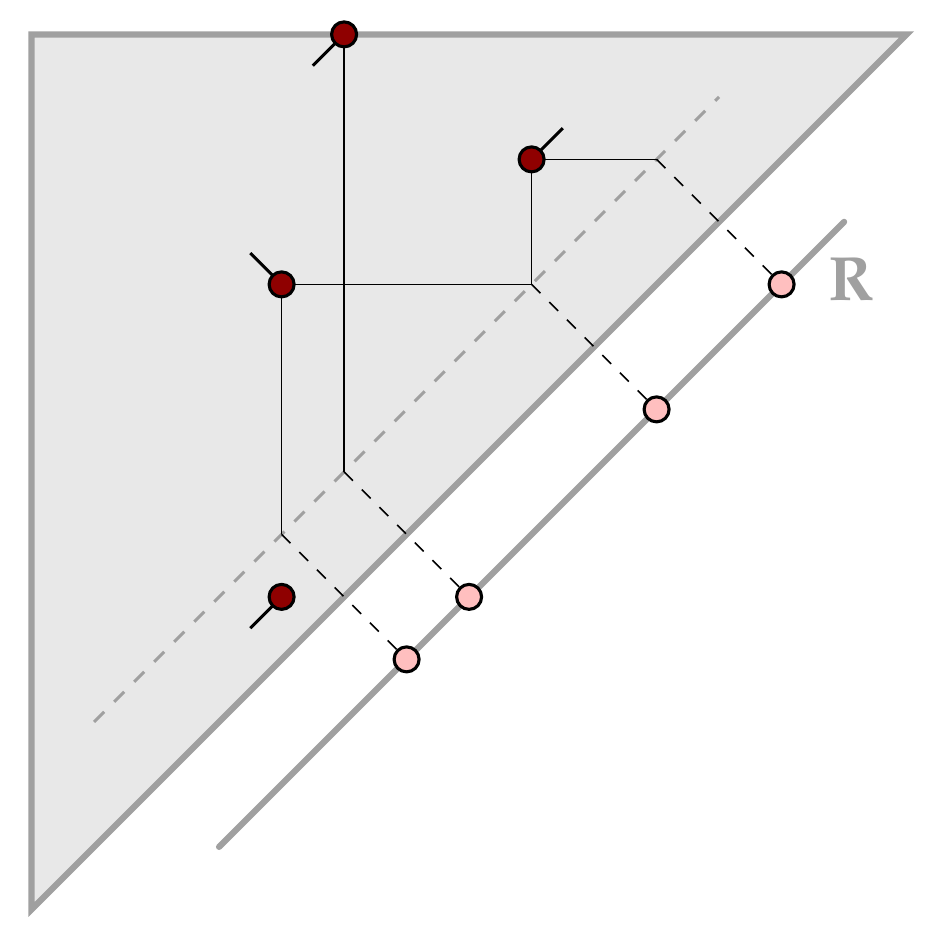}
\hfill
\includegraphics[scale=0.6]{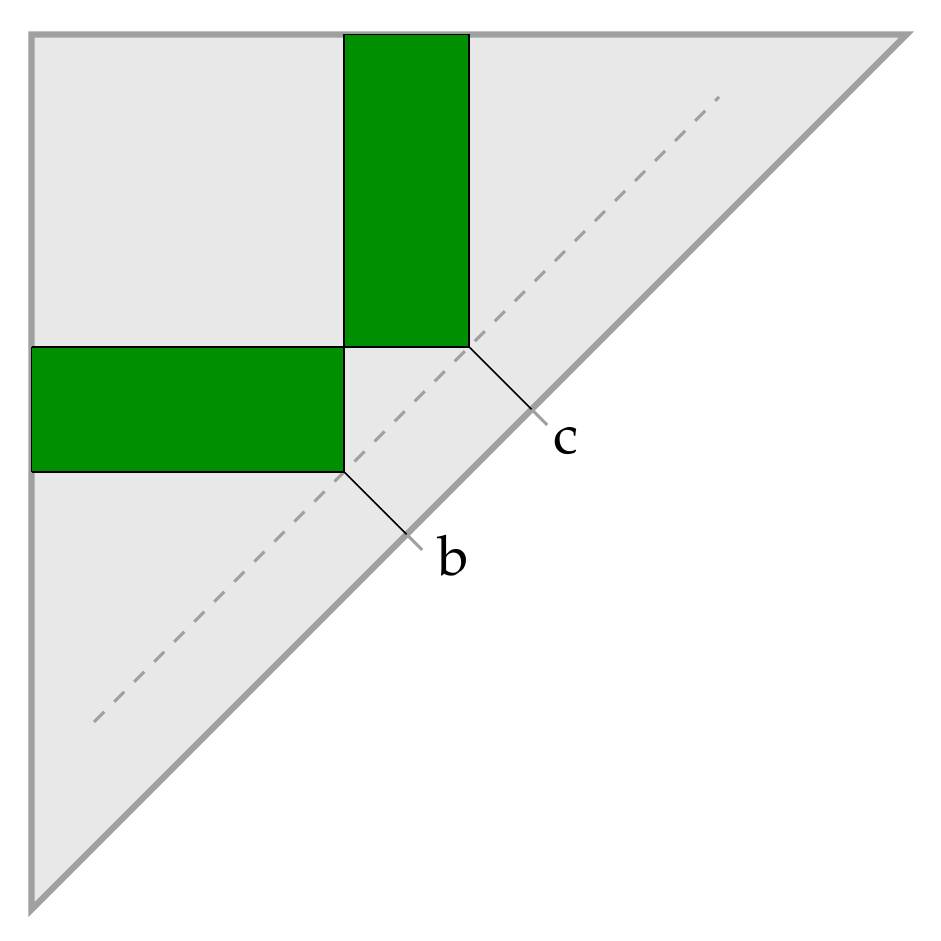}
\hfill
}
\caption{
Constructing the set $\Ss$ from the persistence diagram (left).
The exclusion zone for $V^\epsilon_b \to V^\epsilon_c$ to be an isomorphism (right) .
}
\label{fig:rays}
\end{figure}

We split the claim into two statements.

{\bf\small Part 1.} The map $V^\epsilon_{b} \to V^\epsilon_{c}$ is injective if and only if
\[
\mu_\Vv([-\infty, b-\epsilon] \times [b+\epsilon, c+\epsilon]) = 0.
\]

{\bf\small Part 2.} The map $V^\epsilon_{b} \to V^\epsilon_{c}$ is surjective if and only if
\[
\mu_\Vv([b-\epsilon, c-\epsilon] \times [c+\epsilon, +\infty]) = 0.
\]

We approach this in the usual way. Consider the following diagram of vector spaces:
\[
\begin{diagram}
\node[2]{V^\epsilon_b}
   \arrow{e}
   \arrow{se}
\node{V^\epsilon_c}
   \arrow{se}
\\
\node{V_{b-\epsilon}}
   \arrow{ne}
   \arrow{e}
\node{V_{c-\epsilon}}
   \arrow{ne}
\node{V_{b+\epsilon}}
   \arrow{e}
\node{V_{c+\epsilon}}
\end{diagram}
\]
where the maps~$\nearrow$ are surjective and the maps~$\searrow$ are injective.
Using the abbreviations
\[
b_- = b-\epsilon,\quad
b_+ = b+\epsilon,\quad
c_- = c-\epsilon,\quad
c_+ = c+\epsilon
\]
we can write down the `forbidden' configurations
\[
\Qoff{b_-}\qem\Qon{b}
\quad\;
\Qon{b}\qem\Qoff{b_+}
\quad\;
\Qoff{c_-}\qem\Qon{c}
\quad\;
\Qon{c}\qem\Qoff{c_+}
\]
which occur with multiplicity zero.

For Part~1, note that $V^\epsilon_{b} \to V^\epsilon_{c}$ is injective if and only if
\[
\langle \Qon{b}\qem\Qoff{c} \mid \Vv^\epsilon \rangle = 0.
\]
Taking the forbidden configurations into account, we calculate:
\begin{alignat*}{2}
\langle \Qon{b}\qem\Qoff{c} \mid \Vv^\epsilon \rangle
&= 
\langle \Qno\qem\Qon{b}\qem\Qoff{c}\qem\Qno\rangle
\\
&= 
\langle \Qno\qem\Qon{b}\qem\Qoff{c}\qem\Qoff{c_+}\rangle
\\
&= 
\langle \Qno\qem\Qon{b}\qem\Qno\qem\Qoff{c_+}\rangle
\qquad
&&
%%%
\text{since $\langle \Qon{c}\qem\Qoff{c_+}\rangle$ is zero}
%%%
\\
&= 
\langle \Qon{b_-}\qem\Qon{b}\qem\Qon{b_+}\qem\Qoff{c_+}\rangle
\qquad
&&
%%%
\text{since $\langle \Qoff{b_-}\qem\Qon{b}\rangle$ and $\langle \Qon{b}\qem\Qoff{b_+}\rangle$ are zero}
%%%
\\
&= 
\langle \Qon{b_-}\qem\Qno\qem\Qon{b_+}\qem\Qoff{c_+}\rangle
\end{alignat*}
which is exactly $\mu_\Vv([-\infty, b-\epsilon] \times [b+\epsilon, c+\epsilon])$.

For Part~2, note that $V^\epsilon_{b} \to V^\epsilon_{c}$ is surjective if and only if
\[
\langle \Qoff{b}\qem\Qon{c} \mid \Vv^\epsilon \rangle = 0.
\]
We calculate:
\begin{alignat*}{2}
\langle \Qoff{b}\qem\Qon{c} \mid \Vv^\epsilon \rangle
&= 
\langle \Qno\qem\Qoff{b}\qem\Qon{c}\qem\Qno\rangle
\\
&= 
\langle \Qoff{b_-}\qem\Qoff{b}\qem\Qon{c}\qem\Qno\rangle
\\
&= 
\langle \Qoff{b_-}\qem\Qno\qem\Qon{c}\qem\Qno\rangle
%%%
\qquad
&&
\text{since $\langle \Qoff{b_-}\qem\Qon{b}\rangle$ is zero}
%%%
\\
&= 
\langle \Qoff{b_-}\qem\Qon{c_-}\qem\Qon{c}\qem\Qon{c_+}\rangle
%%%
\qquad
&&
\text{since $\langle \Qoff{c_-}\qem\Qon{c}\rangle$ and $\langle \Qon{c}\qem\Qoff{c_+}\rangle$ are zero}
%%%
\\
&= 
\langle \Qoff{b_-}\qem\Qon{c_-}\qem\Qno\qem\Qon{c_+}\rangle
\end{alignat*}
which is exactly $\mu_\Vv([b-\epsilon, c-\epsilon] \times [c+\epsilon, +\infty])$.

This completes the proof of the claim.

The claim implies that $\Vv$ is constant on each connected component of $\Rr - \Ss$. Indeed, if $\Ss \cap [b,c] = \emptyset$ then the exclusion zone does not contain any points of $\Dgm(\Vv)$, and therefore $(v^\epsilon)_b^c$ is an isomorphism.

Thus $\Vv^\epsilon$ satisfies both conditions of Proposition~\ref{prop:locally-finite} and is therefore locally finite.
\end{proof}

We are now ready to prove the converse stability theorem for q-tame persistence modules, using the triangle inequalities for $\inter, \bottle$ and our results on $\epsilon$-smoothing.

\begin{proof}[Proof of (\ref{thm:isometry}$''$)]
Let $\Uu, \Vv$ be q-tame persistence modules. For any $\epsilon > 0$, the $\epsilon$-smoothings $\Uu^\epsilon, \Vv^\epsilon$ are decomposable, so the converse stability theorem applies to them.
Then:
\begin{alignat*}{2}
\inter(\Uu, \Vv)
%&\leq \inter(\Uu, \Uu^\epsilon) + \inter(\Uu^\epsilon, \Vv^\epsilon) + \inter(\Vv, \Vv^\epsilon)
%\\
& \leq \inter(\Uu^\epsilon, \Vv^\epsilon) + 2\epsilon
&&\text{by Proposition~\ref{prop:smooth-inter}}
\\
& \leq \bottle(\dgm(\Uu^\epsilon), \dgm(\Vv^\epsilon)) + 2\epsilon
\qquad
&&\text{by Theorem~\ref{thm:pre-converse}}
\\
%& \leq \bottle(\dgm(\Uu^\epsilon), \dgm(\Uu)) + \bottle(\dgm(\Uu), \dgm(\Vv)) + \bottle(\dgm(\Vv), \dgm(\Vv^\epsilon)) + 2\epsilon
%\\
& \leq \bottle(\dgm(\Uu), \dgm(\Vv)) + 4\epsilon
&&\text{by Corollary~\ref{cor:smooth-bottle}}
\end{alignat*}
Since this is true for all $\epsilon > 0$, we deduce that
\[
\inter(\Uu,\Vv) \leq \bottle(\dgm(\Uu),\dgm(\Vv)).
\]
The converse stability theorem for q-tame modules is proved.
\end{proof}

We finish this section with a characterisation of q-tame modules.

\begin{theorem}
\label{thm:qtame-char}
A persistence module $\Vv$ is q-tame if and only if it can be approximated, in the interleaving distance, by locally finite modules.
\end{theorem}

\begin{proof}
If $\Vv$ is q-tame then it is approximated by the modules $\Vv^\epsilon$, which are locally finite by Proposition~\ref{prop:smooth-lf}.
Conversely, suppose $\Vv$ is approximated by locally finite modules. Suppose $b < c$ is given. Let $\Ww$ be a locally finite module which is $\epsilon$-interleaved with~$\Vv$, for some $\epsilon < (c-b)/2$. Then
\begin{align*}
\rk_b^c
&= \rank[V_b \to V_c]
\\
&= \rank[V_b \to W_{b+\epsilon} \to W_{c-\epsilon} \to V_c]
\\
&\leq \dim(W_{b+\epsilon})
\\
&< \infty.
\end{align*}
It follows that $\Vv$ is q-tame.
\end{proof}

It is easy to see that there are q-tame modules which are not locally finite, such as:
\[
\bigoplus_{n = 1}^{\infty} \Ii [ 0, \textstyle\frac{1}{n} ]
\]
More interestingly, the closely related module,
\[
\prod_{n = 1}^{\infty} \Ii [ 0, \textstyle\frac{1}{n} ]
\]
constructed by cartesian product rather than direct sum, is q-tame but does not admit an interval decomposition. Indeed, the module is uncountable dimensional at~0 and countable dimensional everywhere else, so any interval decomposition must include uncountably many copies of $\Ii[0,0]$. However, every nonzero element of $V_0$ persists to some $V_t$, so there cannot be any copies of $\Ii[0,0]$. This example is due to Crawley-Boevey~\cite{CrawleyBoevey_2012pc}.

%--------------------------------------------------
\subsection{The stability theorem}
\label{subsec:stability}

The inequality~(\ref{thm:isometry}$''$) can be expressed in the following form:

\begin{theorem}
\label{thm:stability1}
Let $\Uu, \Vv$ be q-tame persistence modules which are $\delta^+$-interleaved. Then there exists a $\delta$-matching between the multisets $\dgm(\Uu)$, $\dgm(\Vv)$.
\end{theorem}

It is easier to prove the following. (Notice the missing $^+$.)

\begin{theorem}
\label{thm:stability2}
Let $\Uu, \Vv$ be q-tame persistence modules which are $\delta$-interleaved. Then there exists a $\delta$-matching between the multisets $\dgm(\Uu)$, $\dgm(\Vv)$.
\end{theorem}

Theorem~\ref{thm:compact} allows us to deduce Theorem~\ref{thm:stability1} from Theorem~\ref{thm:stability2}:
if $\Uu, \Vv$ are $\delta^+$-interleaved then there is an $\eta$-matching between their diagrams for every $\eta > \delta$, hence there is a $\delta$-matching.

The proof of Theorem~\ref{thm:stability2} depends on two main ingredients:

\begin{vlist}
\item
The interpolation lemma~(\ref{lem:interpolation}), which embeds $\Uu, \Vv$ within a 1-parameter family.

\item
The box lemma~(\ref{lem:box}), which relates the persistence measures of $\Uu, \Vv$ locally.

\end{vlist}

Once these ingredients are in place, the theorem can be proved using the continuity method of~\cite{CohenSteiner_E_H_2007}.
Our persistence diagrams may have infinite cardinality, so we will need an additional compactness argument to finish off the proof. 

Let $R = [a,b] \times [c,d]$ be a rectangle in $\RR^2$. The {\bf $\delta$-thickening} of~$R$ is the rectangle
\[
R^\delta = [a-\delta, b+\delta] \times [c-\delta, d+\delta].
\]
For convenience we will write
\[
A = a-\delta, \quad
B = b+\delta, \quad
C = c-\delta, \quad
D = d+\delta
\]
in this situation.
For infinite rectangles, note that $-\infty - \delta = -\infty$ and $+\infty + \delta = +\infty$.

We can also thicken an individual point: if $\alpha = (p,q)$ then
\[
\alpha^\delta = [p-\delta,p+\delta]\times[q-\delta,q+\delta]
\]
for $\delta > 0$.

\begin{lemma}[Box lemma~\cite{CohenSteiner_E_H_2007}]
\label{lem:box}
Let $\Uu, \Vv$ be a $\delta$-interleaved pair of persistence modules. Let $R$ be a rectangle whose $\delta$-thickening~$R^\delta$ lies above the diagonal. Then $\mu_\Uu(R) \leq \mu_\Vv(R^\delta)$ and $\mu_\Vv(R) \leq \mu_\Uu(R^\delta)$.
\end{lemma}

If we use the extension convention (section~\ref{subsec:non-finite}) we can state the lemma without the requirement that $R^\delta$ lies above the diagonal, because the convention gives $\mu(R^\delta) = \infty$ if it doesn't.

\begin{proof}
Write $R = [a,b] \times [c,d]$ and $R^\delta = [A,B] \times [C,D]$ as above.
Thanks to the interleaving, the finite modules
\[
\Uu_{a,b,c,d} \;:\;
U_a \to U_b \to U_c \to U_d
\]
and
\[
\Vv_{A,B,C,D} \;:\;
V_{A} \to V_{B} \to V_{C} \to V_{D}
\]
are restrictions of the following 8-term module
\[
\Ww \;:\;
V_A \stackrel{\Psi}{\longrightarrow} U_a
\longrightarrow
U_b \stackrel{\Phi}{\longrightarrow} V_B
\longrightarrow
V_C \stackrel{\Psi}{\longrightarrow} U_c
\longrightarrow
U_d \stackrel{\Phi}{\longrightarrow} V_D
\]
where $\Phi, \Psi$ are the interleaving maps.

Using the restriction principle, we calculate:
\begin{align*}
\mu_\Vv([A, B] \times [C, D])
&=
\langle\, 
  \Qoff{A}\qem\qno\qem\qno
  \qem\Qon{B}\qem\Qon{C}\qem
  \qno\qem\qno\qem\Qoff{D}
\mid  
  \Vv
\,\rangle
\\
&=
\langle\, 
  \Qoff{A}\qem\qno\qem\qno
  \qem\Qon{B}\qem\Qon{C}\qem
  \qno\qem\qno\qem\Qoff{D}
\mid  
  \Ww
\,\rangle
\\
&=
\langle\,
  \Qoff{A}\qem\qoff{a}\qem\qon{b}
  \qem\Qon{B}\qem\Qon{C}\qem
  \qon{c}\qem\qoff{d}\qem\Qoff{D}
\mid
  \Ww
\,\rangle
\\
&\qquad + \text{eight other terms}
\\
&\geq
\langle\,
  \Qoff{A}\qem\qoff{a}\qem\qon{b}
  \qem\Qon{B}\qem\Qon{C}\qem
  \qon{c}\qem\qoff{d}\qem\Qoff{D}
\mid
  \Ww
\,\rangle
\\
&=
\langle\, 
  \Qno\qem\qoff{a}\qem\qon{b}
  \qem\Qno\qem\Qno\qem
  \qon{c}\qem\qoff{d}\qem\Qno
\mid
  \Ww
\,\rangle
\\
&=
\langle\, 
  \Qno\qem\qoff{a}\qem\qon{b}
  \qem\Qno\qem\Qno\qem
  \qon{c}\qem\qoff{d}\qem\Qno
\mid
  \Uu
\,\rangle
\\
&=
\mu_\Uu([a,b] \times [c,d])
\end{align*}
This proves $\mu_\Uu(R) \leq \mu_\Vv(R^\delta)$. The inequality $\mu_\Vv(R) \leq \mu_\Uu(R^\delta)$ follows by symmetry.
\end{proof}

Recall the measures at infinity defined in section~\ref{subsec:infinity}. By considering the appropriate limits, we immediately have:

\begin{proposition}[Box inequalities at infinity]
\label{prop:inf-box}
Let $\mu, \nu$ be r-measures on $\RR^2$ which satisfy a one-sided box inequality with parameter~$\delta$
\[
\mu(R) \leq \nu(R^\delta)
\]
for all rectangles $R \in \rect(\RR^2)$.
Then
\begin{alignat*}{2}
\mu([a,b], -\infty) &\leq \nu([A,B], -\infty),
&
\qquad
\mu(-\infty, [c,d]) &\leq \nu(-\infty, [C,D]),
\\
\mu([a,b], +\infty) &\leq \nu([A,B], +\infty),
&
\qquad
\mu(+\infty, [c,d]) &\leq \nu(+\infty, [C,D]),
\end{alignat*}
for all $a < b$ and $c < d$; and
\begin{alignat*}{2}
\mu(-\infty,-\infty) &\leq \nu(-\infty,-\infty),
&
\qquad
\mu(+\infty,-\infty) &\leq \nu(+\infty,-\infty),
\\
\mu(-\infty,+\infty) &\leq \nu(-\infty,+\infty),
&
\qquad
\mu(+\infty,+\infty) &\leq \nu(+\infty,+\infty).
\end{alignat*}
Here $A = a-\delta$, $B = b+\delta$, $C = c-\delta$, $D = d+\delta$.
\qed
\end{proposition}

Consequently, if $\Uu, \Vv$ are $\delta$-interleaved persistence modules then $\mu_\Uu, \mu_\Vv$ satisfy (two-sided) box inequalities on $(-\infty,\Rr)$ and $(\Rr, +\infty)$ as well as the equality $\mu_\Uu(-\infty,+\infty) = \mu_\Vv(-\infty, +\infty)$.

%--------------------------------------------------
\subsection{The measure stability theorem}
\label{subsec:mstability}

We now embed Theorem~\ref{thm:stability2} as a special case of a stability theorem for the diagrams of abstract r-measures. The more general statement is no more difficult%
\footnote{In fact it's a little easier to prove, because the compactness argument for diagrams with infinitely many points can be conducted more cleanly in this generality.}
to prove, and seems to be the natural home for the result.

Let $\Dd$ be an open subset of $\RR^2$. For $\alpha \in \Dd$, define the {\bf exit distance} of~$\alpha$ to be
\[
\exit(\alpha,\Dd)
= \ellinf(\alpha, \RR^2 - \Dd)
= \min\left( \ellinf(\alpha,x) \mid x \in \RR^2 - \Dd \right).
\]
For instance, for the extended half-plane we have $\exit(\alpha,\Upper) = \ellinf(\alpha, \Delta)$.

Let $\Aa, \Bb$ be multisets in~$\Dd$. A {\bf $\delta$-matching} between $\Aa, \Bb$ is a partial matching $\Mm \subset \Aa \times \Bb$ such that 
\begin{alignat*}{2}
\ellinf(\alpha, \beta) &\leq \delta
	&&\qquad\text{if $\alpha,\beta$ are matched,}
\\
\exit(\alpha, \Dd) &\leq \delta
	&&\qquad\text{if $\alpha \in \Aa$ is unmatched,}
\\
\exit(\beta, \Dd) &\leq \delta
	&&\qquad\text{if $\beta \in \Bb$ is unmatched.}
\end{alignat*}
If $\Dd$ is not clear from the context, we refer to~$\Mm$ as a `$\delta$-matching between $(\Aa, \Dd)$ and $(\Bb, \Dd)$'.

With the same proof as Proposition~\ref{prop:bottle-triangle}, we have:

\begin{proposition}[triangle inequality]
If $\Aa, \Bb, \Cc$ are multisets in~$\Dd$ and there exists a $\delta_1$-matching between $(\Aa, \Dd), (\Bb,\Dd)$ and a $\delta_2$-matching between $(\Bb,\Dd), (\Cc,\Dd)$, then there exists a $(\delta_1+\delta_2)$-matching between $(\Aa,\Dd), (\Cc,\Dd)$.
\qed
\end{proposition}

Now for the main theorem.

\begin{theorem}[stability for finite measures]
\label{thm:m-stability}
Suppose $( \mu_x \mid x \in [0, \delta] )$ is a 1-parameter family of finite r-measures on an open set~$\Dd \subseteq \RR^2$. Suppose for all $x,y \in [0,\delta]$ the box inequality
\[
\mu_x(R) \leq \mu_y(R^{|y-x|})
\]
holds for all rectangles $R$ whose $|y-x|$-thickening $R^{|y-x|}$ belongs to~$\rect(\Dd)$. Then there exists a $\delta$-matching between the undecorated diagrams $(\dgm(\mu_{0}), \Dd)$ and $(\dgm(\mu_{\delta}), \Dd)$.
\end{theorem}

In view of the Interpolation Lemma~(\ref{lem:interpolation}), this implies Theorem~\ref{thm:stability2} (take $\mu_x = \mu(\Uu_x)$ and $\Dd = \UpperInt$)
and therefore the stability theorem~(\ref{thm:isometry}$''$) for q-tame modules.

\begin{example}
The existence of a 1-parameter family interpolating between $\mu_0$ and~$\mu_\delta$ may seem unnecessarily strong. It is natural to hope that two measures $\mu, \nu$ which satisfy the (two-sided) box inequality with parameter~$\delta$ will have diagrams $\dgm(\mu), \dgm(\nu)$ which are $\delta$-matched. This is simply not true, and in fact there is no universal bound on the bottleneck distance between the two diagrams. See Figure~\ref{fig:badbottle}.
\begin{figure}
\includegraphics[scale=0.6]{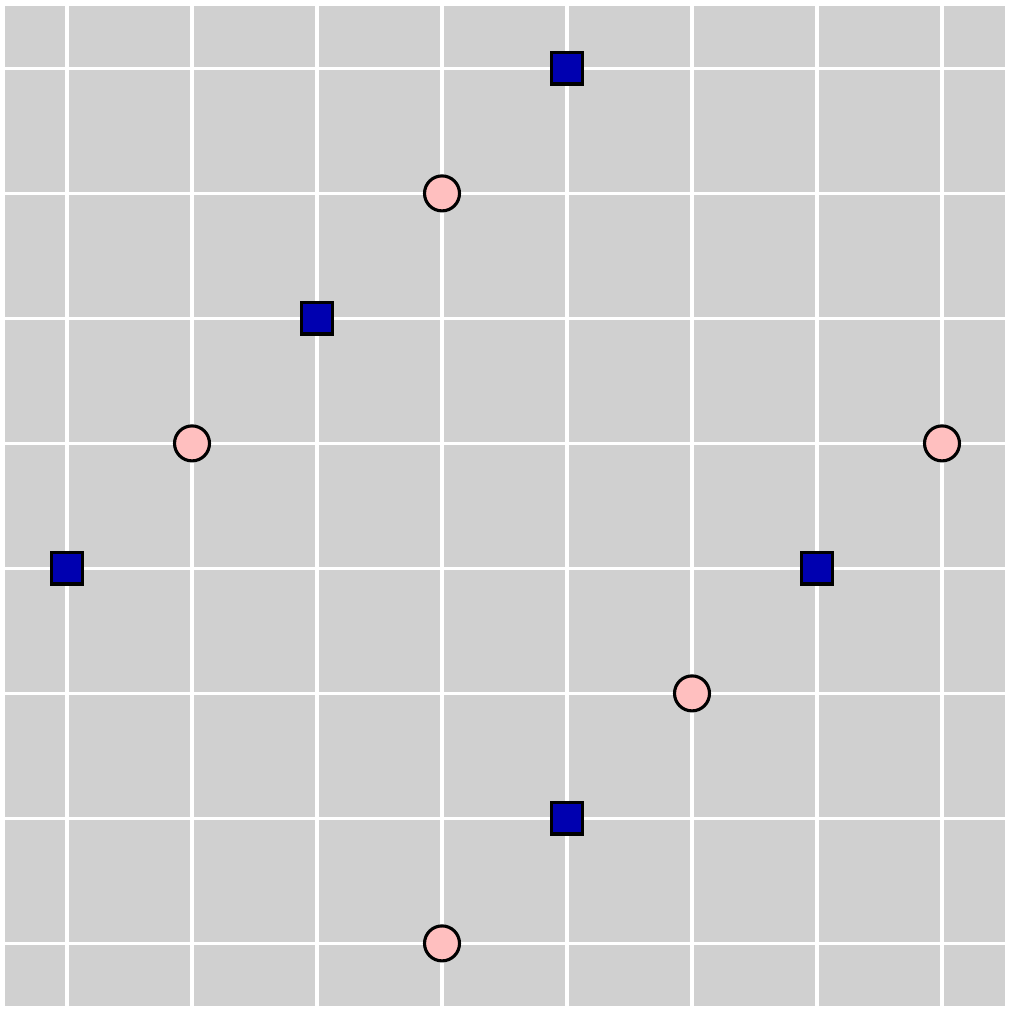}
\caption{The two diagrams (5 dark blue squares; 5 light pink circles) have box distance~1 and bottleneck distance~3. Generalising this example, one can exhibit a pair of diagrams with $4k+1$ points each, which have box distance~1 and bottleneck distance $2k+1$.
}
\label{fig:badbottle}
\end{figure}
\end{example}

Our goal for the rest of this section is to prove Theorem~\ref{thm:m-stability}. Parts 1 and~2 closely follow the method of Cohen-Steiner, Edelsbrunner and Harer~\cite{CohenSteiner_E_H_2007}.

Afterwards, in section~\ref{subsec:mstability2}, we generalise the theorem to r-measures that are not finite.

{\bf\small Initial remark.} Because the metric $\ellinf$ separates $\RR^2$ into nine strata (the standard plane, the four lines at infinity, and the four points at infinity), we seek separate $\delta$-matchings for each stratum that meets~$\Dd$. We begin with the points in the standard plane.

{\bf\small Temporary hypothesis.} Suppose initially that $\Dd \subseteq \Rr^2$.

{\bf\small Part 1.} {\it The Hausdorff distance between $(\dgm(\mu_x), \Dd)$, and $(\dgm(\mu_y), \Dd)$ is at most $|y-x|$.}

Write $\Aa = \dgm(\mu_x)$, $\Bb = \dgm(\mu_y)$, and $\eta = |y-x|$. The assertion is understood to mean:

\begin{itemize}
\item
If $\alpha \in \Aa$ and $\exit(\alpha, \Dd) > \eta$, then there exists $\beta \in \Bb$ with $\ellinf(\alpha, \beta) \leq \eta$.

\item
If $\beta \in \Bb$ and $\exit(\beta, \Dd) > \eta$, then there exists $\alpha \in \Aa$ with $\ellinf(\alpha, \beta) \leq \eta$.
\end{itemize}

\begin{proof}
By symmetry, it is enough to prove the first statement.

Given such an~$\alpha$, let $\epsilon > 0$ be small enough that $\eta + \epsilon < \exit(\alpha,\Dd)$. Then the box inequality gives
\[
1 \leq \mu_x(\alpha^\epsilon) \leq \mu_y(\alpha^{\eta+\epsilon})
\]
so there is at least one point of~$\Bb$ in the square $\alpha^{\eta+\epsilon}$. This is true for all sufficiently small $\epsilon > 0$, and moreover $\Bb$ is locally finite. Therefore there is at least one point of~$\Bb$ in $\alpha^\eta$.
\end{proof}

Henceforth, we will write $\Aa_x = \dgm(\mu_x)$ for all~$x$.

{\bf\small Part 2.} {\it The theorem is true if $\Aa_x$ has finite cardinality for all~$x$.}

\begin{proof}
(i) The triangle inequality for matchings includes the implication
\[
\left.
\begin{array}{l}
	\text{$\Aa_0, \Aa_x$ are $x$-matched}
	\\
	\text{$\Aa_x, \Aa_y$ are $(y-x)$-matched}
\end{array}
\right\}
\Rightarrow
\text{$\Aa_0, \Aa_y$ are $y$-matched}
\]
whenever $0 < x < y$.

(ii) We claim that for every $x \in [0,\delta]$ there exists $\rho(x) > 0$ such that $\Aa_x, \Aa_y$ are $|y-x|$-matched whenever $y \in [0,\delta]$ with $|y-x| < \rho(x)$.

Suppose $\alpha_1, \dots, \alpha_k$ is an enumeration of the distinct points of~$\Aa_x$, with respective multiplicities $n_1, \dots, n_k$. Let $\rho(x)$ be chosen to satisfy the following finite set of constraints:
\[
0< \rho(x) \leq 
\begin{cases}
{\textstyle \half} \exit(\alpha_i, \Dd)
	&\text{all~$i$}
\\
{\textstyle \half} \ellinf(\alpha_i, \alpha_j)
	&\text{all $i,j$ distinct}
\end{cases}
\]
We must show that if $|y-x| < \rho(x)$ then $\Aa_x, \Aa_y$ are $|y-x|$-matched.

Write $\eta = |y-x|$ and let
\[
(\Rr^2 - \Dd)^{\eta} =
\left\{ \alpha \in \Dd \mid \exit(\alpha, \Dd) \leq \eta \right\}.
\]
It follows from Part~1 that $\Aa_y$ is contained entirely in the closed set
\[
(\Rr^2 - \Dd)^{\eta} \cup
%\left( 
\alpha_1^\eta \cup \dots \cup \alpha_k^\eta
%\right)
\]
and it follows from the definition of $\rho(x)$ that the terms in the union are disjoint.
It is easy to count the points of~$\Aa_y$ in each square $\alpha_i^\eta$.
Let $\epsilon > 0$ be small enough that $2\eta + \epsilon < 2\rho(x)$. Then the box inequality gives
\[
n_i
	= \mu_x(\alpha_i^\epsilon)
	\leq \mu_y(\alpha_i^{\eta+\epsilon})
	\leq \mu_x(\alpha_i^{2\eta+\epsilon})
	= n_i.
\]
Thus $\mu_y(\alpha_i^{\eta+\epsilon}) = n_i$ for all small $\epsilon > 0$. We conclude that the square $\alpha_i^\eta$ contains precisely $n_i$ points of~$\Aa_y$.

This completes the proof of~(ii), because we can match the $n_i$ copies of~$\alpha_i$ with the $n_i$~points of $\Aa_y$ in the square $\alpha_i^\eta$, for each~$i$, to define an $\eta$-matching between $(\Aa_x,\Dd)$, $(\Aa_y,\Dd)$. All points of $\Aa_x$ are matched, and the only unmatched points of $\Aa_y$ lie in $\Rr^2-\Dd$ and do not need to be matched.

\medskip
Items (i) and~(ii) formally imply that $\Aa_0, \Aa_\delta$ are $\delta$-matched, using the standard Heine--Borel argument.
Indeed, let
\[
m = \sup( x \in [0,\delta] \mid \text{$\Aa_0$ and $\Aa_x$ are $x$-matched} ).
\]
First, $m$ is positive; specifically $m \geq \rho(0)$.
Applying (i) to $0 < m' < m$, where $\Aa_0, \Aa_{m'}$ are $m'$-matched and $m - m' < \rho(m)$, we deduce that $\Aa_0, \Aa_m$ are $m$-matched.
Suppose $m < \delta$. Applying (i) to $0 < m < m''$, where $m'' - m < \rho(m)$, we deduce that $\Aa_0, \Aa_{m''}$ are $m''$-matched. This contradicts the definition of~$m$. Therefore $m = \delta$, and $\Aa_0, \Aa_\delta$ are $\delta$-matched.
\end{proof}

{\bf\small Part 3.} {\it The theorem is true without assuming finite cardinality.}

\begin{proof}
Let $(\Dd_n) $ be an increasing sequence of open subsets of~$\Dd$ whose union equals~$\Dd$ and such that each~$\Dd_n$ has compact closure. Because $\Aa_x$ is locally finite, it follows that $\Aa_x \cap {\Dd_n}$ is finite for all $x, n$. We can therefore restrict the family of measures to each $\Dd_n$ in turn, and apply Part~2 to get a $\delta$-matching $\Mm_n$ between $(\Aa_0 \cap \Dd_n, \Dd_n)$ and $(\Aa_\delta \cap \Dd_n, \Dd_n)$.

We now take a limit $\Mm$ of the partial matchings~$\Mm_n$, using the construction in the proof of Theorem~\ref{thm:compact}.
(This works because $\Aa_0, \Aa_\delta$ are locally finite and therefore countable.)
Let $\chi, \chi_n$ denote the indicator functions of $\Mm, \Mm_n$.
As before, we have the Lemma that for any finite subset $\Ff \subset \Aa_0 \times \Aa_\delta$, there are infinitely many $n \in \Nn$ for which
\[
\chi(\alpha, \beta) = \chi_n(\alpha, \beta)
\]
for all $(\alpha, \beta) \in \Ff$.

We must show that $\Mm$ is a $\delta$-matching between $(\Aa_0, \Dd)$ and $(\Aa_\delta, \Dd)$. It is immediate that each matched pair is separated by at most~$\delta$, since this is true for every~$\Mm_n$. The argument that each $\alpha$ is matched with at most one~$\beta$, and vice versa, is the same as before.

Finally, suppose $\alpha \in \Aa_0$ with $\exit(\alpha, \Dd) > \delta$. The square $\alpha^\delta$ is contained in~$\Dd$ and is compact, and therefore is contained in $\Dd_n$ for sufficiently large~$n$. This means that $\exit(\alpha, \Dd_n) > \delta$ and hence $\alpha$ is matched in $\Mm_n$ for sufficiently large~$n$. 
Now $\alpha$ has only finitely many $\delta$-neighbours $\beta_1, \dots, \beta_k$ in the locally finite set~$\Aa_\delta$, so by the Lemma there are infinitely many~$n$ such that $\chi(\alpha, \beta_i) = \chi_n(\alpha, \beta_i)$ for all~$i$. By taking a sufficiently large such~$n$, we conclude that
\[
\chi(\alpha, \beta_i) = \chi_n(\alpha,\beta_i) = 1
\]
for some~$i$. Thus $\alpha$ is matched.

By symmetry, any $\beta \in \Aa_\delta$ with $\exit(\beta, \Delta) > \delta$ is matched in~$\Mm$ to some~$\alpha \in \Aa_0$.

It follows that $\Mm$ is the required $\delta$-matching.
\end{proof}

{\bf\small The theorem at infinity.} Now suppose $\Dd \subseteq \RR^2$ meets any of the strata at infinity. For each of the four lines at infinity, the 3-part proof given above works almost verbatim, if we replace $\Dd$ with its intersection with the chosen line, and each r-measure $\mu_x$ with the corresponding measure at infinity. The other change is to replace the word `square' with the word `interval'. The necessary box inequality at infinity is found in Proposition~\ref{prop:inf-box}.

For the four corners $(\pm\infty,\pm\infty)$, it is easier still: the box inequality at each corner implies that $\mu_0, \mu_\delta$ have the same multiplicity there. The interpolating measures are not needed.
\qed

This completes the proof of the stability theorem for finite measures on an open domain~$\Dd$, and hence the stability theorem for q-tame persistence modules, and hence the isometry theorem for q-tame persistence modules.

%--------------------------------------------------
\subsection{The measure stability theorem (continued)}
\label{subsec:mstability2}

The stability theorem generalises to measures that are not necessarily finite. By the extension convention, we may suppose that the measures are defined on~$\RR^2$ (rather than just a subset of~$\RR^2$). Given a 1-parameter family $(\mu_x \mid x \in [0,\delta])$, the finite interiors
\[
\fin_x = \fint(\mu_x)
\]
now depend on~$x$; whereas previously we had $\fin_x = \Dd$ for all~$x$.

For $\fin \subset \RR^2$ an open set and $\delta \geq 0$, the `reverse offset' is the open set
\[
\fin^{-\delta} =
\left\{ \alpha \in \fin \mid \exit(\alpha, \fin) > \delta \right\}
=
\left\{ \alpha \in \fin \mid \alpha^\delta \subset \fin \right\}.
\]
Intuitively, this shrinks~$\fin$ by~$\delta$ at the boundary. Clearly $\fin \supseteq \gin$ implies $\fin^{-\delta} \supseteq \gin^{-\delta}$, and $(\fin^{-\delta_1})^{-\delta_2} = \fin^{-(\delta_1+\delta_2)}$.
Note also that $(\fin\cap\gin)^{-\delta} = \fin^{-\delta} \cap \gin^{-\delta}$. This is easiest to see from the second characterisation.

\begin{remark}
The operation $[\cdot]^{-\delta}$ has no effect on the corners at infinity, and acts independently on the standard plane and on the four lines at infinity.
\end{remark}

We define $\delta$-matchings for multisets in unequal domains.
Let $\fin, \gin$ be open subsets of $\RR^2$, let $\Aa, \Bb$ be multisets in $\fin, \gin$ respectively, and let $\delta > 0$. A {\bf $\delta$-matching between $(\Aa, \fin), (\Bb, \gin)$} is a partial matching $\Mm$ between $\Aa,\Bb$ such that the following four conditions hold:
\begin{itemize}
\item
$\fin \supseteq \gin^{-\delta}$ and $\gin \supseteq \fin^{-\delta}$,

\item
if $(\alpha, \beta) \in \Mm$ then $\ellinf(\alpha,\beta) \leq \delta$,

\item
every $\alpha \in \Aa \cap \gin^{-\delta}$ is matched with some $\beta \in \Bb$,

\item
every $\beta \in \Bb \cap \fin^{-\delta}$ is matched with some $\alpha \in \Aa$.
\end{itemize}
The first of these is a compatibility condition between the domains: they cannot be too unequal. This is automatic if $\fin = \gin$, which is why we haven't seen it before.
Notice the cross-over in the last two conditions: a point in~$\Aa$ is allowed to be unmatched only if it is close to the boundary of $\Bb$'s domain~$\gin$, and vice versa.

\begin{proposition}[triangle inequality]
\label{prop:m-triangle2}
If $\Aa, \Bb, \Cc$ are multisets in~$\fin, \gin, \hin$ respectively, and there exist a $\delta_1$-matching between $(\Aa, \fin), (\Bb,\gin)$ and a $\delta_2$-matching between $(\Bb,\gin), (\Cc,\hin)$, then there exists a $(\delta_1+\delta_2)$-matching between $(\Aa,\fin), (\Cc,\hin)$.
\end{proposition}

\begin{proof}
As usual, compose the two partial matchings to get a partial matching $\Mm$ between $\Aa, \Cc$. Writing $\delta = \delta_1 + \delta_2$, we must check that this is a $\delta$-matching between $(\Aa,\fin), (\Cc,\hin)$. For first condition we see that
\[
\fin \supseteq \gin^{-\delta_1} \supseteq (\hin^{-\delta_2})^{-\delta_1} = \hin^{-\delta}
\quad
\text{and}
\quad
\hin \supseteq \gin^{-\delta_2} \supseteq (\fin^{-\delta_1})^{-\delta_2} = \fin^{-\delta}.
\]
The second condition follows from the triangle inequality for~$\ellinf$.
For the third condition, if $\alpha \in \Aa$ lies in $\hin^{-\delta}$ then by the inclusion above it lies in~$\gin^{-\delta_1}$. Therefore $\alpha$ is matched with $\beta \in \Bb$. Moreover $\beta$ must then lie in $\hin^{-(\delta-\delta_1)} = \hin^{-\delta_2}$ and so is matched with $\gamma \in \Cc$. The fourth condition follows by symmetry.
\end{proof}

\begin{remark}
There is no triangle inequality if the condition on the domains is dropped.
\end{remark}

Here is the main theorem of this section and the final theorem of the paper. Again we use the abbreviation $\fin_x = \fint(\mu_x)$ for the finite interiors.

\begin{theorem}[stability for measures]
\label{thm:m-stability2}
Suppose $( \mu_x \mid x \in [0, \delta] )$ is a 1-parameter family of r-measures on~$\RR^2$. Suppose for all $x,y \in [0,\delta]$ the box inequality
\[
\mu_x(R) \leq \mu_y(R^{|y-x|})
\]
holds for all rectangles $R \in \rect(\RR^2)$. Then there exists a $\delta$-matching between the undecorated diagrams $(\dgm(\mu_{0}), \fin_0)$ and $(\dgm(\mu_{\delta}), \fin_\delta)$.
\end{theorem}

The first condition for a $\delta$-matching, on the domains $\fin_0, \fin_\delta$, can be checked easily:

\begin{proposition}
\label{prop:stable-fint}
Under the hypotheses of Theorem~\ref{thm:m-stability2}, we have inclusions
\[
\fin_x \supseteq \fin_y^{-|y-x|}
\]
for all $x,y \in [0,\delta]$.
\end{proposition}

\begin{proof}
Suppose $\alpha \in \fin_y^{-|y-x|}$, then equivalently $\alpha^{|y-x|} \subset \fin_y$. Since the square $\alpha^{|y-x|}$ is compact and $\fin_y$ is open, there exists $\epsilon > 0$ such that $\alpha^{|y-x| + \epsilon} \subset \fin_y$. The box inequality gives
\[
\mu_x(\alpha^\epsilon) \leq \mu_y(\alpha^{|y-x| + \epsilon})
\]
and the right-hand side is finite by Proposition~\ref{prop:fint-subtle}. Thus $\alpha \in \fint(\mu_x) = \fin_x$.
\end{proof}

\begin{proof}[Proof of Theorem~\ref{thm:m-stability2}]
The argument closely follows the proof of the stability theorem for finite measure, so we will confine ourselves to indicating the necessary modifications. We use the abbreviation $\Aa_x = \dgm(\mu_x)$.

{\bf\small Initial remark.}
Recall that the proof is conducted separately for each of the nine strata. The four corners at infinity are handled easily (each corner belongs to both $\fin_0$ and~$\fin_\delta$, or to neither; in the former case the $\mu_0, \mu_\delta$ multiplicities agree). The proof is described for the points in the standard plane. The same proof applies to each of the four lines at infinity, replacing each $\mu_x$ with the corresponding measure at infinity.

{\bf\small Part 1.} {\it The Hausdorff distance between $(\Aa_x, \fin_x)$, and $(\Aa_y, \fin_y)$ is at most $\eta = |y-x|$.}

The assertion is understood to mean:

\begin{vlist}
\item
If $\alpha \in \Aa_x$ and $\exit(\alpha, \fin_y) > \eta$, then there exists $\beta \in \Aa_y$ with $\ellinf(\alpha, \beta) \leq \eta$.

\item
If $\beta \in \Aa_y$ and $\exit(\beta, \fin_x) > \eta$, then there exists $\alpha \in \Aa_x$ with $\ellinf(\alpha, \beta) \leq \eta$.
\end{vlist}

Proof: for all $\epsilon > 0$ with $\eta + \epsilon < \exit(\alpha, \fin_y)$, we have
$
1 \leq \mu_x(\alpha^\epsilon) \leq \mu_y(\alpha^{\eta+\epsilon})
$
so there is at least one point of~$\Aa_y$ in $\alpha^\eta$.

\medskip
{\bf\small Part 2.} {\it The theorem is true if $\Aa_x$ has finite cardinality for all~$x$.}

Item~(i) is given by the triangle inequality (Proposition~\ref{prop:m-triangle2}). 

Item~(ii) uses the same strategy as before. Let $(\alpha_i)$ be a finite enumeration of the distinct points of $\Aa_x$, with respective multiplicities~$(n_i)$. Then $\rho(x)$ is chosen to satisfy:
\[
0< \rho(x) \leq 
\begin{cases}
{\textstyle \half} \exit(\alpha_i, \fin_x)
	&\text{all~$i$}
\\
{\textstyle \half} \ellinf(\alpha_i, \alpha_j)
	&\text{all $i,j$ distinct}
\end{cases}
\]
If $\eta = |y-x| < \rho(x)$, then Part~1 implies that $\Aa_y$ is contained in the disjoint union
\[
(\Rr^2 - \fin_x)^{\eta} \cup
\alpha_1^\eta \cup \dots \cup \alpha_k^\eta.
\]
The box inequality is then used to count precisely $n_i$~points of~$\Aa_y$ in the square $\alpha_i^\eta$. This defines a partial matching where all points of~$\Aa_x$ are matched and all points of $\Aa_y \cap \fin_x^{-\eta}$ are matched.

The formal deduction of Part~2 from (i) and~(ii) is unchanged, since it is a formal deduction.

\medskip
{\bf\small Part 3.} {\it The theorem is true without assuming finite cardinality.}

The idea is to restrict each measure $\mu_x$ to a relatively compact open subset $\hat\fin_x \subset \fin_x = \fint(\mu_x)$. The subsets satisfy the compatibility condition
\[
\hat\fin_x \supseteq \hat\fin_y^{-|y-x|}
\]
for all $x,y \in [0,\delta]$.

Specifically, for $\epsilon > 0$ and $r > \delta$, let
\[
\hat\fin_x = \fin_x^{-\epsilon} \cap \qin^r
\]
where $\qin^r = (-r,r)\times(-r,r)$ is the open $\ellinf$-disk of radius~$r$.
Define a function on rectangles as follows:
\[
\hat\mu_x(R) = 
\begin{cases}
\mu_x(R)
	&\text{if $R \subset \hat\fin_x$}
\\
\infty
	&\text{otherwise}
\end{cases}
\]
It is easy to check that $\hat\mu_x$ is an r-measure (additivity still holds), that $\fint(\hat\mu_x) = \hat\fin_x$, and that $\dgm(\hat\mu_x) = \dgm(\mu_x) \cap \hat\fin_x$.

\begin{lemma*}
The family $(\hat\mu_x)$ satisfies the box inequality $\hat\mu_x(R) \leq \hat\mu_y(R^{|y-x|})$ for all $x, y \in [0,\delta]$.
\end{lemma*}

\begin{proof}
Since the box inequality is assumed to hold for~$(\mu_x)$, it will automatically hold for $(\hat\mu_x)$; except possibly for rectangles~$R$ where the left-hand side of the inequality has become infinite while the right-hand side hasn't. This happens when $R \not\subset\hat\fin_x$ while $R^{|y-x|} \subset \hat\fin_y$, and we can prevent it by ensuring that $\hat\fin_x \supseteq \hat\fin_y^{-|y-x|}$.
And, indeed,
\[
\hat\fin_y^{-|y-x|}
=
(\fin_y^{-\epsilon} \cap \qin^{r})^{-|y-x|}
=
\fin_y^{-(\epsilon+|y-x|)} \cap \qin^{r-|y-x|}
\subseteq
\fin_x^{-\epsilon} \cap \qin^{r}
=
\hat\fin_x.
\qedhere
\]
\end{proof}

Since $\hat\fin_x$ has compact closure in~$\fin_x$, and $\Aa_x$ is locally finite, it follows that $\hat\Aa_x = \dgm(\hat\mu_x) = \Aa_x \cap \hat\fin_x$ has finite cardinality. We can therefore apply Part~2 to the family $(\hat\mu_x)$ to get a $\delta$-matching between $(\hat\Aa_0, \hat\fin_0)$ and  $(\hat\Aa_\delta, \hat\fin_\delta)$. This can be interpreted as a partial $\delta$-matching between $\Aa_0, \Aa_\delta$ where:
\begin{vlist}
\item
$\alpha \in \Aa_0$ is matched whenever $\alpha \in (\fin_\delta^{-\epsilon} \cap \omega^{r})^{-\delta} = \fin_\delta^{-(\delta+\epsilon)} \cap \omega^{r-\delta}$
\item
$\beta \in \Aa_\delta$ is matched whenever $\beta \in (\fin_0^{-\epsilon} \cap \omega^{r})^{-\delta} = \fin_0^{-(\delta+\epsilon)} \cap \omega^{r-\delta}$
\end{vlist}

Repeat this argument for a sequence $(\epsilon_n, r_n)$ where $\epsilon_n \to 0$ and $r_n \to +\infty$. This gives a sequence of $\delta$-matchings~$\Mm_n$, and we can form a limit~$\Mm$ as before.

If $\alpha \in \Aa_0 \cap \fin_\delta^{-\delta}$ then eventually $\alpha \in \fin_\delta^{-(\delta+\epsilon_n)} \cap \omega^{r_n-\delta}$ and so $\alpha$ is matched by~$\Mm_n$ for all sufficiently large~$n$. The same is true for $\beta \in \Aa_\delta \cap \fin_0^{-\delta}$. With this information, we can complete the usual proof that $\Mm$ is a $\delta$-matching between $(\Aa_0, \fin_0)$ and $(\Aa_\delta, \fin_\delta)$.
\end{proof}

This concludes the proof of Theorem~\ref{thm:m-stability2}. Here is a sample consequence.

\begin{example}[Stability of the Webb module]
Let $\Vv$ be a persistence module which is $\delta$-interleaved with the module~$\Ww$ of Example~\ref{ex:webb2}. By interpolation and the box lemma, the measure stability theorem applies here. We get a $\delta$-matching between $(\dgm(\mu_\Vv), \fint(\mu_\Vv))$ and $(\dgm(\mu_\Ww), \fint(\mu_\Ww))$.
This amounts to the following.

\begin{vlist}
\item
In the finite part $\upper$ of the half-plane:

\medskip
\noindent
The singular support of~$\mu_\Vv$ is contained in the diagonal strip $\Delta_{[0,\delta]}$. Each point of $\dgm(\mu_\Vv)$ outside this strip is matched with some point $(-n,0) \in \dgm(\mu_\Ww)$. Conversely, the only unmatched points of $\dgm(\mu_\Ww)$ must lie within distance~$\delta$ of the diagonal or of the singular support of~$\mu_\Vv$.

\medskip
\noindent
In particular, if $\delta < \frac{1}{4}$ then all points of~$\dgm(\mu_\Ww)$ are matched.

\item
On the line $(-\infty,\Rr)$: All points and singularities of $\mu_\Vv$ are contained in the interval $(-\infty, [-\delta,+\delta])$. There is at least one singular point.

\item
On $(\Rr, +\infty)$ and at $(-\infty, +\infty)$: The measure $\mu_\Vv$ has no points or singularities.

\end{vlist}
\end{example}

%--------------------------------------------------
\section{Examples}
\label{sec:examples}

%--------------------------------------------------
\subsection{Partial interleavings}
\label{subsec:p-interleavings}

In some practical data analysis situations, one considers persistence modules which are only partially interleaved.
One such scenario is presented by Chazal et al.\ in the context of clustering by mode-seeking~\cite{Chazal_G_O_S_2012}. A filtered simplicial complex on an input point cloud is compared with the sublevelset filtration of the density function it was sampled from. In low-density regions, the sample is too sparse to expect there to be an interleaving.
Nevertheless, there is interleaving when the density is sufficiently high.

This leads to the following notion of partial interleaving, adapted from~\cite{Chazal_G_O_S_2012}.
Two persistence modules $\Uu$ and $\Vv$ are said to be \textbf{$\delta$-interleaved up to time~$t_0$} if there are maps $\phi_t:U_t\to V_{t+\delta}$ and $\psi_t:V_t\to U_{t+\delta}$ defined for all $t\leq t_0$, such that the diagrams of equation~\ref{eq:int_expansive} commute for all values $s<t\leq t_0$; that is, for all values where the maps are defined.

A weaker version of the stability theorem, illustrated in Figure~\ref{fig:diagram_quadrants}~(left), can be proven:
\begin{theorem}[from~\cite{Chazal_G_O_S_2012}]
\label{th:partial_interleaving_stability}
Let $\Uu$ and $\Vv$ be two q-tame persistence modules that are $\delta$-interleaved up to time~$t_0$. Then, there is a partial matching $\Mm \subset \dgm(\Uu) \times \dgm(\Vv)$ with the following properties:

\begin{vlist}
\item
Points $(p,q)$ in either diagram for which $\half |p-q| \leq \delta$ are not required to be matched.

\item
Points $(p,q)$ in either diagram for which $p \geq t_0 - \delta$ are not required to be matched.
\end{vlist}

All other points must be matched. Then:

\begin{vlist}
\item
If $\alpha, \beta$ are matched, then the $p$-coordinates of $\alpha, \beta$ differ by at most~$\delta$.

\item
If $\alpha, \beta$ are matched and one of $\alpha, \beta$ lies below the line $q = t_0$, then $\ellinf(\alpha, \beta) \leq \delta$.

\end{vlist}
\end{theorem}
\definecolor{grey6}{rgb}{0.6,0.6,0.6}
\begin{figure}[tb]
\centerline{
\includegraphics[height=6cm]{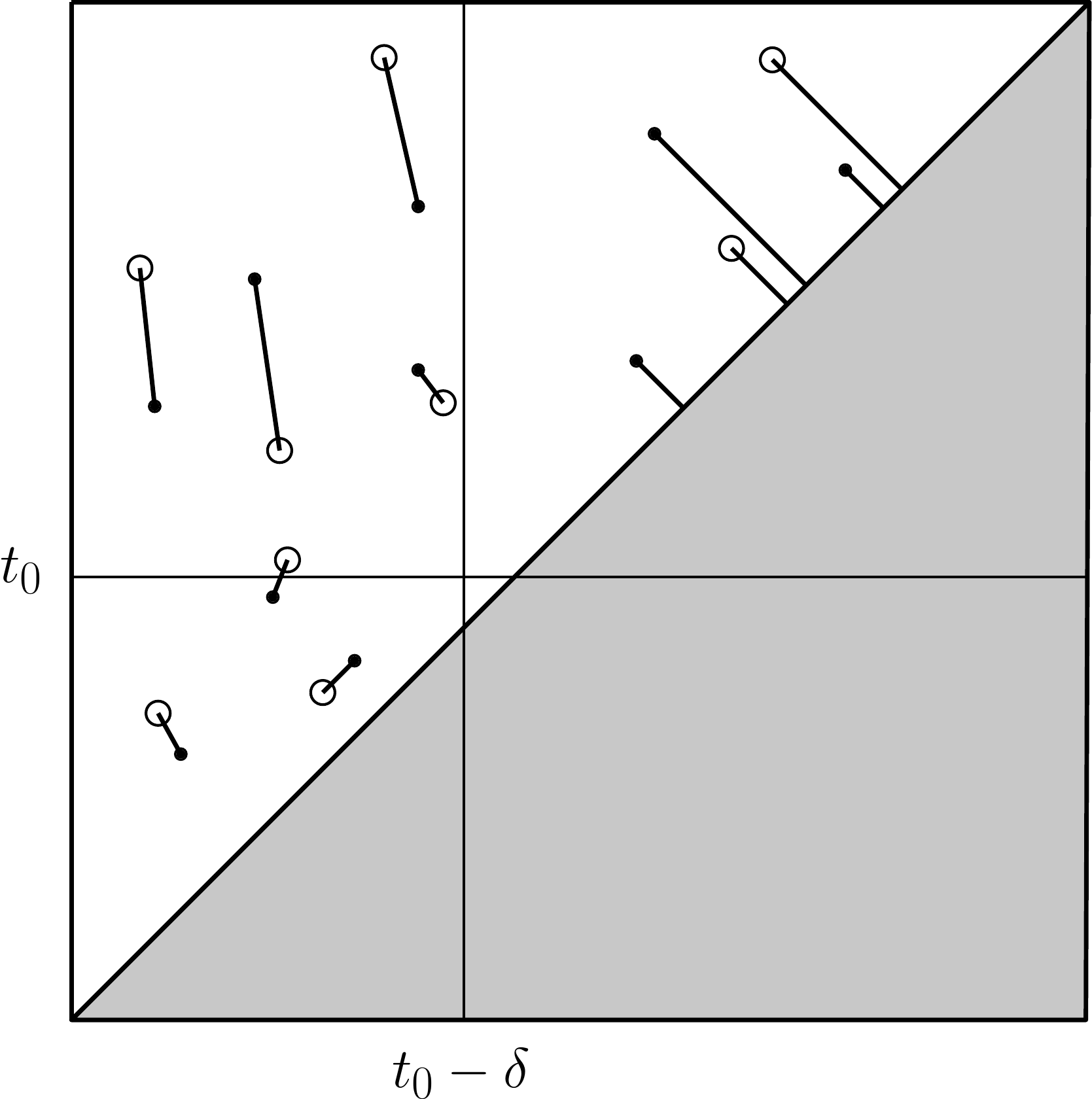}
\hspace{0.05\textwidth}
\includegraphics[height=6cm]{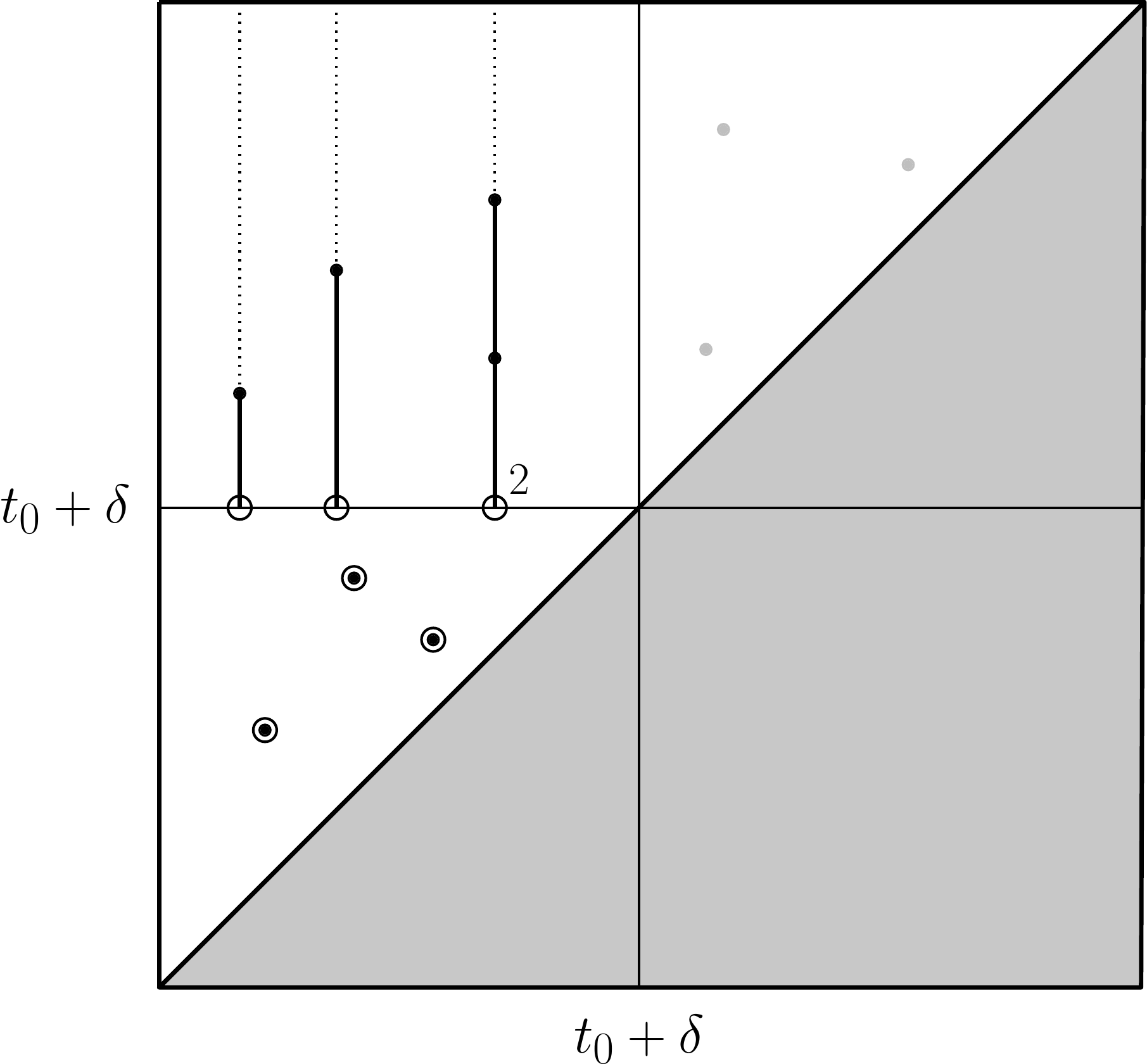}}
\caption{
Left: the partial matching of Theorem~\ref{th:partial_interleaving_stability} between $\dgm(\Uu)$~($\bullet$) and $\dgm(\Vv)$~($\circ$).
Right: projection from $\dgm(\Uu)$~(${\bullet}$ and {\color{grey6}$\bullet$}) to $\dgm(\tilde\Uu)$~($\circ$). The grey points are the ones that disappear.
}
\label{fig:diagram_quadrants} 
\end{figure}
%

%%
%%
%%\newpage
%%A weaker version of the stability theorem, illustrated in
%%Figure~\ref{fig:diagram_quadrants}~(left), can be proven:
%%%
%%\begin{theorem}[from~\cite{Chazal_G_O_S_2012}]\label{th:partial_interleaving_stability}
%%Let $\Uu$ and $\Vv$ be two q-tame persistence modules that are
%%$\delta$-interleaved up to time $t_0$. Then, there is a partial
%%matching $\Mm \subset \dgm(\Uu) \times \dgm(\Vv)$ such that:
%%%
%%\begin{vlist}
%%\item
%%$\forall \alpha\in \dgm(\Uu)\cap [-\infty, t_0]\times [-\infty, t_0]$, $\ellinf(\alpha,\Delta)\leq\delta$ \textrm{or} $\exists\beta$ s.t. $(\alpha,\beta)\in\Mm$ and $\ellinf(\alpha, \beta) \leq \delta$,
%%
%%\item
%%$\forall \beta\in \dgm(\Vv) \cap  [-\infty, t_0]\times [-\infty, t_0]$, $\ellinf(\beta,\Delta)\leq\delta$ \textrm{or} $\exists\alpha$ s.t. $(\alpha,\beta)\in\Mm$ and $\ellinf(\alpha, \beta) \leq \delta$,
%%
%%\item
%%$\forall \alpha\in \dgm(\Uu)\cap [-\infty, t_0]\times [t_0, +\infty]$, $\ellinf(\alpha,\Delta)\leq\delta$ \textrm{or} $\exists\beta$ s.t. $(\alpha,\beta)\in\Mm$ and $\ellx(\alpha, \beta) \leq \delta$,
%%
%%\item
%%$\forall \beta\in \dgm(\Vv) \cap  [-\infty, t_0]\times [t_0, +\infty]$, $\ellinf(\beta,\Delta)\leq\delta$ \textrm{or} $\exists\alpha$ s.t. $(\alpha,\beta)\in\Mm$ and $\ellx(\alpha, \beta) \leq \delta$,
%%\end{vlist}
%%%
%%where $\ellx(\alpha,\beta)$ stands for the absolute value of the difference of the $x$-coordinates of $\alpha,\beta$.
%%\end{theorem}
%%%
%%%

For the proof, we introduce two new persistence modules $\tilde \Uu, \tilde \Vv$.
\begin{alignat*}{4}
\tilde U_t &= U_t \quad&&\text{if $t \leq t_0 + \delta$}
\quad \text{and} \quad
\tilde U_t &= 0 \quad&&\text{otherwise}
\\
\tilde V_t &= V_t \quad&&\text{if $t \leq t_0 + \delta$}
\quad \text{and} \quad
\tilde V_t &= 0 \quad&&\text{otherwise}
\end{alignat*}
with maps
\begin{alignat*}{4}
\tilde u_s^t &= u_s^t \quad&&\text{if $t \leq t_0 + \delta$}
\quad \text{and} \quad
\tilde u_s^t &= 0 \quad&&\text{otherwise}
\\
\tilde v_s^t &= v_s^t \quad&&\text{if $t \leq t_0 + \delta$}
\quad \text{and} \quad
\tilde v_s^t &= 0 \quad&&\text{otherwise}
\end{alignat*}
for all $s\leq t$. We may call $\hat\Uu, \hat\Vv$ the \textbf{truncations} of $\Uu, \Vv$ to $(-\infty,T]$, where $T = t_0 + \delta$.

\begin{proof}
%The proof proceeds in three steps. 
There are three steps.

{\small\bf Step 1.}
The decorated diagram of a persistence module~$\Uu$ determines the decorated diagram of its truncation $\tilde\Uu$, in a straightforward way. Specifically, transform each point $(p^*, q^*) \in \Dgm(\Uu)$ as follows:
\[
(p^*, q^*) \mapsto
\begin{cases}
(p^*, q^*) & \text{if $q^* < T$}
\\
(p^*, T^+) & \text{if $p^* < T < q^*$}
\\
\text{disappears} \;& \text{if $T < p^*$}
\end{cases}
\tag{\P}\label{eq:partial-inter}
\]
Then $\Dgm(\tilde\Uu)$ is the result of this transformation. The consequent relationship between the undecorated diagrams is illustrated in Figure~\ref{fig:diagram_quadrants}~(right).

{\small\bf Step 2.}
If $\Uu, \Vv$ are $\delta$-interleaved up to time~$t_0$, then $\tilde\Uu, \tilde\Vv$ are $\delta$-interleaved.

Combining the first two steps we get the third.

{\small\bf Step 3.}
The stability theorem gives a $\delta$-matching between $\dgm(\tilde\Uu), \dgm(\tilde\Vv)$. This lifts to a matching between $\dgm(\Uu), \dgm(\Vv)$ which has the properties stated in the theorem.

\medskip
The second and third steps are straightforward. Only the first step requires any technical input, intuitively plausible as it may be. The framework developed in~\cite{Chazal_CS_G_G_O_2009} leads to a 2-page argument, presented in the appendix
of~\cite{Chazal_G_O_S_2012}. 

Here is a shorter proof. Write $\mu = \mu_\Uu$ and $\tilde\mu = \mu_{\tilde\Uu}$.
Let $\Aa$ denote the multiset obtained from $\Dgm(\Uu)$ by applying the transformation~\eqref{eq:partial-inter}. Consider an arbitrary rectangle $[a,b] \times [c,d] \in \rect(\Upper)$. We easily see:
\[
\card(\Aa|_{[a,b] \times [c,d]}) = 
\begin{cases}
\mu([a,b] \times [c,d]) & \text{if $d \leq T$}
\\
\mu([a,b] \times [c,+\infty]) & \text{if $c \leq T < d$}
\\
0 & \text{if $T < c$}
\end{cases}
\]
To show that we have correctly determined $\Dgm(\tilde\Uu)$, it suffices to show that $\card(\Aa|_{[a,b] \times [c,d]}) = \tilde\mu([a,b]\times[c,d])$ for all rectangles.
And indeed:
\begin{vlist}
\item
If $d \leq T$, then:
\begin{align*}
\tilde\mu([a,b] \times [c,d])
&=
\langle \qoff{a}\qem\qon{b}\qem\qon{c}\qem\qoff{d} \mid \tilde\Uu \rangle
\\
&=
\langle \qoff{a}\qem\qon{b}\qem\qon{c}\qem\qoff{d} \mid \Uu \rangle
=
\mu([a,b] \times [c,d])
\end{align*}

\item
If $c \leq T < d$, then:
\begin{align*}
\tilde\mu([a,b] \times [c,d])
&=
\langle \qoff{a}\qem\qon{b}\qem\qon{c}\qem\qoff{d} \mid \tilde\Uu \rangle
%\\
%&=
%\langle \qoff{a}\qem\qon{b}\qem\qon{c}\qem\qno \mid \tilde\Uu \rangle
\\
&=
\langle \qoff{a}\qem\qon{b}\qem\qon{c}\qem\qno \mid \Uu \rangle
=
\mu([a,b] \times [c,+\infty])
\end{align*}
since $\tilde{U}_d = 0$.

\item
If $T < c$, then:
\begin{align*}
\tilde\mu([a,b] \times [c,d])
&=
\langle \qoff{a}\qem\qon{b}\qem\qon{c}\qem\qoff{d} \mid \tilde\Uu \rangle
=
0
\end{align*}
since $\tilde{U}_c = 0$.
\end{vlist}
It follows that $\Dgm(\tilde\Uu) = \Aa$ as claimed.
\end{proof}

%--------------------------------------------------
\subsection{Extended persistence}
\label{subsec:extended}

Cohen-Steiner, Edelsbrunner and Harer~\cite{CS_E_H_2008} introduced extended persistence to capture the homological information carried by a pair $(X,f)$.
Some but not all of this information is recovered by the sublevelset persistence $\Hgr(\Xx_\sub)$.
The idea is to grow the space from the bottom up, through sublevelsets; and then to relativise the space from the top down, with superlevelsets. Extended persistence is the persistent homology of this sequence of spaces and pairs.

It is usually assumed that $(X,f)$ has finitely many homological critical points~$(a_i)$. One applies a homology functor to the finite sequence%
\footnote{We write $X^t = (X,f)^t = f^{-1}(-\infty,t]$ and $X_t = (X,f)_t = f^{-1}[t,+\infty)$ for sublevelsets and superlevelsets.
}
\[
X^{a_0} \to
X^{a_1} \to
\dots \to
X^{a_{n-1}} \to
X \to
(X, X_{a_{n}}) \to
\dots \to
(X, X_{a_{2}}) \to
(X, X_{a_{1}})
\]
to get a quiver representation. The indecomposable summands of this representation are interpreted as features, and are drawn as points in the `extended persistence diagram'. There are three kinds of feature:
\begin{vlist}
\item
ordinary features (which are born and die before the central~$X$);

\item
relative features (which are born and die after the central~$X$);

\item
extended features (which are born before the~$X$ and die after it).
\end{vlist}
The finiteness assumption is satisfied when $(X,f)$ is a compact manifold with a Morse function, or a compact polyhedron with a piecewise-linear map. In the former situation, there are extra symmetries (Poincar\'{e}, Lefschetz) which are explored in~\cite{CS_E_H_2008}.

In fact, it is perfectly straightforward to define the extended persistence diagram under a weaker hypothesis. Suppose $X$ is a compact polyhedron and $f$ is a continuous real-valued function on~$X$. Then:
\begin{vlist}
\item
$\rank\left( \Hgr(X^s) \to \Hgr(X^t) \right) < \infty$ whenever $s < t$; and

\item
$\rank\left( \Hgr(X, X_s) \to \Hgr(X, X_t) \right) < \infty$ whenever $s > t$.
\end{vlist}

The first of these facts is Theorem~\ref{thm:poly-q-tame}. The second is proved similarly, by factorising the map through some $\Hgr(X,Y)$, where $Y$ is a subpolyhedron of~$X$ nested between $X_s, X_t$.

Define the ordered set
\[
\rR = \{ \ov{t} \mid t \in \Rr \}
\quad
\text{ordered by}
\quad
\ov{s} \leq \ov{t} \;\Leftrightarrow\; s \geq t,
\]
thought of as a `backwards' copy of the real line, with bars under numbers to remind us.
For extended persistence we may work with the set
\[
\Rr_\ep
=
\Rr \cup \{ +\infty \} \cup \rR
\]
with the ordering $s < +\infty < \ov{t}$ for all $s, \ov{t}$.

The extended persistence module $\Xx_\ep = \Xx_\ep^f$ for $(X,f)$ is defined as follows:
\begin{alignat*}{3}
&V_t &&= \Hgr(X^t) && \text{for $t \in \Rr$}
\\
&V_{+\infty} &&= \Hgr(X)
\\
&V_{\ov{t}} &&= \Hgr(X,X_t) \quad && \text{for $\ov{t} \in \rR$}
\end{alignat*}
Since $\Rr_\ep$ is order-isomorphic to the real line, we may interpret $\Xx_\ep$ it as a persistence module over~$\Rr$. The two facts cited above imply that it is q-tame, so the decorated diagram is defined away from the diagonal.

Alternatively, we can define the extended persistence diagram in three pieces:
\begin{alignat*}{3}
\mu_\ord([a,b] \times [c,d])
&=
\langle \qoff{a}\qem\qon{b}\qem\qon{c}\qem\qoff{d} \rangle
\quad
&& \text{for $-\infty \leq a < b \leq c < d \leq +\infty$}
\\
\mu_\rel([\ov{a}, \ov{b}] \times [\ov{c}, \ov{d}])
&=
\langle \qoff{\ov{a}}\qem\qon{\ov{b}}\qem\qon{\ov{c}}\qem\qoff{\ov{d}} \rangle
&&
\text{for $+\ov{\infty} \leq \ov{a} < \ov{b} \leq \ov{c} < \ov{d} \leq -\ov{\infty}$}
\\
\mu_\ext([a, b] \times [\ov{c}, \ov{d}])
&=
\langle \qoff{{a}}\qem\qon{{b}}\qem\qon{\ov{c}}\qem\qoff{\ov{d}} \rangle
&&
\text{for $-\infty \leq \ov{a} < \ov{b} \leq +\infty$ and $+\ov{\infty} \leq \ov{c} < \ov{d} \leq -\ov{\infty}$}
\end{alignat*}
taking $V_{-\infty} = 0$ and $V_{-\ov{\infty}} = 0$ whenever needed.

The measures $\mu_{\text{ord}}, \mu_{\text{rel}}$ are defined over the half-plane $\Upper$, whereas $\mu_{\text{ext}}$ is defined over~$\RR^2$.

Stability for $\dgm_\ord$, $\dgm_\rel$ and~$\dgm_\ext$ is proved individually for each diagram. Given two functions $f,g$ which are $\delta$-close in the supremum norm, there are inclusions
\begin{alignat*}{2}
&
(X,f)^t \subseteq (X,g)^{t+\delta}
\qquad
&
(X,f)_t \subseteq (X,g)_{t-\delta}
\\
&
(X,g)^t \subseteq (X,f)^{t+\delta}
&
(X,g)_t \subseteq (X,f)_{t-\delta}
\end{alignat*}
using which we can prove the box lemma for each measure. Since linear combinations of continuous functions are continuous, we can interpolate between $f$ and~$g$ to satisfy the hypotheses required by the measure stability theorem.

\begin{remark}
In the spirit of Theorem~\ref{thm:poly-hv-tame}, one may treat the case where $X$ is a locally compact polyhedron and $f$ is proper. We leave it as an exercise for the sufficiently persistent reader to carry this out and locate the possible singularities of the measures.
\end{remark}

%--------------------------------------------------
\section*{Acknowledgements}

Various people helped improve this paper. We thank William Crawley-Boevey for some invaluable correspondence regarding the decomposition of persistence modules. Michael Lesnick has been helpful in many ways: making available early versions of~\cite{Lesnick_2011}, drawing our attention to the paper of Webb~\cite{Webb_1985}, and correcting our original statement of Theorem~\ref{thm:gabriel+}.
Last but not least, we deeply appreciate the intellectual influence of David Cohen-Steiner and Leo Guibas, as co-authors of~\cite{Chazal_CS_G_G_O_2008}, and Gunnar Carlsson, as co-author of~\cite{Carlsson_deSilva_2010}.

The authors have been supported by the following research grants.

\begin{vlist}
\item
The Digiteo project C3TTA and the Digiteo Chair (held by the second author) which made this collaboration possible.

\item
European project CG-Learning EC: contract~255827.

\item
ANR GIGA: contract ANR-09-BLAN-0331-01.

\item
DARPA  project Sensor Topology and Minimal Planning (SToMP): HR0011-07-1-0002.
\end{vlist}

We thank our home institutions, for their ongoing support: INRIA Saclay -- Ile-de-France (FC, MG, SO) and Pomona College (VdS).

%-------------------------------------------------------
\bibliographystyle{plain}
\bibliography{pM-final}
%\bibliography{../../BibTeX/vin}

%------------------------------------------------------------------
\end{document}